\documentclass[11pt]{elsarticle}
\usepackage{amsmath}
\usepackage[margin=3.5cm]{geometry} 
\usepackage{amssymb}
\usepackage{amsthm}
\usepackage{bm}
\usepackage{algorithm} 
\usepackage[noend]{algpseudocode}
\usepackage{tikz-cd}
\usepackage{mathtools}
\usepackage{hyperref}
\usepackage{subcaption}
\usepackage{todonotes}
\theoremstyle{plain}
\newtheorem{theorem}{Theorem}[section]
\newtheorem{lemma}[theorem]{Lemma}
\newtheorem{proposition}[theorem]{Proposition}
\newtheorem{corollary}[theorem]{Corollary}
\theoremstyle{definition}
\newtheorem{definition}[theorem]{Definition}
\newtheorem{example}{Example}
\theoremstyle{remark}
\newtheorem{remark}{Remark}
\newcommand{\RR}{\mathbb{R}}
\newcommand{\ZZ}{\mathbb{Z}}
\newcommand{\NN}{\mathbb{N}}
\newcommand{\img}{\mathrm{Im}~}
\newcommand{\mcB}{\mathcal{B}}
\newcommand{\mcC}{\mathcal{C}}
\newcommand{\mcM}{\mathcal{M}}

\newcommand{\mcS}{\mathcal{S}}

\DeclareMathOperator{\id}{id}
\DeclareMathOperator{\spanset}{span}
\DeclareMathOperator{\ambRing}{\mathcal{R}}
\def\idealComplex{\mathcal{I}^{r}}
\def\idealComplexOne{\mathcal{I}^{1}}
\def\splineRing{\mathcal{F}^{r}}
\newcommand{\totalSpace}{\mathcal{T}}
\newcommand{\subSpace}{\mathcal{J}^{r}}
\newcommand{\subSpaceOne}{\mathcal{J}^{1}}
\newcommand{\quotientSpace}{\mathcal{Q}^{r}}
\newcommand{\quotientSpaceOne}{\mathcal{Q}^{1}}
\newcommand{\GrDomain}{\Delta,\Phi}

\newcommand{\glueb}{\mathfrak{b}}
\newcommand{\gluea}{\mathfrak{a}}
\newcommand{\cellulation}{\text{\textsc{cel}}}\usepackage[shortlabels]{enumitem}
\hypersetup{
 colorlinks=true, 
 linkcolor=blue, 
 citecolor=blue, 
 filecolor=blue, 
 urlcolor=blue 
}

\makeatletter
\def\ps@pprintTitle{%
 \let\@oddhead\@empty
 \let\@evenhead\@empty
 \def\@oddfoot{}%
 \let\@evenfoot\@oddfoot
}
\makeatother
\usepackage{cases}
\begin{document}
 \begin{frontmatter}
 \title{An algebraic framework for geometrically continuous splines}
 \author[inria]{Angelos Mantzaflaris}\ead{angelos.mantzaflaris@inria.fr}
 \author[inria]{Bernard Mourrain}\ead{bernard.mourrain@inria.fr}
 \author[swansea]{Nelly Villamizar} \ead{n.y.villamizar@swansea.ac.uk}
 \author[swansea]{Beihui Yuan} \ead{beihui.yuan@swansea.ac.uk}
 \address[inria]{Inria at Universit\'e C\^ote d'Azur, Sophia Antipolis, France}
 \address[swansea]{Department of Mathematics, Swansea University, Swansea, United Kingdom}
 \begin{abstract}
 Geometrically continuous splines are piecewise polynomial functions defined on a collection of patches which are stitched together through transition maps.
 They are called $G^{r}$-splines if, after composition with the transition maps, they are continuously differentiable functions to order $r$ on each pair of patches with stitched boundaries.
 This type of splines has been used to represent smooth shapes with complex topology for which (parametric) spline functions on fixed partitions are not sufficient. 
 In this article, we develop new algebraic tools to analyze $G^r$-spline spaces. 
 We define $G^{r}$-domains and transition maps using an algebraic approach, and establish an algebraic criterion to determine whether a piecewise function is $G^r$-continuous on the given domain. 
 In the proposed framework, we construct a chain complex whose top homology is isomorphic to the $G^{r}$-spline space. 
 This complex generalizes Billera-Schenck-Stillman homological complex used to study parametric splines.
 Additionally, we show how previous constructions of $G^r$-splines fit into this new algebraic framework, and present an algorithm to construct a bases for $G^r$-spline spaces. 
 We illustrate how our algebraic approach works with concrete examples, and prove a dimension formula for the $G^r$-spline space in terms of invariants to the chain complex. 
 In some special cases, explicit dimension formulas in terms of the degree of splines are also given.
 \end{abstract}
 \end{frontmatter}
 \section{Introduction}\label{sec:intro}
 Spline functions are mathematical representations built upon simpler pieces (usually defined by low-degree polynomials)
 which are glued together forming a smooth curve, surface or volume. 
 Splines constitute an appealing tool for shape representation not only for the simplicity of their construction \cite{farin2002handbook,piegl1996nurbs}, but because they are a fundamental component of the approximation of solutions of partial differential equations by the finite element method \cite{strang1973analysis}.
 They play a central role in geometric modeling \cite{farin2002handbook2}, in approximation theory \cite{LaiSchumaker}, and in novel fields such as isogeometric analysis \cite{Isogeometric,smoothiga2021}.
 
 Splines have been traditionally studied within the realm of numerical analysis and computational mathematics \cite{deBoor2001practical,LaiSchumaker}.
 They are defined as piecewise polynomial functions defined on a partition of a real domain which are continuously differentiable up to some order $r\geqslant 0$ on the whole domain. 
 These splines are called $C^r$-splines, or \emph{parametric} $C^r$-continuous splines. 
 If the polynomial degree is bounded then the set of $C^r$-splines on a given domain is a finite dimensional real vector space. 
 In the areas where splines are applied, it is important to be able to construct a basis, often with prescribed properties, for the space of $C^r$-splines on a given partitioned domain and fixed maximal polynomial degree. 
 A more basic task which aids in the construction of a basis is simply to compute the dimension of the spline space. 
 A formula for the dimension of $C^1$-spline spaces defined on triangulations of planar domains was first proposed by Strang~\cite{Strang}
 and proved for polynomial degree $d\geqslant 2$ for \textit{generic} vertex positions by Billera in~\cite{billera1988homology}.
 The seminal work in \cite{billera1988homology} unveiled fascinating connections between the study of splines and homological algebra techniques,
 putting spline theory at the interface between commutative algebra, geometric modelling, and numerical analysis. 
 The problem of computing the dimension of spline spaces on planar triangulations has received considerable attention using a wide variety of techniques. 
 Early works using Bernstein-B\'ezier methods from approximation theory include \cite{AS4r,AS3r,H91,SchumakerU}, see also \cite{LaiSchumaker} and the references therein.
 Some extensions of the algebraic approach introduced in \cite{billera1988homology} to compute the dimension of bivariate spline spaces include \cite{DimSeries,mourrain2013,MinReg,LCoho}. 
 One important feature of $C^r$-splines on triangulations is that the formula of the dimension of the space of splines of degree less than or equal to $d$ for $d\geqslant 3r+2$ is a lower bound on the dimension of the space for any degree $d\geqslant 0$ \cite{SchumakerLower}. 
 If $r+1\leqslant d\leqslant 3r+1$, to find an explicit formula for the dimension of bivariate spline spaces remains an open problem. 
 In this direction, some remarkable results have been proved using tools from rigidity theory \cite{WhiteleyComb}, and homological and commutative algebra \cite{CohVan,yuan2020,stefan,YS19}. 
 The literature on computing the dimension of trivariate splines on tetrahedral partitions is much less conclusive. 
 The dimension has been computed if $r=0$ (see~\cite{LocSup} or~\cite{Alg}), and also if $r=1$, $d\geqslant 8$, and the partition is generic ~\cite{ASWTet}. 
 For $r>1$, bounds on the dimension of trivariate spline spaces have been proved in~\cite{Tri, DV20b,diPasquale2021,Lau,mourrainV2014}. 
 
 The main contribution in this paper is to develop an algebraic framework to study spline functions which are defined on a collection of patches instead of fixed partitions of real domains. 
 The collection of patches are \emph{stitched} together through transition maps, and the continuity conditions are imposed after composition of the piecewise functions with the transition maps. 
 The splines defined on this collection of patches are called \emph{geometric continuous splines} or \emph{$G$-splines}.
 If $r\geqslant 0$ is an integer, the set of $G$-splines which (after the composition with the transition maps) are continuously differentiable functions to order $r$ on each pair of patches with stitched boundaries, are called \emph{$G^r$-continuous}, or simply \emph{$G^r$-splines}. 
 
 \emph{Geometric continuity} is essential for Computer-Aided
 Design (CAD), where typically a sophisticated design is obtained as a
 smooth surface described by many patches joint together. 
 Indeed, the concept
 originated in the automobile and aviation industry, see e.g. the early report~\cite{sabin1968conditions}.
 Parametric continuity (of first order) is instead the condition
 that the parametrization of the patches matches with continuity
 of the first order derivatives in parametric
 coordinates (corresponding to identity transition maps in our
 setting).
 Since parametric continuity is not a necessary nor a
 sufficient condition for the smoothness of the represented
 surface, the study of geometrically continuous representations
 emerged as a tool for flexible representations of CAD models. 
 Geometrically, $G^{1}$-splines are defined by the property that 
 the tangent planes of each patch match at that interface.
 These conditions involve the derivatives of the parametric coordinates
 composed with the transition map assigned to each interface.
 
 The concept of geometric continuity builds on the theory on differential manifolds.
 The first general definition of splines based
 on transition maps or reparametrizations was presented in \cite{derose1985geometric}, where it was used for building smooth surfaces from a collection of 2-dimensional patches. 
 Subsequently, the construction of geometrically continuous splines received
 considerable attention using a wide variety of techniques, here we list only a few representative references.
 Some early works include \cite{gregory1989,hahn1989geometric}, see also the survey \cite{peters2002_handbook} and the references therein.
 The construction of $G^{1}$- and $G^2$-spline surfaces has received considerable attention; see for instance
 \cite{Bonneau2014FlexibleGI,loop_smooth_1994,mourrain2016dimension, Peters:1994,peters_complexity_2010,prautzsch_freeform_1997,reif_biquadratic_1995}.
 In particular, $G^1$-splines have been used for surface fitting and surface reconstruction in geometric modeling; see 
 \cite{lin_adaptive_2007,shi_practical_2004}. 
 Manifold-based constructions of $G^r$-splines has been studied in \cite{vecchia_2008,grimm_1995,gu_manifold_2008,gu_manifold_2006,he_manifold_2006,tosun2011,ying_simple_2004}. 
 See \cite{beccari_rags:_2014}, where rational geometrically continuous spline are considered. 
 
 Research on geometrically continuous splines has been revived during the last decade due to its importance for solving partial differential equations by means of isogeometric analysis, see for instance~\cite{Sangalli_2016,smoothiga2021}, and the references therein.
 In particular, smooth splines are used for the discretization of partial differential equations and yield outstanding benefits compared to piecewise $C^0$ polynomial approximations: superior approximation properties and numerical stability, as well as the possibility to directly
 discretize high-order differential equations, such as the ones
 arising in thin-shell theory, in fracture models, in phase-field
 based multiphase flows, and in geometric flows on surfaces~\cite{smoothiga2021}.
 
 The computation of the dimension and construction of basis functions has been considered mainly for $G^1$-continuous splines. 
 The Bernstein--Bézier techniques for the construction of smooth surfaces on planar unstructured quad-meshes with linear transition maps are developed in~\cite{bercovier2017smooth}.
 For linear transition maps computed from the parametrizations of quad-meshes of planar domains, the space of ${G}^1$-smooth isogeometric spline functions
 is studied (dimension, basis) in~\cite{kapl2017,kapl2019isogeometric,KAPL201955}.
 An analogous definition of transition maps is used in~\cite{birner:hal-02271820} to construct $G^1$-splines on volumetric two-patch domains. 
 In \cite{marsala3656945}, a construction of $G^1$-surfaces approximating Catmull--Clark subdivision surfaces using quadratic transition maps is proposed.
 Algebraic properties of piecewise polynomials and transition maps are used in \cite{mourrain2016dimension} to construct $G^1$-splines for sufficiently large polynomial degree. 
 For this, in \cite{mourrain2016dimension} the authors introduce the concept of a \emph{topological surface}, which consists of a collection of polygons and an identification between pairs of edges of these polygons. 
 The splines on a topological surface are defined as piecewise polynomial functions which are differentiable after composition with {transition maps} associated to pairs of identified edges.
 The analysis of the dimension of $G^1$-splines in \cite{mourrain2016dimension} leads to the construction of basis functions when the topological surface is composed by quadrangular and triangular patches. 
 The construction of topological surfaces was extended to $G^1$-splines on quad meshes with 4-split macro patches in \cite{blidia2017}. 
 A second basis construction for $G^1$-splines on topological surfaces is introduced in \cite{blidia2020} where it is used for fitting of points clouds in isogeometric analysis for the solution of diffusion equations.
 
 In this article, we provide an
 algebraic framework for geometrically continuous $G^r$-splines for any order of continuity $r\geqslant 0$.
 This method extends the results on $G^1$-splines in \cite{mourrain2016dimension}, and generalizes the algebraic approach for (parametric) splines pioneered in \cite{billera1988homology}.
 
 We begin with a characterization of the domains over which the piecewise polynomials are defined. 
 We define them as a couple $(\Delta,\Phi)$ and call them \emph{$G^{r}$-domains}.
 They are composed by a cell complex $\Delta$ with a coordinate system associated to each maximal face, and a collection $\Phi$ of \emph{algebraic transition maps} between adjacent patches. Unlike transition maps in differential geometry, which are diffeomorphisms between open sets, an algebraic transition map is a homomorphism between quotients of polynomial rings. It is a ``stitch" rather than ``gluing". 
 We define $G^{r}$-splines as geometrically continuous piecewise polynomial functions on $G^r$-domains. 
 
 This algebraic setting for $G^r$-splines has several benefits. 
 On one hand, it allows the use of algebraic tools to study the dimension and to construct a basis for $G^r$-spline spaces. 
 Namely, we construct a chain complex of vector spaces whose top homology is isomorphic to the $G^{r}$-spline space. 
 This chain complex allows us to analyze the dimension of $G^{r}$-spline spaces, as we illustrate for some special cases in Section \ref{section:explicit_computations}. 
 More importantly, the algebraic framework enables us to present an algorithm (Algorithm \ref{algorithm:basis_computation_Gr}) to construct a basis for $G^{r}$-spline spaces over any given $G^{r}$-domain.
 Algorithm \ref{algorithm:basis_computation_Gr} is insensitive to the shape of patches in the sense that it works as well on other shapes as it does on triangles and rectangular patches, which makes the choice of domain very flexible.
 On the other hand, we do not lose the perspective gained from the framework in \cite{mourrain2016dimension}. 
 More precisely, we prove in Theorem \ref{prop:relation_notions_geomeric_continuous_functions} that if a $G^{1}$-domain is obtained from a topological surface, then the $G^{1}$-spline space in this paper and the space of differentiable functions on the topological space in \cite{mourrain2016dimension} are the same. 
 Hence, the dimension formulas obtained in \cite{mourrain2016dimension} still apply in the new context.
 As a result, $G^{r}$-splines can be defined and studied in s pure algebraic setting. 
 In particular, the criterion for a function to be $G^{r}$-continuous is algebraic. 
 We formulate concrete examples of algebraic transition maps. 
 As a byproduct of exploring the relation between the concept of geometrically continuous splines in different frameworks, we establish a way of obtaining transition maps from certain kinds of differential manifolds.
 
 This article is organized as follows. 
 In Section \ref{section:definition_construction}, we introduce the algebraic language to define $G^{r}$-splines.
 We define $G^r$-domains and $G^r$-continuity in Section \ref{section_polyhedral_frame}.
 We construct a chain complex whose top homology is isomorphic to the $G^{r}$-spline space in Section \ref{section:generalized_spline_complex}, and consider the bounded degree case of this chain complex in Section \ref{section_BSS_deg_d}. 
 We show that our construction yields the celebrated Billera-Schenck-Stillman complex as a special case in \ref{section_specializeToBSS}.
 The main result in Section \ref{section:the_real_engine} is Lemma \ref{lemma:a_basis_of_Fd}, which does not only 
 leads us to the construction a basis for $G^{r}$-spline spaces over any given $G^{r}$-domain in Algorithm \ref{algorithm:basis_computation_Gr}, but also serves as a general technique in computing generators for certain subspaces of a quotient ring. 
 In Section \ref{Section:Gsplines_by_differential_geometric_method}, we explain how to obtain a $G^{1}$-spline space from a geometrically continuous spline space as defined in \cite{mourrain2016dimension}. 
 The main result in this section is Theorem \ref{prop:relation_notions_geomeric_continuous_functions}, where we prove that these two notions of geometrically continuous splines can be identified. 
 Hence, results in \cite{mourrain2016dimension} and \cite{blidia2017} apply in the new algebraic framework. 
 We describe some interesting families of transition maps in Section \ref{sec:Examples}, including concrete examples of transition maps obtained from manifolds. 
 Section \ref{section:spline_complex_dim2} and \ref{section:explicit_computations} are devoted to the two dimensional case.
 In Section \ref{section:spline_complex_dim2}, we obtain a formula for the Euler characteristic of the chain complex constructed in Section \ref{section_BSS_deg_d} for $G^1$-splines of bounded polynomial degree.
 These computations are also an application of Lemma \ref{lemma:a_basis_of_Fd}.
 We analyze the homology terms of the chain complex in Section \ref{section:explicit_computations}. 
 The formula for the Euler characteristic together with explicit computations of the dimension of the homology terms yield a dimension formula for the dimension of $G^1$-spline spaces defined on two-adjacent patches, or on a star of vertex and continuity conditions given by symmetric gluing data.
 We conclude this section with Example \ref{eg:G1_cube}, where using Algorithm \ref{algorithm:basis_computation_Gr} and \texttt{Macaulay2} \cite{M2}, we compute the exact dimension of a $G^{1}$-spline space defined on the surface of a cube for certain polynomial degrees.
 Finally, we give some concluding remarks in Section \ref{sec:conclusions}. 
 \section{The space of $G^r$--splines}\label{section:definition_construction}
 In this section we introduce $G^r$-continuity and geometrically continuous $G^r$-spline functions. 
 We present a chain complex for $G^r$-splines, which can be seen as a generalization of the Billera-Schenck-Stillman chain complex for (parametric) splines introduced in \cite{billera1988homology} and refined in \cite{LCoho}. 
 First, we introduce the domain of definition of the $G^r$-spline functions which we call $G^r$-domains.
 \subsection{The $G^r$-domains and $G^r$-continuity conditions}\label{section_polyhedral_frame}
 Recall that a cell complex $\Delta$ is \emph{$n$-dimensional}, or \emph{$n$-complex}, if the largest dimension of any face in $\Delta$ is $n$ (see \cite{hatcher} for basics on cell complexes). 
 An $n$-dimensional cell complex is called \emph{pure} if all its maximal faces (with respect to inclusion) are of dimension $n$.
 If $i\geqslant 0$, a face of a cell complex $\Delta$ of dimension $i$ is called an $i$-face.
 We denote by $\Delta_i$ the collection of all $i$-faces of $\Delta$, and say that two $n$-faces $\sigma_{1}$ and $\sigma_{2}$ are \emph{adjacent} if they share an $(n-1)$-face.
 
 Throughout this paper we assume that $\Delta$ is a pure $n$-dimensional cell complex. To each $n$-face $\sigma$ of $\Delta$ we assign a system of coordinates denoted $X(\sigma)$. 
 More precisely, for each $\sigma\in\Delta_n$ we have a map $\psi_{\sigma}\colon \sigma\to X(\sigma)$ such that $\psi_{\sigma}$ is a homeomorphism between $\sigma$ and
 its image $\img{\psi_\sigma}$.
 Since $\dim{\sigma}=n$ then $X(\sigma)\cong \RR^{n}$.
 We refer to $X(\sigma)$ as the coordinates of $\sigma$, and write $\ambRing(\sigma)$ for the corresponding coordinate ring. 
 \begin{definition}
 \label{def:ideals}
 Let $\Delta$ be an $n$-dimensional cell complex with coordinates $X(\sigma)$ for each $n$-face $\sigma\in\Delta$. 
 If $S\subseteq X(\sigma)$, we define $I_{\sigma}(S)$ as the set of all polynomials in $\ambRing(\sigma)$ which vanish at $\psi_\sigma(p)$ for every $p\in S$. Namely, 
 \begin{equation*}
 I_{\sigma}(S)=\bigl\{f\in\ambRing(\sigma)\colon f(\psi_{\sigma}(p))=0~\mbox{for each point}~p\in S\bigr\}.
 \end{equation*}
 \end{definition}
 Notice that $I_{\sigma}(S)$ is an ideal of $\ambRing(\sigma)$ for every $\sigma\in\Delta_n$ and $S\subseteq X(\sigma)$.
 If $\alpha\subseteq\sigma$ is a $k$-face of an $n$-face $\sigma\in\Delta_n$, then $I_{\sigma}(\alpha)\subseteq \ambRing(\sigma)$ is the ideal of all polynomials in $\ambRing(\sigma)$ vanishing at $\psi_\sigma(\alpha)\subseteq X(\sigma)$. 
 In the case, $S=\{p\}$ consists only one point $p\in X(\sigma)$, we simply write $I_{\sigma}(p)$ for the ideal $I_{\sigma}(S)$.
 In particular, if the point $p\in \alpha\subseteq\sigma$ with $\alpha\in\Delta_k$, then $I_\sigma(\alpha)\subseteq I_\sigma(p)$.
 
 Note that we do not require any face $\alpha\subseteq \sigma$ to lie on an algebraic subset of $X(\sigma)$. 
 If $\alpha\subseteq \sigma$ is a $k$-face that lies on an algebraic subset of $X(\sigma)$, then $I_{\sigma}(\alpha)\neq \{0\}$, otherwise $I_{\sigma}(\alpha)=\{0\}$.
 \begin{definition}[Locally algebraic faces] \label{def:localg}
 If $\Delta$ is an $n$-dimensional cell complex with coordinates $X(\sigma)$ for each $n$-face $\sigma\in\Delta$, we say that an $(n-1)$-face $\tau\in\Delta_{n-1}$ is \emph{locally algebraic} if for each $n$-face $\sigma$ containing $\tau$ the ideal $I_{\sigma}(\tau)$ is
 generated by an irreducible polynomial $f\in \ambRing(\sigma)$.
 \end{definition}
 If all the $(n-1)$-faces of $\Delta$ are locally algebraic, 
 and $\alpha\subseteq \sigma\in\Delta_n$ is a $k$-face for $0\leqslant k\leqslant n-2$, then
 \begin{equation*}
 I_{\sigma}(\alpha)=\sum_{\alpha\subseteq\tau\subseteq\sigma}I_{\sigma}(\tau)\,.
 \end{equation*}
 \begin{remark}\label{remark:idealsLocalCoordinates}
 If $\tau=\sigma_1\cap\sigma_2$ is the common $(n-1)$-face of two $n$-faces $\sigma_1$ and $\sigma_2$ of a cell complex $\Delta$, 
 by Definition \ref{def:ideals}, there are two ideals $I_{\sigma_1}(\tau)\subseteq\mathcal{R}(\sigma_1)$ and $I_{\sigma_2}(\tau)\subseteq\mathcal{R}(\sigma_2) $ associated to the face $\tau$, one for each $n$-face containing $\tau$. 
 The subindexes of the ideals indicate the local coordinates at the corresponding $n$-face.
 (These ideals should not be confused with the ideal $\mathcal{I}^r(\tau)$ which we introduce later in Section \ref{section:generalized_spline_complex}. 
 The latter is a unique ideal associated to $\tau$, which depends on both coordinate rings $\mathcal{R}(\sigma_1)$ and $\mathcal{R}(\sigma_1)$ and on the ideals $I_{\sigma_1}(\tau)$ and $I_{\sigma_2}(\tau)$, see Definition \ref{def:ideals2}.)
 \end{remark}
 \begin{definition}[$G^r$-algebraic transition maps]\label{def:algebraic_transition_maps}
 Let $r\geqslant 0$ be and integer, and take two $n$-faces $\sigma_{1}$ and $\sigma_{2}$ of a cell complex $\Delta$ such that $\sigma_{1}\cap \sigma_{2}=\alpha\neq \emptyset$. 
 An $\RR$-algebra homomorphism
 \begin{equation}\label{eqn:def_algebraic_transition_maps}
 \phi_{12}\colon \ambRing(\sigma_{1})/I_{\sigma_{1}}(\alpha)^{r+1}\to \ambRing(\sigma_{2})/I_{\sigma_{2}}(\alpha)^{r+1}
 \end{equation} 
 is called a $G^{r}$-\emph{algebraic transition map} from $\sigma_{1}$ to $\sigma_{2}$, if 
 \begin{equation*}
 \phi_{12}^{-1}\bigl(I_{\sigma_{2}}(p)\bigr)=I_{\sigma_{1}}(p),~\mbox{for each point}~p\in\alpha.
 \end{equation*}
 \end{definition}
 
 \begin{definition}[Compatibility conditions]\label{def:compatibility_conditions_alg}
 Let $\Delta$ be a pure $n$-dimensional cell complex. A collection $\Phi$ of $G^{r}$-algebraic transition maps is \emph{compatible} on $\Delta$ if the following conditions hold:
 \begin{enumerate}[(C1)]
 \item\label{cond1} For each $n$-face $\sigma_i\in\Delta_n$, the only algebraic transition map $\phi_{ii}$ in $\Phi$ from $\sigma_{i}$ to itself is the identity map on $\ambRing(\sigma_{i})$. 
 \item \label{cond2} If the intersection of two $n$-faces $\sigma_1,\sigma_2\in\Delta_n$ is a non-empty face $\alpha=\sigma_{1}\cap\sigma_{2}$, then either $\phi_{12}$ or $\phi_{21}$ is in $\Phi$. If both of them are in $\Phi$, then $\phi_{12}=\phi_{21}^{-1}$. If $\phi_{12},\phi_{12}'\in\Phi$ are transition maps from $\sigma_{1}$ to $\sigma_{2}$, then $\phi_{12}=\phi_{12}'$. 
 \item\label{cond3} For any triple of $n$-faces $\sigma_{1},\sigma_{2},\sigma_{3}\in\Delta_n$, if their intersection is a non-empty face $\beta=\sigma_{1}\cap\sigma_{2}\cap\sigma_{3}$, and $\phi_{12},\phi_{23}\in\Phi$, then $\phi_{13}\in\Phi$ and each of the three maps has a natural restriction
 \begin{equation*}
 \overline{\phi_{ij}}\colon\ambRing(\sigma_{i})/I_{\sigma_{i}}(\beta)^{r+1}\to\ambRing(\sigma_{j})/I_{\sigma_{j}}(\beta)^{r+1},
 \end{equation*}
 where $(i,j)=(1,2),(2,3),(1,3)$ such that $\overline{\phi_{23}}\circ\overline{\phi_{12}}=\overline{\phi_{13}}$.
 \end{enumerate}
 \end{definition}
 We now introduce the domains where the geometrically continuous splines are defined. 
 Recall that an $n$-dimensional \emph{manifold with boundary} is a second-countable Hausdorff space in which every point has neighborhood homeomorphic either to an open subset of $\RR^{n}$ or to an open subset of the closed $n$-dimensional upper half-space \cite{lee2012smooth}.
 \begin{definition}[$G^{r}$-domains]\label{def:Grdomain}
 A \emph{$G^{r}$-domain} $(\GrDomain)$ is an $n$-dimensional cell complex $\Delta$ which is homeomorphic to an $n$-dimensional manifold with boundary, together with a compatible collection of $G^r$-algebraic transition maps $\Phi$. 
 \end{definition}
 In Section \ref{sec:Examples} we collect several examples of $G^r$-domains.
 \begin{remark}\label{rem:PhiBasis}Notice that in a $G^{r}$-domain $(\GrDomain)$ in Definition \ref{def:Grdomain}, the collection of transition maps $\Phi$ satisfies the compatibility conditions \ref{cond1}--\ref{cond3}, hence $\Phi$ is determined by the transition maps for each pair of adjacent $n$-faces of $\Delta$. 
 In the examples below we describe $\Phi$ by giving this subset of transition maps.
 \end{remark}
 
 Next, we construct the space of $G^r$-spline functions, which will be piecewise polynomial functions over $G^{r}$-domains that satisfy certain continuity conditions.
 
 We restrict the definition of $G^r$-splines to $G^r$-domains. Hereafter, they are the only domains we consider in this paper.
 
 \begin{definition}[$G^{r}$-continuity along a face]\label{def:Gr_join_algebraic}
 Let $\Delta$ an $n$-dimensional cell complex, and $\sigma_{1},\sigma_{2}\in \Delta_n$ two $n$-faces such that $\alpha=\sigma_1\cap\sigma_2$ is a non-empty face of $\Delta$. 
 Suppose $\phi_{12}$ is a $G^{r}$-algebraic transition map from $\sigma_1$ to $\sigma_2$ (as in Definition \ref{def:algebraic_transition_maps}).
 We say that $f\in\ambRing(\sigma_{1})$ and $g\in\ambRing(\sigma_{2})$ \emph{join with $G^{r}$-continuity along $\alpha$ with respect to $\phi_{12}$} if 
 \begin{equation*}
 \phi_{12}\bigl(\bar{f}\bigr)=\bar{g}, 
 \end{equation*}
 where $\bar{f}=f+I_{\sigma_{1}}(\alpha)^{r+1}$ and $\bar{g}=g+I_{\sigma_{2}}(\alpha)^{r+1}$ are the images of $f$ and $g$ in the quotient rings $\ambRing(\sigma_{1})/I_{\sigma_{1}}(\alpha)^{r+1}$ and $\ambRing(\sigma_{2})/I_{\sigma_{2}}(\alpha)^{r+1}$, respectively.
 \end{definition}
 
 \begin{definition}[$G^{r}$-splines]\label{def_Gr_spline}
 Let $(\GrDomain)$ be a $G^{r}$-domain for some integer $r\geqslant 0$. 
 The space $G^r(\GrDomain)$ of $G^{r}$-splines on $(\GrDomain)$ is the set of all piecewise polynomial functions $f=(f_\sigma)_{\sigma\in\Delta_n}\in \bigoplus_{\sigma\in\Delta_n}\mathcal{R}(\sigma)$ 
 such that $f_{\sigma_i}$ and $f_{\sigma_j}$ join along the $(n-1)$-face $\tau=\sigma_{i}\cap\sigma_{j}$ with $G^{r}$-continuity with respect to $\Phi$ for every pair of adjacent $n$-faces $\sigma_{i},\sigma_{j}\in \Delta_n$.
 \end{definition}
 \begin{remark}
 The notion of geometrically continuous splines was originally defined within differential geometry framework. The relation between this notion in these two frameworks is explained in Section \ref{Section:Gsplines_by_differential_geometric_method}.
 \end{remark}
 
 \subsection{Chain complex for $G^r$-splines}\label{section:generalized_spline_complex}
 In this section we assume that $\Delta$ is an $n$-dimensional polyhedral complex. We write $\mcC_{\bullet}(\Delta,\partial\Delta)=\mcC_{\bullet}(\Delta)/\mcC_{\bullet}(\partial\Delta)$ for the chain complex of $\Delta$ relative to its boundary $\partial\Delta$. 
 By an abuse of notation, we denote the $k$-faces of $\Delta$ relative to its boundary by $(\Delta,\partial\Delta)_{n}=\Delta_{n}$ and $(\Delta,\partial\Delta)_{k}=\Delta_{k}^{\circ}$, for $0\leqslant k\leqslant n-1$. 
 We call the faces in $\Delta_k^\circ$ the \emph{interior} $k$-faces of $\Delta$.
 Recall that to each $n$-face $\sigma_i\in\Delta_n$, we associate a system of coordinates $X(\sigma_i)\cong \RR^n$ and denote by $\ambRing(\sigma_i)$ the polynomial ring in the coordinates 
 $(u_{i1},\dots,u_{in})$ of $X(\sigma_i)$. 
 In the following we extend this construction and associate a ring to each face of the polyhedral complex $\Delta$.
 
 If $0\leqslant k\leqslant n-1$, and $\alpha\in\Delta_k$ is a $k$-face, we take $\ambRing(\alpha)$ as the tensor product 
 \begin{equation}\label{eq:ringalpha}
 \ambRing(\alpha)=\bigotimes_{\substack{\sigma\supseteq\alpha\\\sigma\in\Delta_{n}}}\ambRing(\sigma).
 \end{equation}
 For simplicity, we write $\ambRing(\sigma_{1},\dots,\sigma_{m})=\ambRing(\sigma_{1})\otimes\dots\otimes\ambRing(\sigma_{m})$, for $n$-faces $\sigma_i\in\Delta_n$.
 
 \begin{definition}\label{def:ideals2}
 If $\sigma_{1},\sigma_{2}\in\Delta_n$ are two adjacent $n$-faces such that $\sigma_1\cap\sigma_2=\tau\in\Delta_{n-1}$ and $\phi_{12}\in\Phi$ is a algebraic transition map from $\sigma_1$ to $\sigma_2$, we define $\idealComplex(\tau)\subseteq\ambRing(\tau)$ as the ideal 
 \begin{equation}\label{eq:idealtau}
 \idealComplex(\tau)=\sum_{i=1}^{n}\bigl\langle u_{1i}-\widetilde{\phi}_{12}(u_{1i}) \bigr\rangle+
 \left(I_{\sigma_{1}}(\tau)^{r+1}\cdot\ambRing(\tau)+I_{\sigma_{2}}(\tau)^{r+1}\cdot\ambRing(\tau)\right),
 \end{equation}
 where $\widetilde{\phi}_{12}$ is a lift of $\phi_{12}$ in \eqref{eqn:def_algebraic_transition_maps}. 
 If $\alpha\in\Delta_{k}$ for $0\leqslant k\leqslant n-2$, we define
 \begin{equation}\label{eq:idealalpha}
 \idealComplex(\alpha)=
 \sum_{\substack{\tau\supseteq\alpha\\\tau\in\Delta_{n-1}}}\idealComplex(\tau)\cdot\ambRing(\alpha).
 \end{equation}
 For $0\leqslant k\leqslant n-1$, and $\alpha\in\Delta_k$, we put
 \begin{equation}\label{eq:complexF}
 \splineRing(\alpha)=\ambRing(\alpha)/\idealComplex(\alpha).
 \end{equation}
 For $\sigma\in\Delta_n$, put $\splineRing(\sigma)=\ambRing(\sigma)$.
 \end{definition}
 \begin{remark}
 Notice that for an $(n-1)$-face $\tau=\sigma_1\cap\sigma_2\in\Delta_{n-1}$, the ideal $\idealComplex(\tau)$ in Definition \ref{def:ideals2} does not depend on the choice of the lift of the algebraic transition map $\phi_{12}$.
 \end{remark}
 For each pair $(\alpha,\beta)\in\Delta_{k}\times\Delta_{k-1}$ such that $\beta\subseteq\alpha$, we have the inclusion
 \begin{equation}\label{eq:alghomomorphishm}
 \rho^{\alpha}_{\beta}\colon \ambRing(\alpha)\to\ambRing(\beta)\,,
 \end{equation}
 for every $k=1, \dots, n$.
 Suppose the $k$-th boundary map of the chain complex $\mcC_{\bullet}(\Delta,\partial\Delta)$ is given by 
 \begin{equation}\label{eq:def_bdry_map}
 \partial_{k}=\sum_{(\alpha,\beta)\in(\Delta,\partial\Delta)_{k}\times(\Delta,\partial\Delta)_{k-1}}a_{\alpha,\beta}e_{\beta}\otimes e_{\alpha}^{*},
 \end{equation}
 where $\{e_{\alpha}\}$ and $\{e_{\beta}\}$ are generators of $\mcC_{k}(\Delta,\partial\Delta)$ and $\mcC_{k-1}(\Delta,\partial\Delta)$, respectively, and $a_{\alpha,\beta}\in\ZZ$. 
 We denote by $\delta_{k}$ the $\RR$-linear map
 \begin{equation}\label{eq:delta}
 \delta_{k}\colon \bigoplus_{\alpha\in(\Delta,\partial\Delta)_{k}}\ambRing(\alpha)\to\bigoplus_{\beta\in(\Delta,\partial\Delta)_{k-1}}\ambRing(\beta)
 \end{equation}
 defined as
 \begin{equation*}
 \delta_{k}=\sum_{(\alpha,\beta)\in(\Delta,\partial\Delta)_{k}\times(\Delta,\partial\Delta)_{k-1}}a_{\alpha,\beta}\rho^{\alpha}_{\beta} e_{\beta}\otimes e_{\alpha}^{*}.
 \end{equation*}
 We first check that $(\ambRing_{\bullet},\delta_{\bullet})$ is a chain complex.
 \begin{lemma}\label{lemma_chaincomplex_condition}
 If $\Delta$ is an $n$-dimensional polyhedral complex, and $\delta_k$ is the linear map in \eqref{eq:delta}, then
 \begin{equation*}
 \delta_{k+1}\circ\delta_{k+2}=0,
 \end{equation*}
 for $k=0,\dots,n-2$, and hence $(\ambRing_{\bullet},\delta_{\bullet})$ is a chain complex of $\RR$-vector spaces.
 \end{lemma}
 \begin{proof}
 Note that for faces $\gamma\in\Delta_{k}$ and $\alpha\in\Delta_{k+2}$, if there are two faces $\beta,\beta'\in\Delta_{k+1}$ such that $\gamma\subseteq\beta\subseteq\alpha$ and $\gamma\subseteq\beta'\subseteq\alpha$, then
 \begin{equation*}
 \rho^{\beta}_{\gamma}\circ\rho^{\alpha}_{\beta}=\rho^{\beta'}_{\gamma}\circ\rho^{\alpha}_{\beta'},
 \end{equation*}
 for the homomorphisms in \eqref{eq:alghomomorphishm}.
 Furthermore, for any $k=0,\dots,n-2$ we have
 \begin{equation*}
 \partial_{k+1}\circ\partial_{k+2}=0.
 \end{equation*}
 Equivalently, this means if $\partial_{k+2}=\sum a_{\alpha,\beta}e_{\alpha}\otimes e_{\beta}^{*}$ and $\partial_{k+1}=\sum b_{\beta,\lambda}e_{\beta}\otimes e_{\lambda}^{*}$, then for faces $\alpha$ and $\lambda$, such that $\gamma\subseteq\alpha$ we have
 \begin{equation*}
 \sum_{\alpha\subseteq\beta\subseteq\gamma}a_{\alpha,\beta}b_{\beta,\lambda}=0.
 \end{equation*}
 Therefore,
 \begin{equation*}
 \bigl(\delta_{k+1}\circ\delta_{k+2}\bigr)_{\alpha,\lambda}
 =
 \sum_{\alpha\subseteq\beta\subseteq\gamma}a_{\alpha,\beta}b_{\beta,\lambda}\rho^{\beta}_{\gamma}\circ\rho^{\alpha}_{\beta}
 =
 \Biggl(\sum_{\alpha\subseteq\beta\subseteq\gamma}a_{\alpha,\beta}b_{\beta,\lambda}\Biggr)\cdot\rho^{\beta'}_{\gamma}\circ\rho^{\alpha}_{\beta'}=0\,,
 \end{equation*}
 where $\beta'$ is a $(k+1)$-face between $\alpha$ and $\lambda$. 
 Since 
 \begin{equation*}
 \delta_{k+1}\circ\delta_{k+2}=
 \sum_{(\alpha,\lambda)\in(\Delta,\partial\Delta)_{k+2}\times(\Delta,\partial\Delta)_{k}}(\delta_{k+1}\circ\delta_{k+2})_{\alpha,\lambda} e_{\alpha}\otimes e_{\lambda}^{*},
 \end{equation*}
 then $\delta_{k+1}\circ\delta_{k+2}=0$.
 \end{proof}
 Note that the restriction of $\rho^{\alpha}_{\beta}$ in \eqref{eq:alghomomorphishm} to $\idealComplex(\alpha)$ gives us a map $\idealComplex(\alpha)\to\idealComplex(\beta)$. Therefore, $\delta_{k}$ restricts to $\oplus_{\alpha\in(\Delta,\partial\Delta)_{k}}\idealComplex(\alpha)\to\oplus_{\beta\in(\Delta,\partial\Delta)_{k-1}}\idealComplex(\beta)$ and it induces a natural map $\oplus_{\alpha\in(\Delta,\partial\Delta)_{k}}\splineRing(\alpha)\to\oplus_{\beta\in(\Delta,\partial\Delta)_{k-1}}\splineRing(\beta)$. 
 Thus, Lemma \ref{lemma_chaincomplex_condition} implies that both $(\idealComplex_{\bullet},\delta_{\bullet})$ and $(\splineRing_{\bullet},\delta_{\bullet})$ are chain complexes, and we have the following immediate consequence of our construction.
 \begin{theorem}\label{Prop_Gr_Hn} 
 The space $G^{r}(\GrDomain)$ of $G^{r}$-splines on a $G^r$-domain $(\GrDomain)$ is the top homology $H_{n}(\splineRing_{\bullet})$ of the chain complex $(\splineRing_{\bullet},\delta_{\bullet})$.
 \end{theorem}
 \begin{proof}
 By definition, a tuple of polynomials $(f_{\sigma})_{\sigma\in\Delta_{n}}\in G^{r}(\GrDomain)$ if 
 \begin{equation}\label{eqn:Gr_join_phiij}
 \phi_{ij}(\bar{f}_{\sigma_{i}})=\bar{f}_{\sigma_{j}},
 \end{equation}
 for each pair of adjacent $n$-faces $\sigma_{i},\sigma_{j}\in\Delta_n$. 
 Let $\tau=\sigma_{i}\cap\sigma_{j}\in\Delta_{n-1}$. 
 We only need to show that \eqref{eqn:Gr_join_phiij} holds if and only if
 \begin{equation*}
 f_{\sigma_{i}}-f_{\sigma_{j}}\in\idealComplex(\tau).
 \end{equation*}
 We know that 
 \begin{equation}\label{eqn:equivalent_conditions_by_graph}
 \widetilde{\phi}_{ij}(g_{i})=g_{j}~\Leftrightarrow~g_{i}-g_{j}\in\left\langle u_{ik}-\widetilde{\phi}_{ij}(u_{ik})\colon k=1,\dots,n\right\rangle,
 \end{equation}
 for any pair of polynomials $(g_{i},g_{j})\in\ambRing(\sigma_{i})\times\ambRing(\sigma_{j})$, where $\widetilde{\phi}_{ij}$ is a lifting of $\phi_{ij}$.
 For $l=i,j$, denote by $\bar{f}_{\sigma_{l}}$ be the quotient image of $f_{\sigma_{l}}$ in $\ambRing(\sigma_{l})/I_{\sigma_{l}}(\tau)^{r+1}$. If \eqref{eqn:Gr_join_phiij} holds, then there exists $\widetilde{f}_{i}$ and $\widetilde{f}_{j}$ which are liftings of $\bar{f}_{\sigma_{i}}$ and $\bar{f}_{\sigma_{j}}$, respectively, such that $\widetilde{\phi}_{ij}(\widetilde{f}_{i})=\widetilde{f}_{j}$. Hence, \eqref{eqn:equivalent_conditions_by_graph} implies $\widetilde{f}_{i}-\widetilde{f}_{j}\in \left\langle u_{ik}-\widetilde{\phi}_{ij}(u_{ik})\colon k=1,\dots,n\right\rangle$. 
 Since $\widetilde{f}_{i}$ is a lifting of $\bar{f}_{\sigma_{i}}$, then $\widetilde{f}_{i}-f_{\sigma_{i}}\in I_{\sigma_{i}}(\tau)^{r+1}$. The same applies for $\widetilde{f}_{j}$ and $f_{\sigma_{j}}$, hence $\widetilde{f}_{j}-f_{\sigma_{j}}\in I_{\sigma_{j}}(\tau)^{r+1}$. 
 Putting these together, we have $f_{\sigma_{i}}-f_{\sigma_{j}}\in\idealComplex(\tau)$.
 
 Conversely, if $f_{\sigma_{i}}-f_{\sigma_{j}}\in\idealComplex(\tau)$, then there exist $\widetilde{f}_{i}$, and $\widetilde{f}_{j}$ such that 
 $f_{\sigma_{i}}-\widetilde{f}_{i}\in I_{\sigma_{i}}(\tau)^{r+1}$, $f_{\sigma_{j}}-\widetilde{f}_{j}\in I_{\sigma_{j}}(\tau)^{r+1}$, and $\widetilde{f}_{i}-\widetilde{f}_{j}\in \left\langle u_{ik}-\widetilde{\phi}_{ij}(u_{ik})\colon k=1,\dots,n\right\rangle$. Again, by \eqref{eqn:equivalent_conditions_by_graph}, we have $\widetilde{\phi}_{ij}(\widetilde{f}_{i})=\widetilde{f}_{j}$, which implies \eqref{eqn:Gr_join_phiij}.
 \end{proof}
 The chain complex $(\splineRing_{\bullet},\delta_{\bullet})$ specializes to the Billera-Schenck-Stillman spline chain complex when all transition maps are identities. 
 This is a technical result which for completeness we have included in \ref{section_specializeToBSS}. 
 \subsection{Chain complex for $G^r$-splines of bounded polynomial degree}\label{section_BSS_deg_d}
 In the following we consider the subspace of $G^r$-splines over a given $G^r$-domain $(\GrDomain)$ which are of degree less than or equal to a fixed $d\geqslant 0$. 
 For each $d$, we associate a chain complex to the subspace of $G^{r}(\GrDomain)$ of splines of degree at most $d$ and relate it to the chain complex $(\splineRing_{\bullet},\delta_{\bullet})$ defined in Section \ref{section:generalized_spline_complex}.
 \begin{definition}
 Let $d\geqslant 0$ be an integer. 
 The set of $G^r$-spline of degree at most $d$ over a $G^r$-domain $(\GrDomain)$, denoted 
 $G^{r}_{d}(\GrDomain)$, is the subset of splines $f\in G^{r}(\GrDomain)$ whose restriction to each $n$-face is a polynomial of degree less than or equal to $d$. Namely,
 \begin{equation*}
 G^{r}_{d}(\GrDomain)=\bigl\{f\in G^{r}(\GrDomain)\colon\deg(f_{\sigma})\leqslant d~\mbox{for every}~\sigma\in\Delta_{n}\bigr\},
 \end{equation*}
 where $f=(f_\sigma)_{\sigma\in\Delta_n}$, and $f|_\sigma=f_\sigma$.
 \end{definition}
 It is clear that $G^{r}_{d}(\GrDomain)$ is a finite dimensional $\RR$-vector space. 
 To find the dimension of this vector space, we construct a chain complex of finite dimensional vector spaces as follows.
 
 If $\sigma\in\Delta_{n}$, denote by $\ambRing_{d}(\sigma)=\ambRing(\sigma)_{\leqslant d}$ the vector space of polynomials of degree at most $d$ in the coordinates $X(\sigma)\cong \RR^n$. 
 For a $k$-face $\alpha\in\Delta_k$, for $0\leqslant k\leqslant n$, let 
 \begin{equation}\label{eq:ringT}
 \totalSpace_{d}(\alpha)= \Biggl\{ \sum_{\alpha\subseteq\sigma\in\Delta_{n}} f_\sigma \colon f_\sigma\in \ambRing_d(\sigma)\Biggr\}
 \end{equation} 
 be the image of the inclusion
 \begin{equation*}
 \bigoplus_{i=1}^{s}\ambRing_{d}(\sigma_{i})
 \xhookrightarrow{\phantom{in}}
 \ambRing(\sigma_{1},\dots,\sigma_{s})=\ambRing(\alpha),
 \end{equation*}
 where $\sigma_{1},\dots,\sigma_{s}\in\Delta_n$ are the $n$-faces in $\Delta$ containing $\alpha$. 
 Notice that the set $\totalSpace_d(\alpha)$ is a subspace of the tensor product polynomial ring $\ambRing(\alpha)$ defined in \eqref{eq:ringalpha}. 
 The polynomials in $\totalSpace_d(\alpha)$ do not contain mixed terms i.e., each term is a monomial in one of the rings $\ambRing(\sigma)$ for some $n$-face $\sigma\supseteq \alpha$. 
 If $\sigma$ is an $n$-face, then $\totalSpace_{d}(\sigma)=\ambRing_{d}(\sigma)$. 
 
 Additionally, for each $\alpha\in\Delta_k$, for $0\leqslant k\leqslant n-1$, we define
 \begin{equation}\label{eq:idealJ}
 \subSpace_{d}(\alpha)=\totalSpace_{d}(\alpha)\cap\idealComplex(\alpha),
 \end{equation}
 and 
 \begin{equation}\label{eq:quotientQ}
 \quotientSpace_{d}(\alpha)=\totalSpace_{d}(\alpha)/\subSpace_{d}(\alpha),
 \end{equation}
 where $\idealComplex(\alpha)$ is the ideal defined in \eqref{eq:idealtau} and \eqref{eq:idealalpha} for $k=n-1$ and $0\leqslant k\leqslant n-2$, respectively. If $\sigma\in\Delta_n$, we take $\quotientSpace_d(\sigma)=\totalSpace_d(\sigma)$.
 
 Since the differential map $\delta$ in the chain complex $\bigl(\ambRing_{\bullet},\delta\bigr)$ restricts to
 \begin{equation*}
 \delta_{k}\colon\bigoplus_{\alpha\in\Delta_{k}}\totalSpace_{d}(\alpha)\to\bigoplus_{\beta\in\Delta_{k-1}}\totalSpace_{d}(\beta),
 \end{equation*}
 then $(\totalSpace_{d,\bullet},\delta)$ is a chain complex, and hence so are $(\subSpace_{d,\bullet},\delta)$ and $(\quotientSpace_{d,\bullet},\delta)$. 
 
 As a corollary of Theorem~\ref{Prop_Gr_Hn}, we have the following proposition. 
 \begin{proposition}\label{prop:Grd_is_top_homology_of_Qd}
 Let $\bigl(\GrDomain\bigr)$ be a $G^r$-domain and $d\geqslant 0$ an integer, then 
 \begin{equation}\label{eq:Grd_equals_HnQ}
 G^{r}_{d}(\GrDomain)= H_{n}\bigl(\quotientSpace_{{d},\bullet}\bigr),
 \end{equation}
 where $H_{n}\bigl(\quotientSpace_{{d},\bullet}\bigr)$ is the top homology of the chain complex $\bigl(\quotientSpace_{d,\bullet},\delta\bigr)$.
 \end{proposition}
 \begin{proof}
 By definition, $G^{r}_{d}(\GrDomain)=G^{r}(\GrDomain)\cap \oplus_{\sigma\in\Delta_{n}}\totalSpace_{d}(\sigma)$. By Theorem \ref{Prop_Gr_Hn}, $G^{r}(\GrDomain)$ is the kernel of $\delta_{n}:\oplus_{\sigma\in\Delta_{n}}\ambRing(\sigma)\to\oplus_{\tau\in\Delta_{n-1}^{\circ}}\splineRing(\tau)$. Hence, $G^{r}_{d}(\GrDomain)$ is the kernel of $\delta_{n}$ restricted to $\oplus_{\sigma\in\Delta_{n}}\totalSpace_{d}(\sigma)$, which is exactly $H_{n}(\quotientSpace_{d,\bullet})$.
 \end{proof}
 With Proposition \ref{prop:Grd_is_top_homology_of_Qd}, we estimate $\dim G^{r}_{d}(\GrDomain)$ by comparing it to the \emph{Euler characteristic} $\chi(\quotientSpace_{d,\bullet})$ of $\quotientSpace_{d,\bullet}$, which is defined as
 \begin{equation*}
 \chi(\quotientSpace_{d,\bullet})=\sum_{k=0}^{n}(-1)^{k}\cdot\sum_{\alpha\in(\Delta,\partial\Delta)_{k}}\dim \quotientSpace_{d}(\alpha).
 \end{equation*}
 The dimension of $G^{r}$-spline spaces and Euler characteristic $\chi(\quotientSpace_{d,\bullet})$ are related by the following corollary:
 \begin{corollary}\label{cor:dim_Grd_equals_chi_plus_H}
 For any integer $d\geqslant 0$, and any $G^{r}$-domain $(\GrDomain)$, we have 
 \begin{equation}\label{eq:dim2cellcase_general_form}
 \dim G^{r}_{d}(\GrDomain)=(-1)^{n}\chi(\quotientSpace_{d,\bullet})+\sum_{k=0}^{n-1}(-1)^{n+k+1}\dim H_{k}(\quotientSpace_{d,\bullet}).
 \end{equation}
 \end{corollary}
 \begin{proof}
 For any chain complex $C_{\bullet}$ of finite dimensional vector spaces such that $C_{k}=0$ if $k>n$ or $k<0$, we have
 \begin{equation*}
 \chi(C_{\bullet})=\sum_{k=0}^{n}(-1)^{k}\dim C_{k}=\sum_{k=0}^{n}(-1)^{k}\dim H_{k}(C_{\bullet}).
 \end{equation*}
 In particular, this holds for $C_{\bullet}=\quotientSpace_{d,\bullet}$. Hence, \eqref{eq:Grd_equals_HnQ} implies \eqref{eq:dim2cellcase_general_form}.
 \end{proof}

 Specializing to $n=2$ cases, the construction and results in this section can easily be adapted to the bidegree case.
 More precisely, if $\Delta$ is a two dimensional complex, and $\sigma\in\Delta_2$ then a polynomial $f(u,v)\in\ambRing(\sigma)$ is of bidegree $(d,d)$ if $\deg f(u,1)$, and $\deg f(1,v)$ are $\leqslant d$. 
 We write $\ambRing_{(d,d)}(\sigma)$ for the set of polynomials in $\ambRing(\sigma)$ of bidegree $(d,d)$. 
 The space of bidegree $(d,d)$ splines is the set
 \begin{equation*}
 G^{r}_{(d,d)}(\GrDomain)=\bigl\{f\in G^{r}(\GrDomain)\colon\deg(f_{\sigma})\in\ambRing_{(d,d)}(\sigma)~\mbox{for every}~\sigma\in\Delta_{2}\bigr\},
 \end{equation*}
 where $f=(f_\sigma)_{\sigma\in\Delta_2}$, and $f|_\sigma=f_\sigma$.
 
 Similarly, we can adapt \eqref{eq:ringT}, \eqref{eq:idealJ}, and \eqref{eq:quotientQ} and define $\totalSpace_{(d,d)}(\alpha)$, $\subSpace_{(d,d)}(\alpha)$, and $\quotientSpace_{(d,d)}(\alpha)$, respectively, for $\alpha\in\Delta_k$ for any $0\leqslant k\leqslant n$.
 We refer to $G^{r}_{d}(\GrDomain)$, $\totalSpace_{d,\bullet}$, $\subSpace_{d,\bullet}$, $\quotientSpace_{d,\bullet}$ as \emph{total degree spaces}, and $G^{r}_{(d,d)}(\GrDomain)$, $\totalSpace_{(d,d),\bullet}$, $\subSpace_{(d,d),\bullet}$, $\quotientSpace_{(d,d),\bullet}$ as \emph{bidegree spaces (or bigrading)}.
 
 We study the terms of the splines dimension formula \eqref{eq:dim2cellcase_general_form} in Sections \ref{section:spline_complex_dim2} and \ref{section:explicit_computations}, where we restrict the study to $G^r$-splines on 2-dimensional cell complexes $\Delta$. 
 The results will be presented for both $\dim G^{r}_{d}(\GrDomain)$ and $\dim G^{r}_{(d,d)}(\GrDomain) $.

 \subsection{Constructing a basis for geometrically continuous spline spaces}\label{section:the_real_engine}
 
 This section is devoted to develop a general method to generate certain subspace of a quotient ring. In particular, it gives an algorithm to construct a basis for a $G^{r}$-spline space. 
 The main result in this section is Lemma \ref{lemma:a_basis_of_Fd}, which describes the method on generating certain subspaces of a quotient ring.
The algorithm for the $G^{r}$-spline basis will be an immediate consequence of this result.
 Later in Section \ref{section:spline_complex_dim2}, Lemma \ref{lemma:a_basis_of_Fd} will also play an important role for computing the dimension of the terms in \eqref{eq:dim2cellcase_general_form}. 
 Lemma \ref{lemma:a_basis_of_Fd} does not depend on our setting so we state it in general for a finitely generated real polynomial ring. 
 We introduce the following notation that will be helpful to define the subspaces of the quotient ring we want to describe.
 \begin{definition}\label{def:remainder_of_f} 
 Let $f\in S$ be polynomial in a finitely generated polynomial ring $S$ over $\RR$, and let $I\subseteq S$ be an ideal.
 We define the \emph{remainder of $f$ with respect to $I$} to be the remainder of $f$ by performing the division algorithm to $f$ with respect to the reduced Gr\"{o}bner basis of $I$ for a fixed monomial order $\succ$ on $S$. 
 We denote this by $\mathfrak{r}(f,I)$, or simply $\mathfrak{r}(f)$ if $I$ is clear from the context.
 \end{definition}
 \begin{lemma}\label{lemma:a_basis_of_Fd}
 Let $S$ be a finitely generated polynomial ring over $\RR$, with a fixed monomial order $\succ$, and let $W\subseteq S$ is a finite dimensional $\RR$-vector space with basis $\mcB$. 
 Assume $I\subseteq S$ is an ideal.
 Then, the linear span $\spanset\{\mathfrak{r}(b)\colon b\in\mcB\}\cong W/(W\cap I)$, where $\mathfrak{r}(b)$ is the remainder of $b\in \mcB$ with respect to $I$. 
 \end{lemma}
 \begin{proof}
 Let $V=\spanset\{\mathfrak{r}(b)\colon b\in\mcB\}\subseteq S$.
 Then we can define a map $\mathfrak{r}\colon W\to V$ given by $f\mapsto \mathfrak{r}(f)$. 
 If $W\to W/(W\cap I)$ is the quotient map $f\mapsto\overline{f}$, then there exists a map $V\to W/(W\cap I)$ for which the following diagram commutes
 \begin{equation*}
 \begin{tikzcd}
 W \arrow[rd] \arrow[d, "\mathfrak{r}"]
 &\\
 V \arrow[r, "h", dashed]& W/(W\cap I)\,.
 \end{tikzcd}
 \end{equation*}
 Indeed, for any $g\in V$, we may take $g'\in W$ such that $\mathfrak{r}(g')=g$ and define $h(g)=\overline{g'}$.
 This map is well-defined, because if $\mathfrak{r}(g'_{1})=\mathfrak{r}(g'_{2})=g$ for $g_1',g_2'\in W$, then $g'_{1}-g'_{2}\in I$, which implies that $\overline{g'_{1}}=\overline{g'_{2}}$.
 It is easy to verify this map is linear, injective and surjective.
 Hence, $h$ is an isomorphism of vector spaces.
 \end{proof}
 
 By Proposition \ref{prop:Grd_is_top_homology_of_Qd} we know that $G^{r}_{d}(\GrDomain)$ is the kernel of the map $\delta_{n}\colon\quotientSpace_{d,n}\to\quotientSpace_{d,n-1}$, where 
 $\quotientSpace_{d,k}$ is the $k$-th terms in the chain complex $(\quotientSpace_{d,\bullet},\delta_{\bullet})$.
 Recall that by definition we take $\quotientSpace_{d,n}=\totalSpace_{d,n}$, see \eqref{eq:ringT} and \eqref{eq:quotientQ}.
 By construction, for each $\tau\in\Delta_{n-1}^{\circ}$, the set $\totalSpace_{d}(\tau)$ is a linear subspace of the ring $\ambRing_d(\tau)$, and $I=\idealComplex(\tau)$ is an ideal of $\ambRing_d(\tau)$. Hence, Lemma \ref{lemma:a_basis_of_Fd} implies $\mathfrak{r}\bigl(\totalSpace_{d}(\tau)\bigr)\cong\quotientSpace_{d}(\tau)$.
 
 This allows to see $G^{r}_{d}(\GrDomain)$ as the kernel of the map $\mathfrak{r}\circ\delta_{n}\colon\totalSpace_{d,n}\to \mathfrak{r}(\totalSpace_{d,n-1})$, where (by an abuse of notation) we write $\delta_{n}\colon\totalSpace_{d,n}\to\totalSpace_{d,n-1}$ for the restriction of the differential map in $(\ambRing_{\bullet},\delta_{\bullet})$ to $\totalSpace_{d,n}$.
 
 We start the algorithm to compute a basis of $G^{r}_{d}(\GrDomain)$ by fixing a maximal polynomial degree $d$. 
 Taking $k=n$ in \eqref{eq:def_bdry_map}, the boundary map $\partial_{n}$ 
 can be written as a rectangular matrix $\left(a_{\sigma,\tau}\right)_{\tau\in\Delta_{n-1}^\circ\; \sigma\in\Delta_n}$, whose rows are labeled by the interior $(n-1)$-faces 
 and the columns by the maximal $n$-faces of $\Delta$.
 We compute the reduced Gr\"obner basis of the ideal $\idealComplex(\tau)$, for each $\tau\in\Delta_{n-1}^{\circ}$. 
 Then, for every $n$-face $\sigma\supseteq\tau$, if $\mcB(\sigma)$ is the monomial basis of $\totalSpace_{d}(\sigma)$, we use the division algorithm to find the remainder of each polynomial in $\mcB(\sigma)$ with respect to the Gr\"{o}bner basis of $\idealComplex(\tau)$.
 By Lemma \ref{lemma:a_basis_of_Fd}, we know that the collection of those remainders generate $\quotientSpace_{d,n-1}$.
 Next we construct a matrix $A$ which will encode the map $\delta_{n}\colon\totalSpace_{d,n}\to\quotientSpace_{d,n-1}$, 
 Explicitly, if
 \begin{equation*}
 \mathfrak{r}\circ\delta_{n}
 =
 \sum_{\substack{\sigma\in\Delta_{n},\\ m \in \mcB(\sigma)
 }}\mathopen{\raisebox{-1.8ex}{$\Biggl($}}\,\sum_{\substack{\tau\in\Delta_{n-1}^{\circ},\\ q\in \mathfrak{r}\left(\mcB(\sigma)\right)}}c_{\sigma,\tau,m,q}(qe_{\tau})\otimes(me_{\sigma})^{*}\mathclose{\raisebox{-1.8ex}{$\Biggr)$}},
 \end{equation*}
 where $c_{\sigma,\tau,m,q}$ is the product of $a_{\sigma,\tau}$ in the boundary map and the coefficient of $q$ in $\mathfrak{r}(m)$. 
 The entries of $A$ in the row labeled by $(\tau,q)$ and the column labeled by $(\sigma,m)$ is $c_{\sigma,m,\tau,q}$. 
 Notice that by Proposition \ref{prop:Grd_is_top_homology_of_Qd}, the kernel of $A$ is precisely the space $G^{r}_{d}(\GrDomain)$. 
 Therefore, a basis of this kernel give us a basis for the vector space $G^r_d(\Delta,\Phi)$, which is the list $\textsc{LB}$ of tuples $\bigl(f(\sigma)\colon\sigma\in\Delta_{n}\bigr)$ of functions $f(\sigma)\in\ambRing(\sigma)_{\leqslant d}$ in Algorithm \ref{algorithm:basis_computation_Gr}.
 
 \begin{algorithm}
 \caption{Basis computation of the space $G^{r}_{d}(\GrDomain)$}\label{algorithm:basis_computation_Gr}
 \begin{algorithmic}
 \State Fix a graded reverse lexicographic order $\succ$ for polynomial ring $\otimes_{\sigma\in\Delta_{n}}\ambRing(\sigma)$;
 \State Compute the reduced Gr\"{o}bner basis $\textsc{GB}(\tau)$ of $\idealComplex(\tau)$ for each $\tau\in\Delta_{n-1}^{\circ}$;
 \For{$(\sigma,\tau)$ in $\Delta_{n}\times \Delta_{n-1}^{\circ}$ with $\tau\subseteq\sigma$}
 \State $\mathrm{sign}(\sigma,\tau) \gets $ sign of $\partial_{n}$ from $\sigma$ to $\tau$;
 \For{$m$ in the list of all the monomials in $\ambRing(\sigma)_{\leqslant d}$}
 \State $\mathfrak{r}(m) \gets$ remainder of $m$ with respect to $\textsc{GB}(\tau)$;
 \For{$q$ in the list of all monomials in $\ambRing(\tau)_{\leqslant d}$}
 \State $c(\sigma,\tau,m,q) \gets$ coefficient of $q$ in $\mathfrak{r}(m)$ multiplied by $\mathrm{sign}(\sigma,\tau)$;
 \EndFor
 \EndFor
 \EndFor
 \State Define a matrix $A$ by letting the entry of $A$ in row $(\tau,q)$ and column $(\sigma,m)$ to be $c(\sigma,\tau,m,q)$;
 \State Calculate a list of basis $LB$ of kernel of $A$, which is a basis of $G^{r}_{d}(\GrDomain)$;
 \end{algorithmic}
 \end{algorithm}
 
 In Section \ref{section:explicit_computations}, the computational results on the dimension of a $G^1$-spline space on a cube are obtain in \texttt{Macaulay2} \cite{M2} by an implementation of Algorithm \ref{algorithm:basis_computation_Gr}.
 
 \section{Topological surfaces and $G^r$-domains}\label{Section:Gsplines_by_differential_geometric_method}
 The notion of geometric continuity was originally defined by differential-geometric methods, see for example \cite{hahn1989geometric} and \cite{peters2002_handbook}. 
 In \cite{mourrain2016dimension}, the space of $G^{1}$-spline functions is defined over a collection of 2-dimensional patches glued together by transition maps. 
 In this section, we explain how the notion of $G^1$-geometric continuity in our framework and that in \cite{mourrain2016dimension} are related. 
 
 First, recall from \cite[Definition 2.1]{mourrain2016dimension}, a \emph{topological surface} $\mcM$ is a collection $\mcM_2$ of polygons together with a set of homeomorphisms between pairs of polygonal edges with the condition that each polygonal edge can be paired with at most one other edge, and it cannot be glued with itself. 
 Each polygon is assumed to be embedded in $\RR^2$.
 By an abuse of notation, two polygons are said to share an edge if there is such a homeomorphism between two edges one in each of the two polygons.
 From this, it is clear that the collection of polygons and homeomorphisms between pairs of polygonal edges of a topological surface $\mcM$ define a $2$-dimensional cell complex $\Delta$ such that the set of 2-faces $\Delta_{2}=\mcM_{2}$. 
 The coordinates $X(\sigma)$ associated to each polygon $\sigma\in\Delta_2$ are given by the embedding of $\sigma$ in $\RR^{2}$.
 As in Section \ref{section_polyhedral_frame}, we denote the corresponding homeomorphism by $\psi_{\sigma}\colon \sigma\rightarrow X(\sigma)$. 
 Notice that these embeddings of the faces $\sigma\in\Delta_2$ as polygons in $\RR^2$ mean that every edge or $\sigma$ is locally algebraic (see Definition \ref{def:localg}). 
 
 Suppose $\Delta$ is such a 2-dimensional cell complex obtained from a topological surface $\mcM$, with coordinates $\bigl\{X(\sigma)\colon \sigma\in\Delta_{2}\bigr\}$. 
 The (geometric) transition maps associated to pairs of polygons in $\mcM$ as introduced in \cite[Definition 2.2]{mourrain2016dimension} can be interpreted in our setting as follows.
 \begin{definition}[Geometric transition maps]\label{def_transition_map}
 Let $\sigma_1$ and $\sigma_2$ two faces in $\Delta_2$, such that $\sigma_1\cap\sigma_2=\tau\in\Delta_1$. 
 A \emph{(geometric) transition map} associated to $\tau$ consists of the following:
 \begin{itemize}
 \renewcommand{\labelitemi}{\scriptsize$\blacksquare$}
 \item for each face $\sigma_i\in\Delta_2$ containing $\tau$, an open set $U_{\tau,\sigma_i}\subseteq X(\sigma_i)$ containing $\psi_{\sigma_i}(\tau)$;
 \item a $C^{1}$-diffeomorphism between the open sets
 \begin{equation}\label{eqn_transition_map}
 \varphi_{21}
 \colon 
 U_{\tau,\sigma_{1}}\to U_{\tau,\sigma_2}
 \end{equation}
 such that\,
 $\varphi_{21}|_{\psi_{\sigma_{1}}(\tau)}=\psi_{\sigma_{2}}\circ\psi_{\sigma_{1}}^{-1}|_{\psi_{\sigma_{1}}(\tau)}$\,.
 \end{itemize} 
 \end{definition}
 
 \begin{remark}\label{rmk:geometric_transition_maps}
 The map $\varphi_{21}$ is called transition map in \cite{mourrain2016dimension}, we called it here geometric transition map to distinguish it from the algebraic transition maps from Definition \ref{def:algebraic_transition_maps}. 
 Notice that the notion of geometric transition maps can be generalized to an $n$-dimensional cell complex $\Delta$, for this we can simply replace $\Delta_{2}$ with $\Delta_{n}$, and $\Delta_{1}$ with $\Delta_{n-1}$ in Definition \ref{def_transition_map}.
 \end{remark}
 Our aim is to obtain a $G^{r}$-domain from a given topological surface and a given collection of geometric transition maps. However, a geometric transition map associated to an edge $\tau$, which is $C^r$-continuous along $\tau$, is not necessarily polynomial and therefore it does not necessarily corresponds to an algebraic transition map.
 To describe the type of geometric transition maps from which we are able to obtain algebraic ones, we introduce the property of being \emph{almost algebraic}. 
 We need some preparations before we can state such a notion in Definition \ref{def_almost_algebraic}.
 
 Let $U\subseteq\RR^{n}$ with coordinates $(x_{1},\dots,x_{n})$, and take $p\in U$. 
 If $\bm{j}=(j_{1},\dots,j_{n})\in\NN^{n}$ is a multiindex, henceforth we follow the convention that $\bm{x}^{\bm{j}}=x_{1}^{j_{1}}\cdots x_{n}^{j_{n}}$ and $|\bm{j}|=\sum_{i=1}^{n}j_{i}$.
 
 Recall that if $s\leqslant r$, then the $s$-th order \emph{jet} $J^{s}_{p} (f)$ of a function $f$ which is $C^{r}$-continuous at the point $p$ is defined as the truncated Taylor expansion of $f$ at $p$ of degree $s$. Explicitly, if the coordinates of $p$ are $(a_{1},\dots,a_{n})= \bm{a}$, then
 \begin{equation}\label{eqn:taylor_truncation}
 \begin{split}
 J^{s}_{p}\colon C^{r}_{p}&\to{\RR[x_{1},\dots,x_{n}]}/{I(p)^{s+1}}\\
 f&\mapsto\sum_{|\bm{j}|\leqslant s}\frac{\partial^{\bm{j}}f}{\partial \bm{x}^{\bm{j}}}(p)(\bm{x}-\bm{a})^{\bm{j}},
 \end{split} 
 \end{equation}
 where $C^{s}_{p}$ is the set of all functions that are $C^{s}$-continuous at $p$, and
 \begin{equation}\label{eq:idealp}I(p)=\bigl\langle x_{1}-a_{1},\dots,x_{n}-a_{n}\bigr\rangle
 \end{equation}
 is the ideal of $\RR[x_1,\dots, x_n]$ generated by the polynomials $x_i-a_i$.
 
 Denote by $C^{1}(U)$ the set of all functions $f\colon U\rightarrow \RR$ which are continuously differentiable functions on $U\subseteq \RR^n$.
 \begin{definition}\label{def_vanishing_order}
 If $f\in C^{1}(U)$, we say that $f$ \emph{vanishes to order $s$ at a point $p\in U$} if $f$ is $C^{s}$-continuous at $p$ and
 \begin{equation*}
 J^{s}_{p}(f)=0.
 \end{equation*}
 If $S\subseteq U$, then we say $f$ \emph{vanishes to order $s$ on $S$} if $f$ vanishes to order $s$ at $p$ for each $p\in S$.
 \end{definition}
 Since a geometric transition map $\varphi_{{21}}$ is a $C^{1}$-diffeomorphism, then it induces a homomorphism 
 \begin{align}\label{eq:pullback}
 \varphi_{{21}}^{*}\colon C^{1}(U_{\tau,\sigma_{1}})&\to C^{1}(U_{\tau,\sigma_{2}})\\ \nonumber
 f&\mapsto f\circ\varphi_{{21}}.
 \end{align}
 \begin{definition}[Almost algebraic property]\label{def_almost_algebraic}
 Let $\varphi_{{21}}$ be a geometric transition map across $\tau$ as introduced in (\ref{eqn_transition_map}). 
 We say $\varphi_{{21}}$ is \emph{almost algebraic} (up to order $r$) if there exist a homomorphism 
 \begin{equation} \label{eq:atm}
 \phi_{21}\colon\ambRing(\sigma_{2})/I_{\sigma_{2}}(\tau)^{r+1}\to\ambRing(\sigma_{1})/I_{\sigma_{1}}(\tau)^{r+1}
 \end{equation}
 such that $\bigl(\widetilde{\phi}_{21}-\varphi^{*}_{{21}}\bigr)(f)$ vanishes to order $r$ along $\psi_{\sigma_{1}}(\tau)$, for any $f\in\ambRing(\sigma_{2})$, where $\varphi^*_{21}$ is the homomorphism in \eqref{eq:pullback}, and $\widetilde{\phi}_{21}\colon\ambRing(\sigma_{2})\to\ambRing(\sigma_{1})$ is a lift of $\phi_{21}$. If this is the case, we call $\phi_{21}$ the \emph{$G^r$-algebraic transition map associated to $\varphi_{21}$}. 
 \end{definition}
 It is clear that Definition \ref{def_almost_algebraic} does not depend on the choice of the lift $\widetilde{\phi}_{21}$. In Lemma \ref{lem:algebraic_trans_unique} we prove that if there exists a $G^{r}$-algebraic transition map associated to an geometric transition map $\varphi_{21}$, then it is unique.
 \begin{lemma}\label{lem:algebraic_trans_unique}
 Let $\tau=\sigma_1\cap\sigma_2\in\Delta_1^\circ$, for $\sigma_1,\sigma_2\in\Delta_2$. If $\varphi_{{21}}$ is an almost algebraic transition map up to order $r$ across $\tau$, then the $G^{r}$-algebraic transition map $\phi_{21}$ associated to $\varphi_{{21}}$ is unique. 
 \end{lemma}
 \begin{proof}
 Suppose both $\phi_{21}$ and $\phi_{21}'$ are algebraic transition maps induced by $\varphi_{{21}}$. 
 If $f\in\ambRing(\sigma_{2})$, then
 \begin{align*}
 \bigl(\widetilde{\phi}_{21}-\widetilde{\phi}_{21}'\bigr)(f)
 &=\bigl(\widetilde{\phi}_{21}-\varphi_{{21}}^{*}+\varphi_{{21}}^{*}-\widetilde{\phi}_{21}'\bigr)(f)\\
 &=\bigl(\widetilde{\phi}_{21}-\varphi_{{21}}^{*}\bigr)(f)+\bigl(\varphi_{{21}}^{*}-\widetilde{\phi}_{21}'\bigr)(f).
 \end{align*}
 Both $\bigl(\widetilde{\phi}_{21}-\varphi_{{21}}^{*}\bigr)(f)$ and $\bigl(\varphi_{{21}}^{*}-\widetilde{\phi}_{21}'\bigr)(f)$ vanishes to order $r$ along $\tau$. 
 Hence so does their sum. 
 Because $\widetilde{\phi}_{21}-\widetilde{\phi}_{21}'$ is algebraic, then $\bigl(\widetilde{\phi}_{21}-\widetilde{\phi}_{21}'\bigr)(f)\in I_{\sigma_{1}}(\tau)^{r+1}$.
 This holds for any $f\in\ambRing(\sigma_{2})$, so this implies $\phi_{21}=\phi_{21}'$.
 \end{proof}
 It follows from Lemma \ref{lem:algebraic_trans_unique} that given a collection $\Phi_{G}$ of geometric transition maps, which are almost algebraic up to order $r$, we obtain a collection of $G^{r}$-algebraic transition maps $\Phi_A$ associated to $\Phi_{G}$. 
 The compatibility conditions \ref{cond1}--\ref{cond3} in Definition \ref{def:compatibility_conditions_alg} applied to $\Phi_A$ lead to a collection of algebraic transition maps $\Phi \supseteq \Phi_A$ which (besides the transition maps between adjacent faces) contains all the algebraic transition maps between pair of $n$-faces whose intersection is nonempty, see Remark \ref{rem:PhiBasis}.
 In this case, we say that $\Phi$ is \emph{associated to} $\Phi_{G}$ and the $G^{r}$-domain $(\Delta,\Phi)$ is obtained from the topological surface $\mcM$ with geometric transition maps $\Phi_{G}$.
 
 In fact, for $2$-dimensional cell complexes and $r=1$, if the collection $\Phi$ of $G^1$-algebraic transition maps is compatible on $\Delta$ (i.e., it satisfies the compatibility conditions in Definition \ref{def:compatibility_conditions_alg}), then the maps in $\Phi_{G}$ satisfy the $G^1$-\emph{compatibility conditions at the vertices} of $\Delta$ given in \cite[Section 2.3]{mourrain2016dimension}.
 In Section \ref{sec:Examples}, we given several examples of $\Phi_{G}$ such that its associated $\Phi$ satisfies compatibility conditions, and thus defines a $G^{r}$-domain.
 
 We now relate the definition of geometric continuity introduced in \cite{mourrain2016dimension} and the one in our setting given in Section \ref{section:definition_construction}, Definition \ref{def:Gr_join_algebraic}. 
 \begin{definition}
 Let $\mcM$ be a topological surface and $\Delta$ the cell complex obtained from $\mcM$. 
 The \emph{space of differentiable functions} $\mcS^{1}(\mcM)$ over $\mcM$ is defined as the subset of $\oplus_{\sigma\in\Delta_{2}}\ambRing(\sigma)$ such that $f=(f_{\sigma})\in \mcS^{1}(\mcM)$ if for every pair $\sigma_{1},\sigma_{2}\in\Delta_2$ of adjacent faces, we have
 \begin{equation*}
 J_p^{1}\bigl(f_{\sigma_{1}}\circ\varphi_{{21}})
 =
 J_p^{1}(f_{\sigma_{2}}\bigr),~\mbox{for each}~p\in\sigma_{1}\cap\sigma_{2},
 \end{equation*}
 where $J_{p}^{1}$ is the $1$-st order jet of $f$, as defined in \eqref{eqn:taylor_truncation}. 
 The set of all differentiable functions $f\in \mcS^1(\mcM)$ such that 
 $\deg f_{\sigma}\leqslant d$ for each $\sigma\in\Delta_{2}$ is denoted $\mcS^{1}_{d}(\mcM)$ .
 \end{definition}
 For a topological surface $\mcM$, let $\Phi_{G}=\{\varphi_{21}\colon \sigma_1,\sigma_2\in\mcM_2 \mbox{\; share a common edge}\},$
 be a collection of geometric transition maps.
 The relation between the notation of geometric continuity is the following.
 \begin{theorem}\label{prop:relation_notions_geomeric_continuous_functions}
 If the $G^{1}$-domain $(\GrDomain)$ is obtained from a topological surface $\mcM$ with geometric transition maps $\Phi_{G}$, then
 \begin{equation*}
 G^{1}(\GrDomain)=\mcS^{1}(\mcM),
 \text{ and \, } 
 G^{1}_{d}(\GrDomain)=\mcS^{1}_{d}(\mcM).
 \end{equation*}
 In particular,
 $
 \dim G^{1}_{d}(\GrDomain)=\dim \mcS^{1}_{d}(\mcM).
 $
 \end{theorem}
 We prepare the proof of Proposition \ref{prop:relation_notions_geomeric_continuous_functions} with the following two lemmas. Similarly as before, we denote by $\bm{x}=(x_1,\dots, x_n)\in\RR^n$, and $\bm{j}=(j_1,\dots,j_n)\in\NN^n$ a multiindex, where $\bm{x}^{\bm{j}}=x_{1}^{j_{1}}\cdots x_{n}^{j_{n}}$ and $|\bm{j}|=\sum_{i=1}^nj_i$.
 \begin{lemma}\label{lemma_ec_poln}
 Let $p\in\RR^n$ and $f,\ell\in\RR[x_{1},\dots,x_{n}]$, with $\ell $ irreducible. 
 Then the following two conditions on $f$ are equivalent:
 \begin{enumerate}[(a)]
 \item the jet $J^{r}_{p}(f)=0$ whenever $\ell(p)=0$, \label{a} 
 \item $f\in\langle \ell^{r+1}\rangle$.\label{b}
 \end{enumerate}
 \end{lemma}
 \begin{proof}
 Since the ideal $\langle \ell^{r+1}\rangle\subseteq I(p)^{s+1}$ for all $s\leqslant r$ whenever $\ell(p)=0$, where $I(p)$ is the ideal in \eqref{eq:idealp}, then clearly \ref{b} implies \ref{a}.
 It remains to prove the converse. 
 We prove this by induction on the order of smoothness $r$.
 Note that condition \ref{a} is equivalent to saying that $\partial^{\bm{j}} f/\partial\bm{x}^{\bm{j}}(p)=0$ for all $|\bm{j}|\leqslant r$ whenever $
 \ell(p)=0$, which in turn, is equivalent to saying that $\partial^{\bm{j}} f/\partial\bm{x}^{\bm{j}}\in\langle \ell\rangle$ for all $|\bm{j}|\leqslant r$.
 
 We first prove that for any $s\geqslant 1$, if $f, {\partial f}/{\partial x_{1}},\ldots, ~{\partial f}/{\partial x_{n}}\in\langle \ell^{s}\rangle$, then $f\in\langle \ell^{s+1}\rangle$. 
 Indeed, if $f=g\ell^{s}$ for some $g\in \RR[x_{1},\dots,x_{n}]$, then ${\partial f}/{\partial x_{j}}=\bigl({\partial g}/{\partial x_{i}}\bigr)\cdot \ell^{s}+g\cdot s\ell^{s-1}\bigl({\partial \ell}/{\partial x_{j}}\bigr)$ for $j=1,\dots,n$. 
 Since $\ell$ is non-constant, then there exists at least one variable $x_{j}$ such that ${\partial \ell}/{\partial x_{j}}\neq 0$, so we have ${\partial \ell}/{\partial x_{j}}\not\in \langle \ell\rangle$. 
 Note that $\langle\ell\rangle$ is a prime ideal since $\ell$ is irreducible, so $\partial f/\partial x_{j}\in\langle \ell^{s}\rangle$ implies $g\in \langle \ell\rangle$. 
 Hence $f=g\cdot\ell^{s}\in\langle \ell^{s+1}\rangle$.
 If follows that for $r=1$, condition \ref{a} implies \ref{b}.
 
 Now assume that \ref{a} implies \ref{b} for any $r = 1,\dots,s$.
 For $j=1,\dots,n$, put $h_{j}=\partial f/\partial x_{j}$. 
 Note that for any point $p\in\RR^n$, if $J^{s+1}_{p}(f)=0$ then $J^{s}_{p}(h_{j})=0$. 
 Hence by the induction hypothesis, $h_{j}\in\langle\ell^{s}\rangle$, for $j=1,\dots,n$. 
 Since $J^{s+1}_{p}(f)=0$ implies $J^{s}_{p}(f)=0$, then applying \ref{a} to $f$ and $r=s+1$ yields $f\in\langle\ell^{s}\rangle$. 
 Thus, we have $f, ~{\partial f}/{\partial x_{1}},\ldots, ~{\partial f}/{\partial x_{n}}\in\langle \ell^{s}\rangle$. 
 Therefore, $f\in\langle\ell^{s+1}\rangle$, which completes the proof.
 \end{proof}
 
 \begin{lemma}\label{lemma:equivalent_conditions_Gr_join}
 Let $\Delta$ be a cell complex obtained from a topological surface $\mcM$, with 
 $\sigma_{1},\sigma_{2}\in\Delta_2$ two adjacent faces such that $\tau=\sigma_{1}\cap \sigma_{2}$. 
 Assume the transition map $\varphi_{{21}}$ across $\tau$ is almost algebraic, with associated algebraic transition map $\phi_{12}$, as given in \eqref{eq:atm}.
 Then, for polynomials $f\in\ambRing(\sigma_{1})$ and $g\in\ambRing(\sigma_{2})$, the following conditions are equivalent:
 \begin{enumerate}[(a)]
 \item\label{a-2} $J^{r}_{p}(f\circ\varphi_{{21}})=J^{r}_p(g)$ for each $p\in\tau$;
 \item\label{b-2}
 $\phi_{12}(\bar{f})=\bar{g}$, where $\bar{f}$ and $\bar{g}$ are the classes of $f$ and $g$ in $\ambRing(\sigma_{1})/I_{\sigma_{1}}(\tau)^{r+1}$ and $\ambRing(\sigma_{2})/I_{\sigma_{2}}(\tau)^{r+1}$, respectively.
 \end{enumerate}
 \end{lemma}
 \begin{proof}
 Let
 $
 \widetilde{\phi}_{12}\colon \ambRing(\sigma_{1})\to\ambRing(\sigma_{2})
 $ be a lift of $\phi_{12}$. 
 Then \ref{b-2} holds if and only if
 \begin{equation}\label{eqn_condition_b}
 \widetilde{\phi}_{12}(f)-g\in I_{\sigma_{2}}(\tau)^{r+1}.
 \end{equation}
 Since the edge $\tau$ is locally algebraic, then by Lemma \ref{lemma_ec_poln}, a polynomial $h\in\ambRing(\sigma_{2})$ vanishes along $\tau$ up to order $r$ if and only if $h\in I_{\sigma_{2}}(\tau)^{r+1}$. 
 Therefore, equation \eqref{eqn_condition_b} is equivalent to saying that $\widetilde{\phi}_{12}(f)-g$ vanishes along $\tau$ up to order $r$. 
 Because $J^{k}_{p}\bigl(\widetilde{\phi}_{12}(f)-g\bigr)=J^{k}_{p}\bigl(\widetilde{\phi}_{12}(f)\bigr)-J^{k}_{p}(g)$, then \eqref{eqn_condition_b} is also equivalent to
 \begin{equation}\label{eqn_ec_Gk_join}
 J^{r}_{p}\bigl(\widetilde{\phi}_{12}(f)\bigr)=J^{r}_{p}(g),~\mbox{for each}~p\in\tau.
 \end{equation}
 Since
 $J^{r}_{p}\bigl(\widetilde{\phi}_{12}(f)\bigr)
 =
 J^{r}_{p}\bigl(f\circ\varphi_{{21}}\bigr)$ for any $p\in\tau$, then \eqref{eqn_ec_Gk_join} holds if and only if condition \ref{a-2} holds.
 \end{proof}
 We are now ready to prove Theorem \ref{prop:relation_notions_geomeric_continuous_functions}.
 \begin{proof}[Proof to Theorem \ref{prop:relation_notions_geomeric_continuous_functions}]
 If $(f_{\sigma})\in G^{1}(\Delta,\Phi)$, then for any pair of adjacent faces $\sigma_{1},\sigma_{2}\in\Delta_2$, the polynomials $f_{\sigma_{1}}\in\ambRing(\sigma_{1})$ and $f_{\sigma_{2}}\in\ambRing(\sigma_{2})$ satisfy condition \ref{b-2} in Lemma \ref{lemma:equivalent_conditions_Gr_join}, and hence also condition \ref{a-2}. This implies $(f_{\sigma})\in \mcS^{1}(\mcM)$. 
 Conversely, a similar argument implies $\mcS^{1}(\mcM)\subseteq G^{1}\bigl(\Delta,\Phi\bigr)$, and so we have $G^{1}_{d}(\Delta,\Phi)=\mcS^{1}_{d}(\mcM)$.
 \end{proof}

 \section{Examples of $G^r$-domains} \label{sec:Examples}
 This section is devoted to examples of $G^r$-domains $(\Delta,\Phi)$.
 Recall from Remark \ref{rem:PhiBasis} in Section \ref{section:definition_construction} that to describe a collection of algebraic transition maps $\Phi$ we only need to give the set of transition maps in $\Phi$ associated to each pair of adjacent $n$-faces in $\Delta$.
 
 \begin{example}\label{ex:gen2patches}
 Let $\Delta$ be an $n$-dimensional polyhedral complex composed by any two adjacent $n$-faces $\sigma_{1}, \sigma_{2}$ sharing a common $(n-1)$-face $\tau$, for some $n\geqslant 2$.
 Similarly as above, we may assume up to a change of coordinates that the ideals $I_{\sigma_i}(\tau)$ of $\tau$ in $\ambRing(\sigma_i)$ are given by
 \begin{alignat*}{3}
 \ambRing(\sigma_{1})&=\RR[u_{11}, \dots, u_{1n}],&\qquad \ambRing(\sigma_{2})&=\RR[u_{21},\dots, u_{2n}],\\
 I_{\sigma_{1}}(\tau)&=\langle u_{11}\rangle,
 &\quad I_{\sigma_{2}}(\tau)&=\langle u_{2n}\rangle.
 \end{alignat*} 
 From this, we see that $\tau$ is a locally algebraic face of $\Delta$. 
 A $G^r$-algebraic transition map is of the form
 \begin{align*}
 \phi_{12}\colon \RR[u_{11},\dots, u_{1n}]/\langle u_{11}^{r+1}\rangle &\to\RR[u_{21},\ldots, u_{2n}]/\langle u_{2n}^{r+1}\rangle
 \\
 \phi_{12}(u_{1i})&=
 \begin{cases}
 \sum_{j=1}^{r} u_{2n}^{j}p_{j} &\mbox{ if } i=1,\\
 u_{2(i-1)} + \sum_{j=1}^{r} u_{2n}^{j}q_{2j} &\mbox{ otherwise},
 \end{cases}
 \end{align*}
 where $q_{2j}$, $p_j$ are polynomials in $\RR[u_{21},\dots,u_{2(n-1)}]$.
 \hfill$\diamond$ 
 \end{example}
 \subsection{Algebraic transition maps associated to a topological surface}
 In the following examples we obtain $G^{1}$-domains from a topological surface and geometric transition maps as described in Section \ref{Section:Gsplines_by_differential_geometric_method}. 
 
 \begin{example}\label{ex:twoSquares}
 Let $\mcM$ be the topological surface in Figure \ref{fig:gluing_data} which is composed by two square patches $\sigma_1$ and $\sigma_2$ together with a homeomorphism between the identified edges $\tau_{i}\subseteq\sigma_i$, for $i=1,2$. 
 This example was studied in 
 \cite{blidia2017}, and we can considered it in our setting as follows. 
 As we described in Section \ref{Section:Gsplines_by_differential_geometric_method}, a 2-dimensional polyhedral complex $\Delta$ is naturally defined by $\mcM$, it consists of two square faces $\sigma_1$ and $\sigma_2$ joining along a common edge $\tau$ which arises from the homeomorphism $\tau_{1}\to \tau_{2}$.
 
 \begin{figure}[ht!]
 \centering
 \includegraphics[scale=1.5]{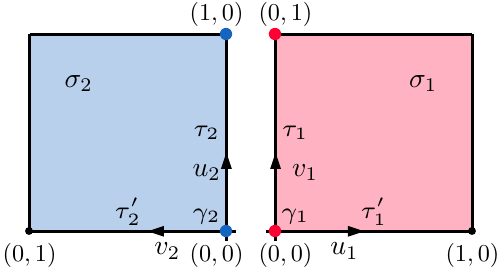}
 \caption{Topological surface $\mcM$ composed by the two squares together and a homeomorphism between the edges $\tau_2\subseteq\sigma_2$ and $\tau_1\subseteq \sigma_1$
 In Example \ref{ex:twoSquares}, the edge $\tau$, and the vertex $\gamma$, denote the identification via this homeomorphism between the edges $\tau_i$ and vertices $\gamma_i$, respectively.
 }\label{fig:gluing_data}
 \end{figure}
 Up to a change of coordinates, we may assume that the ideals $I_{\sigma_i}$ of polynomials vanishing at the edge $\tau$ in the coordinate rings $\ambRing(\sigma_i)$ (in Definition \ref{def:ideals}) are given by:
 \begin{alignat*}{3}
 \ambRing(\sigma_1)&=\RR[u_1,v_1],&\qquad
 \ambRing(\sigma_2)&=\RR[u_2,v_2], \\
 I_{\sigma_1}(\tau)&=\langle u_1\rangle,&
 I_{\sigma_2}(\tau)&=\langle v_2\rangle.
 \end{alignat*}
 In particular, because both ideals $I_{\sigma_i}$ are generated by an irreducible polynomials in $\ambRing(\sigma_i)$ then $\tau$ is locally algebraic, see Definition \ref{def:localg}. 
 Following Definition \ref{def:algebraic_transition_maps}, a $G^1$-algebraic transition map is an $\RR$-algebra homomorphism which in this case can be written as
 \begin{equation}\label{eq:gluing_data_transitionOne}
 \begin{split}
 \phi_{12}\colon \RR[u_1,v_1]/\langle u_1^{2}\rangle&\to\RR[u_2,v_2]/\langle v_2^{2}\rangle\\
 u_1&\mapsto v_2\glueb(u_2)\\
 v_1&\mapsto u_2+v_2\gluea(u_2),
 \end{split}
 \end{equation}
 where $\gluea(u_2)$ and $\glueb(u_2)$ are univariate polynomials.
 \hfill$\diamond$
 \end{example}
 \begin{remark}\label{rem:transmapsformGluingData}
 In Section \ref{Section:Gsplines_by_differential_geometric_method} we described how to obtain a $G^{r}$-domain from a topological surface. If $(\GrDomain)$ is the $G^{1}$-domain obtained from the topological surface $\mcM$ in Example \ref{ex:twoSquares}, then the polynomials from \eqref{eq:gluing_data_transitionOne} define the couple $[\gluea,\glueb]$ which in \cite{mourrain2016dimension} is called the \emph{gluing data} at the vertex $\gamma$ along the edge $\tau$. 
 In this case, we say that an algebraic transition map $\phi_{12}$ on $\Delta$ is \emph{defined from the gluing data}.
 \end{remark}

 \begin{definition}[$G^1$-gluing data]\label{def:G1_gluingdata}
 Let $\Delta$ be a $2$-dimensional polyhedral complex, and let $\sigma_1,\sigma_2$ be two adjacent $2$-faces in $\Delta$ such that $\sigma_1\cap \sigma_2=\tau\in\Delta_{1}^\circ$. 
 We write $(u_i,v_i)$ for the coordinates of $\sigma_i$ as in Figure \ref{fig:gluing_data}. 
 If the algebraic transition map $\phi_{12}$ is of the form in \eqref{eq:gluing_data_transitionOne}, then we say that $\phi_{12}$ is defined by the \emph{gluing data} $[\gluea_{12}^\tau, \glueb_{12}^\tau]$, where $\gluea_{12}^\tau,\glueb_{12}^\tau\in\RR[u_2]$.
 We refer to $[\gluea_{12}^\tau, \glueb_{12}^\tau]$, or simply $[\gluea_{12}, \glueb_{12}]$, as the gluing data from $\sigma_1$ to $\sigma_2$, or associated to $\tau$ if the order of the $n$-faces is clear from the context. 
 
 \end{definition}
 \begin{remark}
 In Definition \ref{def:G1_gluingdata}, for simplicity, we use $(u_1, v_1)$ and $(u_2,v_2)$ as relative coordinates to define the gluing data associated to an edge. 
 This choice of coordinates are related to Figure \ref{fig:gluing_data} and they are independent of the notation $X(\sigma_1)$ and $X(\sigma_2)$ when $\Delta$ has more than one interior edge.
In the cases we consider in this paper, the gluing data associated to different edges of a given partition $\Delta$ is given by the same formula once it is written in the relative coordinates $(u_1,v_1)$, $(u_2,v_2)$ for each $\tau\in\Delta_1^\circ$. 
If that is the case, we simply write $[\gluea,\glueb]$, and by an abuse of notation, we say that the collection of  $G^{1}$-algebraic transition map $\Phi$ is given by the gluing data $[\gluea,\glueb]$.
\end{remark}
 
 
 For the next example, recall that the \emph{valence} of a vertex $\gamma$ in a $2$-dimensional polyhedral cell complex $\Delta$ is the number of edges (or $1$-faces) in $\Delta$ which contain $\gamma$ as one of their vertices.
 \begin{example}[Symmetric gluing data]\label{section:symmetric_gluing_data}
 Here is an example of transition maps for $G^1$-splines between two adjacent faces of a surface using the so-called \emph{symmetric gluing data} proposed in \cite [\S 8.2]{hahn1989geometric}, see also \cite[Example 2.5]{mourrain2016dimension} where this example is studied in the setting of a topological surface and (geometric) transition maps. 
 
 Let $\tau$ be the edge in $\Delta$ with vertices $\gamma$ and $\gamma'$ which is
 shared by the two adjacent $2$-faces $\sigma_{1},\sigma_{2}\in\Delta$.
 As in Example \ref{ex:twoSquares}, we denote by $\phi_{12}$ the transition map from $\sigma_1$ to $\sigma_2$, which is given by \eqref{eq:gluing_data_transitionOne}. Take
 \begin{equation}\label{eq:gluing_data_symmetric}
 \begin{split}
 \gluea(u_2) & = 2 \cos \left(\frac{2\pi}{w}\right) h (u_2) 
 - 2\cos \left(\frac{2\pi}{w'}\right) q (u_2), \text{ and} \\
 \glueb(u_2) & = -1 , 
 \end{split} 
 \end{equation}
 where $w$ and $w'$ are the valences of the vertices $\gamma$ and $\gamma'$, respectively, and $h$ and $q$ are univariate polynomials. 
 Then, the algebraic transition map $\phi_{12}\colon \RR[u_{1},v_{1}]/\langle v_{1}^{2}\rangle\to\RR[u_{2},v_{2}]/\langle u_{2}^{2}\rangle$ with $G^1$-gluing data $[\gluea,\glueb]$ from $\sigma_1$ to $\sigma_2$ is given by
 \begin{align*}\label{eq:hahngluing}
 \phi_{12}(u_{1})& = -v_{2},\\
 \phi_{12}(v_{1})& = u_{2} + 2v_{2} \biggl(\cos \left(\frac{2\pi}{w}\right) h (u_{2}) 
 - \cos \left(\frac{2\pi}{w'}\right) q (u_{2}) \biggr).
 \end{align*}
 Assume without loss of generality that $\gamma$ and $\gamma'$ as vertices of $\sigma_2$ have coordinates $u_2=0$, $v_2=0$, and to $u_2=1$, $v_2=0$, respectively.
 To satisfy the compatibility conditions 
 required in Definition \ref{def:compatibility_conditions_alg}, we will additionally require that the functions $h$ and $q$
 interpolate $0$ and $1$. 
 Namely, we require $h(0)=1,h(1)=0,q(0) =0$ and $q(1)=1$.
 A possible solution is to take 
 \begin{equation}\label{eq:symgluedata}
 \gluea(u_2) = 2\cos \left(\frac{2\pi}{w}\right) (1-u_2)^2
 - 2\cos\left(\frac{2\pi}{w'}\right) u_2^2,\quad \text{ and } \glueb=-1.
 \end{equation}
 Then, if $\phi_{12}$ is given by $[\gluea,\glueb]$ as in \eqref{eq:symgluedata}, a good property of this kind of transition maps is that the gluing data of the inverse map $\phi_{21}$ from $\sigma_2$ to $\sigma_1$ of $\phi_{12}$ is of the same form \eqref{eq:symgluedata}. 
 Explicitly, let
 \[ u_{1}'=1-v_{1},~v_{1}'=u_{1},~u_{2}'=v_{2},~v_{2}'=1-u_{2},
 \]
 and
 \begin{equation*}
 \gluea(u_{1}')=2\cos \left(\frac{2\pi}{w'}\right) (1-u_{1}')^2
 - 2\cos\left(\frac{2\pi}{w}\right) u_{1}'^2, \text{ and } \glueb=-1.
 \end{equation*}
 Then the map $\phi_{21}\colon\RR[u_{2},v_{2}]/\langle v_{2}^{2}\rangle\to\RR[u_{1},v_{1}]/\langle u_{1}^{2}\rangle$ given by
 \begin{align*}
 \phi_{21}(u_{2}')& = v_{1}'\glueb,\\
 \phi_{21}(v_{2}')& = u_{1}' + v_{1}'\gluea(u_{1}')
 \end{align*}
 is the algebraic transition map from $\sigma_2$ to $\sigma_1$ which is the inverse of $\phi_{12}$.
 \hfill$\diamond$
 \end{example}
 
 \begin{example}[Bilinearly parametrized patches]\label{eg:bilinear}
 \begin{figure}[ht!]
 \centering
 \includegraphics[scale=1.5]{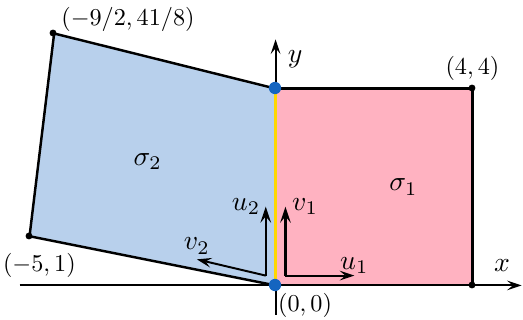}
 \caption{
 Two bilinearly parameterized patches considered in Example \ref{eg:bilinear}. 
 We embed the polyhedral complex $\Delta$ consisting of these two patches $\sigma_{1},\sigma_2$ in the $(x,y)$-plane, and write $(u_{i},v_{i})$ for the coordinate system associated to $\sigma_i$.
 }
 \label{fig:twpatch-rectangle} 
 \end{figure}
 In this example we look at the case of two quadrilateral patches, embedded in the $(x,y)$-plane (cf. Figure~\ref{fig:twpatch-rectangle}).
 The location of the patches in the $(x,y)$-plane is given by \emph{bilinear parametrizations}, which means that the inverse of each map $\psi_{i} \colon \sigma_{i}\,\to X(\sigma_{i})\cong \RR^{2}$, defined by $(x,y) \mapsto(u_{i},v_{i})$
 is described by two polynomials in $\RR[u_{i},v_{i}]_{\leqslant (1,1)}$. 
 Since in this case we always assume the image of $\psi_{i}$ is a unit square, we shall refer to $\psi_{i}^{-1} \colon [0,1]^{2}\to \sigma_{i}$ as the parametrization of the face $\sigma_i$.
 It is uniquely determined by four corner points $\psi_{i}^{-1}(0,0),\, \psi_{i}^{-1}(0,1),\, \psi_{i}^{-1}(1,0)$, $\psi_{i}^{-1}(1,1)$ of $\sigma_i$.
 Explicitly, assume $\psi_{1}^{-1}$ and $\psi_{2}^{-1}$ are defined by
 \begin{equation*}
 \psi_{1}^{-1}(u_{1},v_{1})=(a_{11}u_{1}v_{1}+a_{10}u_{1}+a_{01}v_{1}+a_{00},~b_{11}u_{1}v_{1}+b_{10}u_{1}+b_{01}v_{1}+b_{00}),
 \end{equation*} 
 and
 \begin{equation*}
 \psi_{2}^{-1}(u_{2},v_{2})=(c_{11}u_{2}v_{2}+c_{10}u_{2}+c_{01}v_{2}+c_{00},~e_{11}u_{2}v_{2}+e_{10}u_{2}+e_{01}v_{2}+e_{00}).
 \end{equation*} 
 Then 
 \begin{equation*}
 \begin{pmatrix}
 a_{11}& b_{11}& c_{11}& e_{11}\\
 a_{10}& b_{10}& c_{10}& e_{10}\\
 a_{01}& b_{01}& c_{01}& e_{01}\\
 a_{00}& b_{00}& c_{00}& e_{00}
 \end{pmatrix}=\begin{pmatrix}
 0&0&0&1\\
 0&1&0&1\\
 0&0&1&1\\
 1&1&1&1
 \end{pmatrix}^{-1}\begin{pmatrix}
 \psi_{1}^{-1}(0,0)& \psi_{2}^{-1}(0,0)\\
 \psi_{1}^{-1}(1,0)& \psi_{2}^{-1}(1,0)\\
 \psi_{1}^{-1}(0,1)& \psi_{2}^{-1}(0,1)\\
 \psi_{1}^{-1}(1,1)& \psi_{2}^{-1}(1,1)
 \end{pmatrix}.
 \end{equation*}
 By choosing coordinates, we may always assume that in the $(x,y)$-plane, the edge $\tau$ shared by the two faces lies on the line defined by $x=0$, and that $\psi_{1}^{-1}(0,0)=\psi_{2}^{-1}(0,0)=(0,0)$. 
 We may also assume $I_{\sigma_{1}}(\tau)=\langle u_{1}\rangle$, and $I_{\sigma_{2}}(\tau)=\langle v_{2}\rangle$. 
 This choice of coordinates implies 
 \begin{equation*}
 a_{01}=a_{00}=b_{00}=c_{10}=c_{00}=e_{00}=0.
 \end{equation*}
 Hence, $\psi_{1}$ defines a map
 \begin{equation*}
 \begin{split}
 \eta_{1}\colon \RR[x,y]/\langle x^{2}\rangle&\to \ambRing(\sigma_{1})/\langle u_{1}^{2}\rangle,\\
 x&\mapsto a_{11}u_{1}v_{1}+a_{10}u_{1},\\
 y&\mapsto b_{11}u_{1}v_{1}+b_{10}u_{1}+b_{01}v_{1}.
 \end{split} 
 \end{equation*}
 Similarly, $\psi_{2}$ defines a map
 \begin{equation*}
 \begin{split}
 \eta_{2}\colon \RR[x,y]/\langle x^{2}\rangle&\to \ambRing(\sigma_{2})/\langle v_{2}^{2}\rangle\\
 x&\mapsto c_{11}u_{2}v_{2}+c_{01}v_{2},\\
 y&\mapsto e_{11}u_{2}v_{2}+e_{10}u_{2}+e_{01}v_{2}.
 \end{split} 
 \end{equation*}
 If there is a map $\phi_{12}$ defined from gluing data $\gluea(u_2)$ and $\glueb(u_2)$ as in \eqref{eq:gluing_data_transitionOne}, and such that $\phi_{12}\circ\eta_{1}=\eta_{2}$, then we say \emph{the gluing data is determined by the parametrization of $\sigma_1$ and $\sigma_2$}. 
 Explicitly, $\phi_{12}\circ\eta_{1}=\eta_{2}$ means that the system of equations
 \begin{equation}\label{eq:gluing_data_bilinear}
 \begin{cases}
 (a_{11}u_{2}+a_{10})\glueb(u_{2})=c_{11}u_{2}+c_{01}\\
 (b_{11}u_{2}+b_{10})\glueb(u_{2})+b_{01}\gluea(u_{2})=e_{11}u_{2}+e_{01}
 \end{cases}
 \end{equation}
 holds. 
 Hence, the existence of such a $\phi_{12}$ is equivalent to say that \eqref{eq:gluing_data_bilinear} has a polynomial solution for both $\gluea(u_{2})$ and $\glueb(u_{2})$. 
 In this case, it forces $a_{11}=b_{11}=0$.
 
 Using this gluing data, the parametric spline functions defined over these specific quadrilaterals on the $(x,y)$-plane are geometrically continuous over the $G^{1}$-domain $(\GrDomain)$, where $\Delta$ is the cell complex composed by the two patches sharing an edge and $\Phi$ is determined by $\phi_{12}$.
 
 This construction of gluing data has been extensively studied for $G^1$-splines and isogeometric applications on planar domains 
 \cite{kapl2017,kapl2019isogeometric,KAPL201955}, and volumetric two-patch domains \cite{birner:hal-02271820,BirnerKapl19}. 
 There, the gluing data is defined using the determinants of the minors of the following $2\times 3$ matrix
 \[
 \begin{pmatrix}
 \partial_{u_1} \psi^{-1}_1 & \partial_{v_1} \psi^{-1}_1& \partial_{v_2} \psi^{-1}_2
 \end{pmatrix}.
 \]
 Explicitly,
 \begin{align*}
 \gluea (u_2)&= \bigl({ \det(\partial_{u_1} \psi^{-1} _1 , ~\partial_{v_2} \psi^{-1} _2 ) }/{ \det \nabla \psi^{-1} _1 }\bigr) |_{v_1=u_2,~u_1=v_2=0} , \text{ and}\\
 \glueb (u_2)& = \bigl(- {\det (\partial_{v_{1}}\psi_{1}^{-1},~ \partial_{v_{2}}\psi_{2}^{-1}) }/{ \det \nabla \psi^{-1} _1}\bigr)|_{v_1=u_2,~u_1=v_2=0} \,, 
 \end{align*}
 and two polynomials $g_1$ and $g_2$ are said to join with $G^1$-continuity if and only if
 \[
 \bigl(\gluea (u_2) \, \partial_{v_1} g_1 + \glueb(u_2) \, \partial_{u_1} g_1 - \partial_{v_2} g_2\bigr) |_{v_{1}=u_{2},u_{1}=v_{2}=0} = 0\,.
 \]
 One may check that this construction of gluing data is equivalent to solve the system of equations \eqref{eq:gluing_data_bilinear}.
 
 As a particular example, we consider the planar patches on the $(x,y)$-plane where $\sigma_1$ has vertices $(0,0)$, $(0,4)$, $(4,0)$, and $(4,4)$. 
 The parametrization $\psi^{-1}_1$ is given by the equation
 \[
 \psi^{-1}_1(u_1,v_1) = ( 4 u_1, \, 4v_1).
 \]
 The patch $\sigma_{2}$ is a quadrilateral with vertices $(0, 0)$, $(0,4)$, $(-5,1)$, $(-9/2,41/8)$ on the $(x,y)$-plane. 
 The parametrization $\psi_{2}^{-1}$ is given by the bilinear map
 \[
 \psi_{2}^{-1}(u_2,v_2) = \left( \frac 1 2 u_2v_2- 5 v_2, \, \frac1 8 u_2v_2+4 u_2+ v_2\right).
 \]
 The common edge between the two patches has vertices $(0, 0)$ and $(0,4)$.
 The gluing data take the form
 \[
 \gluea(u_{2}) = -\frac {1}{32} u_2 - \frac{1}{4}, \text{ and\;} 
 \quad 
 \glueb(u_{2}) = -\frac 1 8 u_2 + \frac 5 4. 
 \]
 The corresponding transition map is given by
 $\phi_{12}(u_{1})=v_2 \glueb(u_{2})$, and
 $\phi_{12}(v_{1})=u_2 + v_2\gluea(u_{2})$. \hfill$\diamond$
 \end{example}
 
 \subsection{$G^{r}$-domains obtained from differential manifolds with cellulation} \label{section:obtain_transition_map_from_mfds}
 We denote by $M$ a differential $n$-manifold with a given atlas. 
 A \emph{cellulation} of $M$ is given by a homeomorphism $\cellulation\colon\Delta\to M$, such that $\Delta$ is a cell complex.
 Assume $\Delta$ is such a cell complex, and suppose for each $\sigma\in\Delta_{n}$ there is a unique chart $U(\sigma)$ in the given atlas such that $\cellulation(\sigma)\subseteq U(\sigma)$. 
 Then a system of coordinates $X(\sigma)$ on each $\sigma\in\Delta_n$ is determined by the atlas in a natural way, and for every $(n-1)$-face $\tau$ of $\Delta$ which is common to two $n$-faces, a geometric transition map across $\tau$ is defined by the atlas. 
 We denote this collection of geometric transition maps by $\Phi_{G}$. 
 We assume that every such $(n-1)$-face $\tau$ of $\Delta$ is locally algebraic (i.e., the ideal $I_\sigma(\tau)$ is generated by an irreducible polynomial for every $n$-face $\sigma\supseteq \tau$, see Definition \ref{def:localg}), and each map in $\Phi_{G}$ is almost algebraic (see Definition \ref{def_almost_algebraic}).
 Then, from $\Phi_G$, we can construct a set of compatible algebraic transition maps $\Phi$ which lead us to a $G^{r}$-domain over $\Delta$. 
 In this case, we simply say that $\Phi$ is \emph{obtained from the atlas} of the manifold $M$, and that the $G^{r}$-domain $(\GrDomain)$ is \emph{obtained from the cellulation} of $M$.
 
 In this section we give some examples in which we start with a differential manifold and construct $G^{r}$-domain. 
 In the first example, we consider an \emph{affine manifold} i.e., a differential manifold admitting an atlas where all transition maps are affine \cite{gu_manifold_2006}.
 \begin{example}\label{eg:affine_manifolds}
 Let $\Delta$ be a $1$-dimensional simplicial complex given by
 \begin{equation*}
 \Delta=\bigl\{\emptyset, [1], [2], [3], [12], [23], [13]\bigr\}.
 \end{equation*}
 A geometric realization of $\Delta$ is homeomorphic to a circle, as illustrated in Figure \ref{fig:circle}. 
 We denote the vertices by $[i]=\gamma_i$, for $i=1,2,3$, and the maximal faces are given by $\tau_{1}=[12]$, $\tau_{2}=[23]$, and $\tau_{3}=[13]$. 
 We choose coordinates $\psi_{i}\colon\tau_{i}\to X(\tau_{i})\cong\RR^{1}$, where each point $p\mapsto u_{i}(p)$ and such that $\psi_{i}(\tau_{i})=[0,1]$ on $X(\tau_{i})$, for $i=1,2,3$. 
 Then, the maps
 \begin{equation}\label{eqn:affine_alg_trans_maps_circle}
 \begin{split}
 \phi_{12}\colon\RR[u_{1}]/\langle (u_{1}-1)^{2}\rangle&\to \RR[u_{2}]/\langle u_{2}^{2}\rangle\\
 u_{1}&\mapsto u_{2}+1,\\
 \phi_{23}\colon\RR[u_{2}]/\langle (u_{2}-1)^{2}\rangle&\to \RR[u_{3}]/\langle u_{3}^{2}\rangle\\
 u_{2}&\mapsto u_{3}+1,\\
 \phi_{31}\colon \RR[u_{3}]/\langle (u_{3}-1)^{2}\rangle&\to \RR[u_{1}]/\langle u_{1}^{2}\rangle\\
 u_{3}&\mapsto u_{1}+1. 
 \end{split} 
 \end{equation}
 together with the identification map $\phi_{ii}$ on each $\tau_{i}$ form a collection $\Phi$ of algebraic transition maps satisfying the compatibility conditions in Definition \ref{def:compatibility_conditions_alg}. 
 Thus we get a $G^{1}$-domain $(\Delta,\Phi)$.
 
 We can see this $G^{1}$-domain $(\Delta,\Phi)$ as obtained from a cellulation of a differential $1$-fold as follows. 
 On $X(\tau_i)$, at each end point $\psi(\gamma_j)$ of the interval $\psi_{i}(\tau_{i})$, we take $U_{\gamma_j,\tau_i}$ to be a neighborhood of $\psi(\gamma_j)$ of radius ${1}/{4}$. 
 If $V_i=\psi_{i}(\tau_{i})\cup_{\gamma_j\subseteq\tau_{i}} U_{\gamma_j,\tau_{i}}$
 Then $\bigl\{(V_{i},\psi_{i})\colon i=1,2,3\bigr\}$ is an atlas of the circle, with transition maps 
 $\varphi_{12}\colon U_{\gamma_2,\tau_{2}}\to U_{\gamma_2,\tau_{1}}$, 
 $\varphi_{23}\colon U_{\gamma_3,\tau_{3}}\to U_{\gamma_3,\tau_{2}}$, and 
 $ \varphi_{31}\colon U_{\gamma_1,\tau_{1}}\to U_{\gamma_1,\tau_{3}}$, all defined by 
 $a\mapsto a-1$.
 Each of these maps $\varphi_{ij}$ is affine and therefore they are almost algebraic. 
 Thus, the maps in \eqref{eqn:affine_alg_trans_maps_circle} can be viewed as obtained from the atlas $\bigl\{(V_{i},\psi_{i})\colon i=1,2,3\bigr\}$.
 \hfill$\diamond$
 \end{example}
 Since a torus is also an affine manifold, a construction similar to that in Example \ref{eg:affine_manifolds} applies to a torus. 
 Next, we present an example in which the manifold is not affine.
 \begin{example}\label{eg:sphere_stereographic}
 \begin{figure}
 \centering
 \begin{subfigure}{0.45\textwidth} 
 \includegraphics{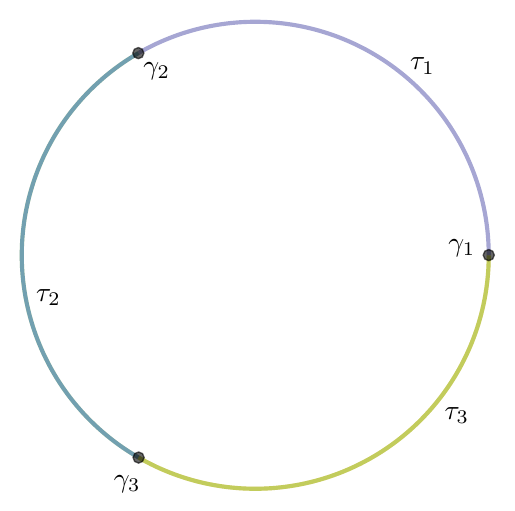}
 \caption{Cellulation of a circle by dividing it with three vertices. 
  In Example \ref{eg:affine_manifolds}, we use this cellulation to obtain a collection of algebraic transition maps and a $G^{1}$-domain over the circle.} \label{fig:circle} 
 \end{subfigure}
 \hspace{1em}
 \begin{subfigure}{0.45\textwidth}
 \includegraphics{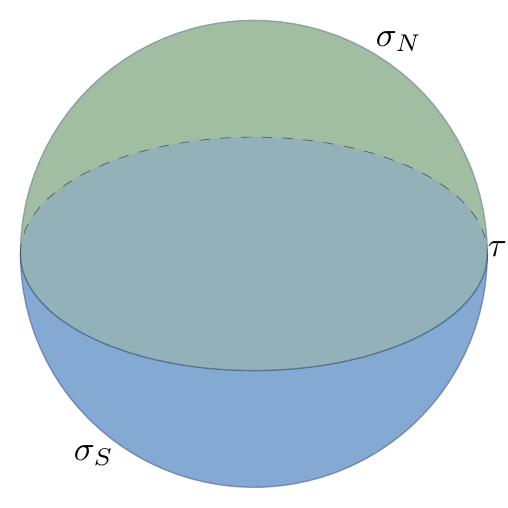}
 \caption{Cellulation of a sphere by dividing it with its equator. 
  In Example \ref{eg:sphere_stereographic}, we use this cellulation to obtain a collection of algebraic transition maps and a $G^{r}$-domain over the sphere.} \label{fig:sphere} 
 \end{subfigure} 
 \label{fig:sphereEx}
 \caption{Examples of differential manifolds with cellulation.}
 \end{figure}
 Let $M$ be a unit sphere in $\RR^3$. 
 We denote the north and south poles of $M$ by $N=(0,0,1)$ and $S=(0,0,-1)$, respectively,
 and consider the atlas consisting of two charts $M\setminus\{N\}$, and $M\setminus\{S\}$. 
 The coordinate system on each of these charts is given by the stereographic projection from $N$ and $S$, respectively.
 We consider a cell complex $\Delta$ given by this sphere divided by the $1$-cell $\tau$ corresponding to the points in the plane $z=0$, as in Figure \ref{fig:sphere}. 
 Let $\sigma_{S}$ and $\sigma_{N}$ be the 2-faces of $\Delta$ which correspond to the points of $M$ with $z\leqslant 0$ (the south hemisphere) and $z\geqslant 0$ (the north hemisphere), respectively. 
 Notice that each of these two 2-faces falls on only one of the charts of the atlas.
 Explicitly, we may consider $\sigma_{S}$ and $\sigma_{N}$ 
 in the coordinates $X(\sigma_{N})\cong\RR^{2}$ and
 $X(\sigma_{S})\cong\RR^{2}$ via the maps
 \begin{align*}
 \psi_{S}\colon \sigma_{S}&\to X(\sigma_{S})\\
 (x,y,z)&\mapsto(u_{S},v_{S}) =\left(\frac{x}{1-z},\frac{y}{1-z}\right),
 \intertext{and}
 \psi_{N}\colon\sigma_{N}&\to X(\sigma_{N})\\
 (x,y,z)&\mapsto (u_{N},v_{N})=\left(\frac{x}{1+z},\frac{y}{1+z}\right).
 \end{align*}
 The $1$-face $\tau$ shared by $\sigma_S$ and $\sigma_N$ is locally algebraic (see Definition~\ref{def:localg}), because if 
 $w_{S}=1-u_{S}^{2}-v_{S}^{2}$ and $w_{N}=1-u_{N}^{2}-v_{N}^{2}$, then $I_{\sigma_S}(\tau)=\langle w_S\rangle$ and $I_{\sigma_N}(\tau)=\langle w_N\rangle$, respectively. 
 Then, the map 
 \begin{align*}
 \phi_{SN}\colon \ambRing(\sigma_{S})/\langle w_{S}^{r+1}\rangle&\to\ambRing(\sigma_{N})/\langle w_{N}^{r+1}\rangle\\
 u_{S}&\mapsto u_{N}(1+w_{N}+\dots+w_{N}^{r}),\\
 v_{S}&\mapsto v_{N}(1+w_{N}+\dots+w_{N}^{r})
 \end{align*}
 is an algebraic transition map, which can be viewed as obtained from the given atlas of $M$. 
 The map $\phi_{SN}$, its inverse and the two identity maps on $\sigma_S$ and $\sigma_N$ form a collection of compatible algebraic transition maps $\Phi$, and so we obtain a $G^{r}$-domain $\bigl(\Delta, \Phi\bigr)$ from a cellulation of a sphere covered by the stereographic charts.
 \hfill$\diamond$
 \end{example}
 
 \section{Splines on 2-dimensional $G^r$-domains}\label{section:spline_complex_dim2}
 Throughout this section we assume $\Delta$ is a $2$-dimensional cell complex, and we take a collection $\Phi$ of compatible $G^r$-algebraic transition maps on $\Delta$, for some $r\geqslant 0$. 
 We will use the chain complex $(\quotientSpace_{d,\bullet},\delta)$ introduced in \eqref{eq:quotientQ} in Section \ref{section_BSS_deg_d} to study the dimension and construction of bases of the space $G^r_d(\Delta,\Phi)$ of splines on a $G^r$-domain $(\Delta,\Phi)$, for a given degree $d\geqslant 0$.

 Recall that by Proposition \ref{prop:Grd_is_top_homology_of_Qd} we have that $ G^{r}_{d}(\GrDomain)= H_{2}(\quotientSpace_{d,\bullet})$.
 Thus, if $\Delta$ is a $2$-dimensional cell complex, we have
 \begin{equation}\label{eq:dim2cellcase}
 \dim G^{r}_{d}(\GrDomain)=\chi(\quotientSpace_{d,\bullet})+\dim H_{1}(\quotientSpace_{d,\bullet})-\dim H_{0}(\quotientSpace_{d,\bullet}), 
 \end{equation}
 where 
 \begin{equation}\label{eq:Euler_char_Qd}
 \chi(\quotientSpace_{d,\bullet})=\sum_{\sigma\in\Delta_{2}}\dim \quotientSpace_{d}(\sigma)-\sum_{\tau\in\Delta_{1}^{\circ}}\dim\quotientSpace_{d}(\tau)+\sum_{\gamma\in\Delta_{0}^{\circ}}\dim\quotientSpace_{d}(\gamma), 
 \end{equation}
 is the Euler characteristic of the complex $\bigl(\quotientSpace_{d,\bullet},\delta\bigr)$. 
 
 A first estimate of $\dim G^{r}_{d}(\GrDomain)$ arises from considering 
 $\chi(\quotientSpace_{d,\bullet})$ in \eqref{eq:Euler_char_Qd}. 
 In the following, we compute the terms of $ \chi(\quotientSpace_{d,\bullet})$ when $r=1$ and the algebraic transition maps on $\Delta$ are defined from some gluing data (see Remark \ref{rem:transmapsformGluingData}). 
 The homology terms $\dim H_{1}(\quotientSpace_{d,\bullet}\bigr)$ and $\dim H_{0}(\quotientSpace_{d,\bullet})$ in \eqref{eq:dim2cellcase} are not necessarily zero, we consider them in Section \ref{section:explicit_computations}.
 
 Similarly as above, we write $(u_i,v_i)$ for the the coordinates of $X(\sigma_i)$, for a maximal face of $\Delta$, in this case a 2-face $\sigma_i\in\Delta_2$.
 If $\sigma_1\cap\sigma_2=\tau\in\Delta_1$ is an edge shared by two faces $\sigma_{1},\sigma_{2}\in\Delta_2$, the the tensor product ring in \eqref{eq:ringalpha} is simply $\ambRing(\tau)=\ambRing(\sigma_{1})\otimes\ambRing(\sigma_{2})$, and the subset of polynomials without mixed terms \eqref{eq:ringT} can be written as
 \begin{equation*}
 \totalSpace_{d}(\tau)=\bigl\{f_{1}+f_{2}\in\ambRing(\tau)\colon f_{i}\in\ambRing_d(\sigma_{i})\bigr\}.
 \end{equation*}
 The quotient 
 \eqref{eq:quotientQ} is in this case
 $\quotientSpace_{d}(\tau)=\totalSpace_{d}(\tau)/\subSpace_{d}(\tau)$, 
 where $\subSpace_{d}(\tau)=\idealComplex(\tau)\cap \totalSpace_{d}(\tau)$ and 
 \begin{equation*}
 \idealComplex(\tau)=
 \bigl\langle u_{1}-\widetilde{\phi}_{12}(u_{1}),~v_{1}-\widetilde{\phi}_{12}(v_{1})\bigr\rangle
 + 
 I_{\sigma_{1}}(\tau)^{r+1}\cdot\ambRing(\tau)+I_{\sigma_{2}}(\tau)^{r+1}\cdot\ambRing(\tau),
 \end{equation*}
 for a lift $\widetilde{\phi}_{12}$ is of the $G^r$-algebraic transition map $\phi_{12}$ from $\sigma_1$ to $\sigma_2$.
 (See \eqref{eq:idealtau} and \eqref{eq:idealJ} for the general definition of $\idealComplex(\tau)$ and $\subSpace_{d}(\tau)$, respectively.)
 Notice that to obtain a basis for $\subSpace_{d}(\tau)$, one needs to compute the kernel of a natural map of vector spaces
 \begin{equation*}
 \totalSpace_{d}(\tau)\oplus\idealComplex(\tau)\to\ambRing(\tau).
 \end{equation*} 
 \begin{lemma}\label{lemma:Qd_tau_two_patches}
 Let $\tau=\sigma_1\cap\sigma_2\in\Delta_1$ be an edge shared by two faces $\sigma_{1},\sigma_{2}\in\Delta_2$, and let 
 $\phi_{12}$ is a $G^{1}$-algebraic transition map on $\Delta$ given by symmetric gluing data $[\gluea,\glueb]$ as in \eqref{eq:gluing_data_symmetric}. 
 If $d_{\gluea}=\deg \gluea\geqslant 1$, then
 \begin{enumerate}[(a)]
 \item\label{casea} in the total degree case, if $d\geqslant d_{\gluea}+1$, we have
 \begin{equation*}
 \dim\quotientSpaceOne_{d}(\tau)=2d+d_{\gluea}+1, \text{ and } \dim \subSpaceOne_{d}(\tau)=d^{2}+d-d_{\gluea};
 \end{equation*}
 a basis of $\subSpaceOne_{d}(\tau)$ is given by
 \begin{equation}\label{eq:a_basis_of_Id_tau_totaldegree}
 \begin{cases}
  u_{1}^{i}v_{1}^{j},\; u_{2}^{j}v_{2}^{i}, &\mbox{for}~2\leqslant i\leqslant d~\mbox{and}~i+ j\leqslant d,\\
  u_{1}v_{1}^{i}+u_{2}^{i}v_{2},&\mbox{for}~0\leqslant i\leqslant d-1,\\
  v_{1}^{i}-\bigl(u_{2}^{i}+i\, u_{2}^{i-1}v_{2}\,\gluea(u_{2})\bigr), &\mbox{for}~1\leqslant i \leqslant d-d_{\gluea}.
 \end{cases}
 \end{equation}
 \item\label{caseb} For bidegree $(d,d)$ with $d\geqslant d_{\gluea}$, we have
 \begin{equation*}
 \dim\quotientSpaceOne_{(d,d)}(\tau)=2d+d_{\gluea}+1, \text{ and } \dim \subSpaceOne_{(d,d)}(\tau)=2d^{2}+2d-d_{\gluea};
 \end{equation*}
 a basis of $\subSpaceOne_{(d,d)}(\tau)$ is given by
 \begin{equation}\label{eq:bidbasisJtau}
 \begin{cases}
  u_{1}^{i}v_{1}^{j},\; u_{2}^{j}v_{2}^{i},\; u_{1}v_{1}^{j}+u_{2}^{j}v_{2}, &\mbox{for}~2\leqslant i\leqslant d~\mbox{and}~0\leqslant j\leqslant d,\\
  v_{1}^{i}-(u_{2}^{i}+i\, u_{2}^{i-1}v_{2}\gluea(u_{2})),&\mbox{for}~1\leqslant i\leqslant d+1-d_{\gluea}.
 \end{cases}
 \end{equation}
 \end{enumerate}
 \end{lemma}
 \begin{proof}
 Let $\succ$ be the elimination order $u_{1}\succ v_{1}\succ u_{2}\succ v_{2}$.
 If we take a monomial basis $\mcB$ of $\totalSpace_{d}(\tau)$, then Lemma \ref{lemma:a_basis_of_Fd} implies that the space $V=\spanset\{\mathfrak{r}(b)\colon b\in \mcB\}$ is isomorphic to $\quotientSpaceOne_{d}(\tau)$. 
 Note that
 \begin{align*}
 \mathfrak{r}(u_{1}^{i}v_{1}^{j}) & =
 \begin{cases}
 0, &\mbox{if }~i\geqslant 2,\\
 u_{2}^{j}v_{2},&\mbox{if }~i=1,\\
 u_{2}^{j}+j\, u_{2}^{j-1}v_{2}\gluea(u_{2}),&\mbox{if }~i=0,
 \end{cases}\\
 \intertext{and,}
 \mathfrak{r}(u_{2}^{i}v_{2}^{j})&=
 \begin{cases}
 0,& \mbox{for}~j\geqslant 2,\\
 u_{2}^{i}v_{2}^{j},&\mbox{for}~j=0,1.
 \end{cases}
 \end{align*}
 Therefore, by an induction on $d_{\gluea}$, one may verify that either for total degree $d\geqslant d_{\gluea}+1$, or for bidegree $(d,d)$ with $d\geqslant d_{\gluea}$, a basis for $V$ is given by
 \begin{equation*}
 \bigl\{u_{2}^{i}\colon i=0,\dots,d\bigr\}\cup\bigl\{u_{2}^{i}v_{2}\colon i =0,\dots, d-1+d_{\gluea}\bigr\}.
 \end{equation*}
 Hence $\dim V=2d+d_{\gluea}+1$, and this proves the dimension formulas for $\quotientSpaceOne_{d}(\tau)$ and $\quotientSpaceOne_{(d,d)}(\tau)$ in \ref{casea} and \ref{caseb}, respectively.
 Since $\quotientSpaceOne_{d}(\tau)=\totalSpace_{d}(\tau)/\subSpaceOne_{d}(\tau)$, then the formulas for $\dim\subSpaceOne_{d}(\tau)$ in \ref{casea} can be deduced from that of $\dim\quotientSpaceOne_{d}(\tau)$. 
 It is easy to verify that every element in \eqref{eq:a_basis_of_Id_tau_totaldegree} is in $\subSpaceOne_{d}(\tau)$. 
 Because they have distinct leading terms with respect to the monomial order, so they are linearly independent. 
 Hence, they form a basis for $\subSpaceOne_{d}(\tau)$. 
 The same reasoning can be applied to the bidegree $(d,d)$ case.
 \end{proof}
 We still have to deal with the case where $\gluea$ is a constant. 
 If the transition map is given by symmetric gluing data, this is equivalent to saying that the transition map can be lifted to an invertible affine map. 
 The intuition is, if the transition map is affine, then the domain is ``locally planar", in the sense of the following lemma.
 \begin{lemma}\label{lemma:locally_affine_transition_2patches}
 For adjacent faces $\sigma_{1}$ and $\sigma_{2}$, if the $G^{r}$-algebraic transition map $\phi_{12}$ is invertible and can be lifted to an invertible affine map $\widetilde{\phi}_{12}\colon\ambRing(\sigma_{1})\to\ambRing(\sigma_{2})$, then its restriction to $\ambRing_{d}(\sigma_{1})$ is $\widetilde{\phi}_{12}\colon\ambRing_{d}(\sigma_{1})\to\ambRing(\sigma_{2})$ which makes the following diagram commute:
 \begin{equation}\label{eq:commutative_diagram_for_affine_transitions}
 \begin{tikzcd}
 {\mathcal{R}_{d}(\sigma_{1})\oplus\mathcal{R}_{d}(\sigma_{1})} \arrow[d, "\id\oplus\widetilde{\phi}_{12}"] \arrow[r, "\partial_{2}"] & {\mathcal{R}_{d}(\sigma_{1})/I_{\sigma_{1}}(\tau)^{r+1}\cap\mathcal{R}_{d}(\sigma_{1})} \arrow[d, "\eta"] \\
 {\mathcal{R}_{d}(\sigma_{1})\oplus\mathcal{R}_{d}(\sigma_{2})} \arrow[r, "\delta_{2}"] & {\mathcal{Q}^{r}_{d}(\tau)}, 
 \end{tikzcd}
 \end{equation} 
 where $\partial_{2}=\begin{bmatrix}
 1&-1
 \end{bmatrix}$ and $\eta$ is induced by the natural inclusion $\ambRing(\sigma_{1})\to\ambRing(\tau)$.
 Moreover, $\eta$ is an isomorphism and $\ker \partial_{2}\cong\ker \delta_{2}$.
 \end{lemma}
 \begin{proof}
 First of all, since $I_{\sigma_{1}}(\tau)^{r+1}=\idealComplex(\tau)\cap\ambRing(\sigma_{1})$, then $\eta$ is well-defined and injective. 
 To check that $\eta$ is also surjective, note that $\phi_{12}$ is invertible and its inverse $\phi_{21}$ is also affine. 
 If $h(u_{1},v_{1},u_{2},v_{2})\in \totalSpace_{d}(\tau)$, then $h\bigl(u_{1},v_{1},\phi_{21}(u_{2}),\phi_{21}(v_{2})\bigr)\in\ambRing_{d}(\sigma_{1})$ gives a preimage of $h(u_{1},v_{1},u_{2},v_{2})$. Thus, $\eta$ is surjective.
 
 Assume $(f,g)\in\ambRing_{d}(\sigma_{1})\oplus\ambRing_{d}(\sigma_{1})$. Denote by $\overline{f}$ the image of $f$ in the quotient ring $\ambRing(\sigma_{1})/I_{\sigma_{1}}(\tau)^{r+1}$. Then $\delta_{2}\circ(\id\oplus\widetilde{\phi}_{12})(f,g)=\overline{f}-\phi_{12}(\overline{g})=\eta\circ\partial_{2}(f,g)$. Hence, diagram \eqref{eq:commutative_diagram_for_affine_transitions} commutes.
 
 Both $\partial_{2}$ and $\delta_{2}$ in \eqref{eq:commutative_diagram_for_affine_transitions} are surjective and all vertical maps are isomorphisms. Therefore, we conclude that $\ker \partial_{2}=\ker\delta_{2}$.
 \end{proof} 
 A dimension formula for $\quotientSpaceOne_{d}(\tau)$, and hence for $\dim\subSpaceOne_{d}(\tau)$, when $\deg\gluea < 1$ follows as a corollary of Lemma \ref{lemma:locally_affine_transition_2patches}.
 \begin{corollary}\label{Cor:Qd_tau_two_patches_planar}
 Assume the transition map $\phi_{12}$ from $\sigma_1$ to $\sigma_2$ is given by symmetric gluing data $[\gluea,\glueb]$ as in \eqref{eq:gluing_data_symmetric}. 
 If $\gluea$ is constant, and $d\geqslant 0$, then
 \begin{enumerate}[(a)]
 \item\label{caseaconstant} in the total degree case we have,
 \begin{equation*}
 \dim\quotientSpaceOne_{d}(\tau)=2d+1, \text{ and }\dim \subSpaceOne_{d}(\tau)=d^{2}+d.
 \end{equation*}
 A basis of $\subSpaceOne_{d}(\tau)$ is given by
 \begin{equation*}
 \begin{cases}
  v_{1}-u_{2},\\
  u_{1}v_{1}^{i}+u_{2}^{i}v_{2},&\mbox{for}~0\leqslant i\leqslant d-1,\\
  u_{1}^{i}v_{1}^{j},\; u_{2}^{j}v_{2}^{i},\; v_{1}^{i}-u_{2}^{i}, &\mbox{for}~2\leqslant i\leqslant d~\mbox{and}~i+ j \leqslant d,
 \end{cases}
 \end{equation*}
 \item\label{casebconstant} For bidegree $(d,d)$ case, we have
 \begin{equation*}
 \dim\quotientSpaceOne_{(d,d)}(\tau)=2d+2, \text{ and } \dim \subSpaceOne_{(d,d)}(\tau)=2d^{2}+2d-1.
 \end{equation*}
 A basis of $\subSpaceOne_{(d,d)}(\tau)$ is given by
 \begin{equation*}
 \begin{cases}
  v_1-u_2,\\
  u_{1}^{i}v_{1}^{j},\; u_{2}^{j}v_{2}^{i},\; v_{1}^{i}-u_{2}^{i},\; u_{1}v_{1}^{j}+u_{2}^{j}v_{2}, &\mbox{for}~2\leqslant i\leqslant d~\mbox{and}~0\leqslant j\leqslant d,
 \end{cases}
 \end{equation*} 
 \end{enumerate}
 \end{corollary}
 \begin{proof}
 The total degree case is an immediate consequence of Lemma \ref{lemma:locally_affine_transition_2patches}. For the bidegree case, note that for symmetric gluing data, the only case where $\gluea(u)=0$ or $\deg\gluea(u)=0$ is that $(n_{0},n_{1})=(4,4)$, which gives $\phi_{12}(u_{1})=-v_{2}$ and $\phi_{12}(v_{1})=u_{2}$. 
 Hence, $\phi_{12}$ preserves bigrading and Lemma \ref{lemma:locally_affine_transition_2patches} can be applied. This proves the dimension formulas in both cases, \ref{caseaconstant} and \ref{casebconstant}. 
 Since the transition maps are explicit, we can easily verify that the given polynomials constitute a basis for $\subSpaceOne_{d}(\tau)$ and $\subSpaceOne_{(d,d)}(\tau)$ in \ref{caseaconstant} and \ref{casebconstant}, respectively. 
 \end{proof}
 Next, we calculate $\dim\quotientSpaceOne_{d}(\gamma)$ for an interior vertex $\gamma$, when $\Phi$ is given by symmetric gluing data. 
 Let $\sigma_{1},\dots,\sigma_{s}\in\Delta_2$ be the 2-faces with a common vertex $\gamma$, for some $s\geqslant 3$. 
 We assume $\sigma_{i}\cap \sigma_{i+1}=\tau_{i}\in \Delta_1$, for $i=1,\dots,s$, where we identify $\sigma_{s+1}=\sigma_{1}$. 
 
 For each edge $\tau_i$, let $\phi_{i,i+1}\colon\ambRing(\sigma_{i})/I_{\sigma_{i}}(\tau_{i})^{r+1}\to\ambRing(\sigma_{i+1})/I_{\sigma_{i}}(\tau_{i})^{r+1}$ be the $G^r$-algebraic transition map from $\sigma_i$ to $\sigma_{i+1}$, and write $\widetilde{\phi}_{i,i+1}\colon\ambRing(\sigma_{i})\to\ambRing(\sigma_{i+1})$ for a lift of $\phi_{i,i+1}$.
 From \eqref{eq:ringalpha}, we have that in this case
 \[
 \ambRing(\tau_{i})=\RR[u_{i},v_{i},u_{i+1},v_{i+1}],\text{ and } \ambRing(\gamma)=\RR[u_{1},v_{1},\dots,u_{s},v_{s}],\]
 and \eqref{eq:idealtau} gives the ideals 
 \[ 
 \idealComplex(\tau_{i})
 =
 \bigl\langle u_{i}-\widetilde{\phi}_{i,i+1}(u_{i}),~v_{i}-\widetilde{\phi}_{i,i+1}(v_{i})
 \bigr\rangle
 +
 I_{\sigma_{i}}(\tau_{i})^{r+1}\cdot\ambRing(\tau_{i})+I_{\sigma_{i+1}}(\tau_{i})^{r+1}\cdot\ambRing(\tau_{i}),\]
 and 
 \[ \idealComplex(\gamma)
 =
 \sum_{i=1}^{s}
 \bigl\langle u_{i}-\widetilde{\phi}_{i,i+1}(u_{i}),~v_{i}-\widetilde{\phi}_{i,i+1}(v_{i})
 \bigr\rangle
 + 
 \sum_{i=1}^{s}\left(I_{\sigma_{i}}(\tau_{i})^{r+1}+I_{\sigma_{i+1}}(\tau_{i})^{r+1}\right)\cdot\ambRing(\gamma).
 \]
 
 \begin{lemma}\label{lemma:dim_F_gamma_nonsingular}
 Let $\Phi$ be a collection of compatible $G^1$-algebraic transition maps given by symmetric gluing data $[\gluea,\glueb]$ as in \eqref{eq:gluing_data_symmetric}. 
 If $\gamma\in\Delta_{0}^\circ$ is an interior vertex which is common to $s\geqslant 3$ faces $\sigma_i\in\Delta_2$ and $s\neq 4$, then $\idealComplexOne(\gamma)$ is generated by
 \begin{equation}\label{eq:generating_set_Igamma}
 \bigl\{u_{i}+v_{i+1},\; v_{i}-u_{i+1}-a_{0}v_{i+1},\; u_{i}^{2},\; u_{i}v_{i},\; v_{i}^{2}\colon i=1,\dots,s\bigr\},
 \end{equation}
 where $a_{0}$ is the constant term in $\gluea(u)=\sum_{k=0}^{d_{\gluea}}a_{k}u^{k}$. 
 Additionally, if in this case $d\geqslant 1$, then
 \begin{equation}\label{eqn:dim_Fgamma_symm_nonsingular}
 \dim\quotientSpaceOne_{d}(\gamma)= \dim\quotientSpaceOne_{(d,d)}(\gamma)=3.
 \end{equation}
 \end{lemma}
 \begin{proof}
 We first show that every element in \eqref{eq:generating_set_Igamma} is in $\idealComplexOne(\gamma)$. 
 It is clear that $u_{i}^{2}$, $v_{i}^{2}$ and $u_{i}+v_{i+1}$ are in $\idealComplexOne(\gamma)$. To see $u_{i+1}v_{i+1}\in \idealComplexOne(\gamma)$, note that $v_{i}^{2}\in\idealComplexOne(\gamma)$ and that
 \begin{equation*}
 v_{i}^{2}-2a_{0}u_{i+1}v_{i+1}\in 
 \bigl\langle v_{i}-u_{i+1}-v_{i+1}\gluea(u_{i+1}),\; u_{i+1}^{2},~v_{i+1}^{2}
 \bigr\rangle\subseteq\idealComplexOne(\gamma).
 \end{equation*} 
 Moreover, $s\neq 4$ implies $a_{0}\neq 0$. 
 Therefore, $u_{i+1}v_{i+1}\in \idealComplexOne(\gamma)$. 
 Now note that since $u_{i+1}^{2}$, $u_{i+1}v_{i+1}$, and $v_{i+1}^{2}$ are polynomials in $\idealComplexOne(\gamma)$, then
 \begin{equation*}
 \bigl(v_{i}-u_{i+1}-v_{i+1}\gluea(u_{i+1})\bigr)-\bigl(v_{i}-u_{i+1}-a_{0}v_{i+1}\bigr)\in
 \bigl\langle u_{i+1}^{2},~u_{i+1}v_{i+1},~v_{i+1}^{2}
 \bigr\rangle\subseteq\idealComplexOne(\gamma).
 \end{equation*}
 This implies $v_{i}-u_{i+1}-a_{0}v_{i+1}\in\idealComplexOne(\gamma)$.
 
 Next we prove that \eqref{eq:generating_set_Igamma} generates $\idealComplexOne(\gamma)$. 
 Let $I$ be the ideal generated by \eqref{eq:generating_set_Igamma}. 
 Then $I\subseteq \idealComplexOne(\gamma)$. 
 Suppose there exists $f\in\idealComplexOne(\gamma)$ such that $f\not\in I$. Let $\succ$ be the elimination order $u_{1}\succ v_{1}\succ \cdots\succ u_{s}\succ v_{s}$ on the polynomial ring $\ambRing(\gamma)$. Then the leading term of the remainder $\mathfrak{r}(f)$ with respect to $I$ is less than $v_{s}^{2}$, i.e., $\mathfrak{r}(f)=c_{1}u_{s}+c_{2}v_{s}+c_{3}$ for some constants $c_{1}$, $c_{2}$, and $c_{3}$. 
 Hence, $\mathfrak{r}(f)\in \idealComplexOne(\gamma)$, and this implies $c_{1}=c_{2}=c_{3}=0$, which means $f\in I$, which is a contradiction.
 
 Note that for any subspace $W\subseteq\ambRing(\gamma)$ such that $\{u_{s},v_{s},1\}\subseteq W$, the set of remainders $\mathfrak{r}(W)$ with respect to $\idealComplexOne(\gamma)$ is the subspace of $\ambRing(\gamma)$ spanned by $\{u_{s},v_{s},1\}$. 
 Since for $d\geqslant 1$, in the total degree case we have $\{u_{s},v_{s},1\}\subseteq\totalSpace_{d}(\gamma)$, and in the bidegree case $\{u_{s},v_{s},1\}\subseteq\totalSpace_{(d,d)}(\gamma)$, then \eqref{eqn:dim_Fgamma_symm_nonsingular} holds, which concludes the proof.
 \end{proof}
 
 \begin{lemma}\label{lemma:dim_F_gamma_singular}
 Let $\Phi$ be a collection of compatible $G^1$-algebraic transition maps given by
 gluing data $[\gluea,\glueb]$ as in \eqref{eq:gluing_data_symmetric}. 
 If $\gamma\in\Delta_{0}^\circ$ is an interior vertex which is common to 4 faces $\sigma_i\in\Delta_2$, then $\idealComplexOne(\gamma)$ is generated by
 \begin{equation}\label{eq:generating_set_Igamma_singular}
 \bigl\{u_{i}+v_{i+1},\; v_{i}-u_{i+1},\; u_{i}^{2},\; v_{i}^{2}\colon i=1,2,3,4\bigr\}.
 \end{equation}
 Additionally, if $d\geqslant 1$ then $\dim\quotientSpaceOne_{(d,d)}(\gamma)=4$, and if $d\geqslant 2$ then $\dim\quotientSpaceOne_{d}(\gamma)=4$. 
 \end{lemma}
 \begin{proof}
 It is clear that $u_{i}^{2},\; v_{i}^{2},\; u_{i}+v_{i+1}\in\idealComplexOne(\gamma)$. 
 Not that if the number of faces around $\gamma$ is 4, then $\gluea(u_{i})=a_{2}u_{t}^2$. 
 Thus $v_{i}-u_{i+1} = v_{i}-u_{i+1}-v_{i+1}\gluea(u_{i+1})-a_{2}v_{i+1}u_{i+1}^{2}\in\idealComplexOne(\gamma)$. 
 It is easy to verify by an explicit computation that \eqref{eq:generating_set_Igamma_singular} generate $\idealComplexOne(\gamma)$. 
 
 Let $\succ$ be the elimination order $u_{1}\succ v_{1}\succ \dots\succ u_{4}\succ v_{4}$ on the polynomial ring $\ambRing(\gamma)$. 
 Note that for any subspace $W\subseteq\ambRing(\gamma)$ such that $\{u_{4}v_{4},u_{4},v_{4},1\}\subseteq W$, the set $\mathfrak{r}(W)$ of reminders of $W$ with respect to $\idealComplexOne(\gamma)$ is the subspace of $\ambRing(\gamma)$ spanned by $\bigl\{u_{4}v_{4},u_{4},v_{4},1\bigr\}$. 
 Since $\bigl\{u_{4}v_{4},u_{4},v_{4},1\bigr\}\subseteq\totalSpace_{d}(\gamma)$ for total degree $d\geqslant 2$ and $\{u_{4}v_{4},u_{4},v_{4},1\}\subseteq\totalSpace_{(d,d)}(\gamma)$ for bidegree $d\geqslant 1$, then the dimension of $\quotientSpaceOne_{d}(\gamma)$ and $\quotientSpaceOne_{(d,d)}(\gamma)$ is 4 for the correspondent values of $d$, which proves the statement.
 \end{proof}
 
 \section{Explicit computations of the dimension of $G^r$-spline spaces}\label{section:explicit_computations}
 Following the notation in Section \ref{section:spline_complex_dim2}, in this section we also assume $\Delta$ is a $2$-dimensional cell complex. 
 We take a collection $\Phi$ of compatible $G^r$-algebraic transition maps on $\Delta$, for some $r\geqslant 0$. 
 We will focus on the the homology terms $H_{1}(\quotientSpace_{d,\bullet})$ and $H_{0}(\quotientSpace_{d,\bullet})$, which appear in the dimension formula for $\dim G^{r}_{d}(\Delta,\Phi)$ given in \eqref{eq:dim2cellcase}. 
 We compute $\dim H_{1}(\quotientSpace_{d,\bullet})$ and $\dim H_{0}(\quotientSpace_{d,\bullet})$ in two particular cases of $G^1$-domains $(\Delta,\Phi)$, namely when $\Delta$ consists of two faces in Section \ref{section:two_patches} and when $\Delta$ is a star of vertex in Section \ref{sec:star}.
 The dimension formulas for the homology terms together with the Euler characteristic $\chi(\quotientSpace_{d,\bullet})$ computed in Section \ref{section:spline_complex_dim2} lead us to to explicit dimension formulas of $G^1$-spline spaces in these two cases in Proposition \ref{prop:Id_tau_two_patches} and Theorem \ref{theor:dimG1}, respectively. 
 
 First, notice that when $\Delta$ is a $2$-dimensional cell complex, the following diagram commutes
 \begin{equation}\label{eq:complex2D}
 \begin{tikzcd}[cramped]
 {\subSpace_{d,\bullet}\colon} & 0 \arrow[r] \arrow[d] & \bigoplus_{\tau\in\Delta_{1}^{\circ}}\subSpace_{d}(\tau) \arrow[d] \arrow[r, "\delta_{1}"] & \bigoplus_{\gamma\in\Delta_{0}^{\circ}}\subSpace_{d}(\gamma) \arrow[d] \\
 {\totalSpace_{d,\bullet}\colon} & \bigoplus_{\sigma\in\Delta_{2}}\totalSpace_{d}(\sigma) \arrow[r, "\delta_{2}"] \arrow[d] & \bigoplus_{\sigma\supseteq\tau,\, \tau\in\Delta_{1}^{\circ}}\totalSpace_{d}(\sigma) \arrow[r, "\delta_{1}"] \arrow[d] & \bigoplus_{\sigma\supseteq\gamma,\,\gamma\in\Delta_{0}^{\circ}}\totalSpace_{d}(\sigma) \arrow[d] \\
 {\quotientSpace_{d,\bullet}\colon} & \bigoplus_{\sigma\in\Delta_{2}}\quotientSpace_{d}(\sigma) \arrow[r, "\delta_{2}"] & \bigoplus_{\tau\in\Delta_{1}^{\circ}}\quotientSpace_{d}(\tau) \arrow[r, "\delta_{1}"] & \bigoplus_{\gamma\in\Delta_{0}^{\circ}}\quotientSpace_{d}(\gamma), 
 \end{tikzcd}
 \end{equation}
 where $\delta_{1}\colon \bigoplus_{\tau\in\Delta_{1}^{\circ}}\ambRing(\sigma)\to\bigoplus_{\gamma\in\Delta_{0}^{\circ}}\ambRing(\gamma)$ and $\delta_{2}\colon \bigoplus_{\sigma\in\Delta_{2}}\ambRing(\sigma)\to\bigoplus_{\tau\in\Delta_{1}^{\circ}}\ambRing(\tau)$ are the differential maps defined in \eqref{eq:delta}.
 By an abuse of notation, we write $\delta_{i}$ for its restriction to the subspaces in $\totalSpace_{d,\bullet}$ and $\subSpace_{d,\bullet}$, and for the map it induces on the quotient spaces $\quotientSpace_{d,\bullet}$. 
 
 The short exact sequence of chain complexes of vector spaces
 \begin{equation*}
 0\to \subSpace_{d,\bullet}\to\totalSpace_{d,\bullet}\to\quotientSpace_{d,\bullet}\to 0
 \end{equation*}
 leads us to the long exact sequence
 \begin{multline}\label{eq:LES_of_Homologies}
 0\to H_{2}(\subSpace_{d,\bullet})\to H_{2}(\totalSpace_{d,\bullet})\to H_{2}(\quotientSpace_{d,\bullet})
 \to H_{1}(\subSpace_{d,\bullet})\to H_{1}(\totalSpace_{d,\bullet})\to\\ H_{1}(\quotientSpace_{d,\bullet})
 \to H_{0}(\subSpace_{d,\bullet})\to H_{0}(\totalSpace_{d,\bullet})\to H_{0}(\quotientSpace_{d,\bullet})\to 0.
 \end{multline}
 Thus, we can study $\dim H_{1}(\quotientSpace_{d,\bullet})$ and $\dim H_{0}(\quotientSpace_{d,\bullet})$ by analyzing the terms in the sequence \eqref{eq:LES_of_Homologies}. 
 We start with the homologies of $\totalSpace_{d,\bullet}$.
 
 \begin{remark} Notice that the commutativity of the diagram \eqref{eq:complex2D} and the exactness of the long sequence \eqref{eq:LES_of_Homologies} hold if we replace the chain complexes $\totalSpace_{d,\bullet}$, $\subSpace_{d,\bullet}$, $\quotientSpace_{d,\bullet}$ by their bidegree counterparts $\totalSpace_{(d,d),\bullet}$, $\subSpace_{(d,d),\bullet}$, and $\quotientSpace_{(d,d),\bullet}$, for any $d\geqslant 0$.
 To simplify the exposition, we only introduce the notation in terms of the total degree case $G^r_d(\Delta,\Phi)$, but we apply the construction and give explicit formulas for each case $G^r_d(\Delta,\Phi)$ and $G^r_{(d,d)}(\Delta,\Phi)$ in the results and examples below.
 \end{remark}
 
 For an integer $k\geqslant 0$ and a 2-face $\sigma\in\Delta_2$, we denote by $\ambRing(\sigma)_{k}$ the set of homogeneous polynomials in $\ambRing(\sigma)$ of degree $k$.
 We write $(\ambRing_{\bullet})_{k}$ for the chain complex of the $k$-graded pieces of $\ambRing_{\bullet}$, and take $(\totalSpace_{d,\bullet})_{k}=\totalSpace_{d,\bullet}\cap(\ambRing_{\bullet})_{k}$. 
 Using this notation we can rewrite $\totalSpace_{d,\bullet}=\bigoplus_{0\leqslant k\leqslant d}(\totalSpace_{d,\bullet})_{k}$.
 
 \begin{lemma}\label{lemma:HomologyOfambRing}
 If $\Delta$ is a 2-dimensional polyhedral complex with only one connected component and the number of 2-faces of $\Delta$ is strictly greater than
 $1$, then 
 \begin{equation*}
 H_{i}(\totalSpace_{d,\bullet})_{0}= H_{i}\bigl(\Delta,\partial\Delta;\RR\bigr), 
 \end{equation*}
 for $i=0,1,2$. 
 Moreover, for any $k\geqslant 1$, we have 
 \[
 H_{0}(\totalSpace_{d,\bullet})_{k}
 =\bigoplus_{\sigma\in\Delta_{2}'}\ambRing(\sigma)_k,\quad
 H_{1}(\totalSpace_{d,\bullet})_{k}
 =\bigoplus_{\sigma\in\Delta_{2}\setminus\Delta_{2}'}\ambRing(\sigma)_k^{m(\sigma)-1},\,\text{and\quad} H_{2}(\totalSpace_{d,\bullet})_{k}
 =0,
 \]
 where $\Delta'=\bigl\{\sigma\in\Delta_2\colon \sigma\subseteq\Delta\setminus\partial\Delta\bigr\}$, and $m(\sigma)$ is the number of connected components of $\partial\Delta\cap\sigma$ for any $\sigma\in\Delta_2$.
 (We always assume that a 2-face $\sigma\in\Delta$ contains its edges and vertices.)
 In the bidegree case, the corresponding formulas for $H_i\bigl(\totalSpace_{(d,d),\bullet}\bigr)_{k}$ hold.
 \end{lemma}
 \begin{proof}
 We only prove the formulas in the total degree case, the bidegree case follows by a similar reasoning.
 It is clear that
 \begin{equation*}
 (\totalSpace_{d,\bullet})_{0}\colon\bigoplus_{\sigma\in\Delta_{2}}\RR\xrightarrow{\delta_{2}}\bigoplus_{\tau\in\Delta_{1}^{\circ}}\RR\xrightarrow{\delta_{1}}\bigoplus_{\gamma\in\Delta_{0}^{\circ}}\RR,
 \end{equation*}
 and hence $H_{i}(\totalSpace_{d,\bullet})_{0}= H_{i}(\Delta,\partial\Delta;\RR)$, for $i=0,1,2$.
 
 If $k\geqslant 1$, then we can rewrite $\mathcal{T}_{d,\bullet}$ as a direct sum $(\totalSpace_{d}(\tau))_{k}=\bigoplus_{\sigma\supseteq\tau}\ambRing(\sigma)_{k}$ for every edge $\tau\in\Delta_1^\circ$, and $(\totalSpace_{d}(\gamma))_{k}=\bigoplus_{\sigma\supseteq\gamma}\ambRing(\sigma)_{k}$ for every vertex $\gamma\in\Delta_0^\circ$. 
 Thus, if $(\delta_{i})_{k}$ denotes the boundary map in the chain complex $ \bigl(\totalSpace_{d,\bullet}\bigr)_k$ then, we can rewrite it as the direct sum
 \[
 (\delta_{i})_{k}=\bigoplus_{\sigma\in\Delta_{2}}\bigl(\delta_{i}(\sigma)\bigr)_{k},
 \]
 where for each 2-face $\sigma\in\Delta_2$, the map $(\delta_{i}(\sigma))_{k}$ is the boundary map in the chain complex $\ambRing_{\bullet}(\sigma)_{k}$ given by
 \begin{equation*}
 \ambRing_{\bullet}(\sigma)_{k}
 \colon
 \ambRing(\sigma)_{k}\xrightarrow{\delta_{2}(\sigma)_{k}}\bigoplus_{\tau\subseteq\sigma\; \tau\in\Delta_{1}^{\circ}}\ambRing(\sigma)_{k}\xrightarrow{\delta_{1}(\sigma)_{k}}\bigoplus_{\gamma\subseteq\sigma,\; \gamma\in\Delta_{0}^{\circ}}\ambRing(\sigma)_{k}.
 \end{equation*}
 Therefore, we can write the chain complex $(\totalSpace_{d,\bullet})_{k}= \bigoplus_{\sigma\in\Delta_{2}}\ambRing_{\bullet}(\sigma)_{k}$.
 Since
 \begin{equation*}
 \ambRing_{\bullet}(\sigma)_{k}=\ambRing(\sigma)_{k}\otimes C_{\bullet}(\sigma,\partial\Delta\cap \sigma),
 \end{equation*}
 where $C_{\bullet}(\sigma,\partial\Delta\cap \sigma)$ is the quotient complex $C_{\bullet}(\sigma)/C_{\bullet}(\partial\Delta\cap \sigma)$ of the chain complexes of $\sigma$ and $\partial\Delta\cap\sigma$, respectively.
 In particular, $H_{i}(\totalSpace_{d,\bullet})_{k}
 =
 \bigoplus_{\sigma\in\Delta_{2}}\ambRing(\sigma)_{k}\otimes H_{i}(\sigma,\partial\Delta\cap \sigma)$.
 Thus, to compute $H_{i}(\totalSpace_{d,\bullet})_{k}$, we first consider $H_{i}(\sigma,\partial\Delta\cap \sigma)$.
 Notice that the sequence $0\to C_\bullet(\partial\Delta\cap\sigma)\to C_{\bullet}(\sigma)\to C_{\bullet}(\sigma,\partial\Delta\cap\sigma)\to 0$ yields the long exact sequence of homologies
 \begin{multline*}
 0\to H_{2}(\partial\Delta\cap \sigma)\to H_{2}(\sigma)\to H_{2}(\sigma,\partial\Delta\cap \sigma)
 \to H_{1}(\partial\Delta\cap \sigma)\to\\
 H_{1}(\sigma)\to H_{1}(\sigma,\partial\Delta\cap \sigma)
 \to H_{0}(\partial\Delta\cap \sigma)\to H_{0}(\sigma)\to H_{0}(\sigma,\partial\Delta\cap \sigma)\to 0\,.
 \end{multline*}
 By hypothesis $\sigma\neq\Delta$, so
 \begin{align*}
 H_{2}(\sigma)&=H_{1}(\sigma)=0,&H_{0}(\sigma)&=\ZZ,\\
 H_{2}(\partial\Delta\cap \sigma)&=H_{1}(\partial\Delta\cap \sigma)=0,&H_{0}(\partial\Delta\cap \sigma)&=\ZZ^{m(\sigma)}.
 \end{align*}
 From this we see that $H_{2}(\sigma,\partial\Delta\cap \sigma)=0$, which implies $H_{2}(\totalSpace_{d,\bullet})_{k}=0$. 
 Furthermore, we can rewrite the last part of the sequence as
 \begin{equation*}
 0\to H_{1}(\sigma,\partial\Delta\cap \sigma)\to \ZZ^{m(\sigma)}\to\ZZ\to H_{0}(\sigma,\partial\Delta\cap \sigma)\to 0.
 \end{equation*}
 Note that $\sigma\in\Delta_{2}'$ if and only if $m(\sigma)=0$, in which case $H_{1}(\sigma,\partial\Delta\cap \sigma)=0$ and $H_{0}(\sigma,\partial\Delta\cap \sigma)=\ZZ$.
 On the other hand, if $\sigma\notin\Delta'$, or equivalently if $m(\sigma)\geqslant 1$, then the map $\ZZ^{m(\sigma)}\to\ZZ$ (defined by multiplication with $(1,\dots,1)\in \ZZ^{m(\sigma)}$) is surjective. 
 Hence, for $\sigma\notin\Delta'$ we have $H_{1}(\sigma,\partial\Delta\cap \sigma)=\ZZ^{m(\sigma)-1}$ and $H_{0}(\sigma,\partial\Delta\cap \sigma)=0$. 
 Putting these two cases together leads to the formulas for $H_i(\totalSpace_{d,\bullet})_k$ for $i=0,1$, which concludes the proof.
 \end{proof}
 \subsection{$G^r$-splines on a cell complex composed by two patches}\label{section:two_patches}
 If $\Delta$ is a 2-dimensional cell complex, and $\tau\in\Delta_1^\circ$ is an interior edge of $\Delta$, then we define $\Delta_\tau\subseteq \Delta$ as the subcomplex composed by the two $2$-faces $\sigma_{1}, \sigma_{2}\in \Delta$ containing $\tau$. As before, we write $\phi_{12}\colon\ambRing(\sigma_{1})/I_{\sigma_{1}}(\tau)^{r+1}\to\ambRing(\sigma_{2})/I_{\sigma_{2}}(\tau)^{r+1}$ for the algebraic transition map of $\tau$ from $\sigma_1$ to $\sigma_2$.
 The cell complex and $\phi_{12}$ define a $2$-dimensional $G^{r}$-domain, which we denote by $\bigl(\Delta_\tau,\Phi\bigr)$. 
 Since $\Delta_\tau$ does not have interior vertices, the term associated to the $0$-faces in \eqref{eq:complex2D} is zero i.e., 
 \begin{align*}
 \totalSpace_{d,0}=\subSpace_{d,0}=\quotientSpace_{d,0}=0,
 \end{align*}
 $H_{0}(\totalSpace_{d,\bullet})=0$.
 By Lemma \ref{lemma:HomologyOfambRing}, we have $H_{1}(\totalSpace_{d,\bullet})_{0}=H_{1}\bigl(\Delta_\tau,\partial\Delta_\tau;\RR\bigr)=0$, and since the number of connected components $m(\sigma_i)$ of $\partial\Delta_\tau\cap\sigma_i$ is 1 for $i=1,2$, then 
 $H_{1}(\totalSpace_{d,\bullet})_{k}=0$ for every degree $k\geqslant 0$. 
 Thus, 
\begin{equation*}
 H_{1}(\totalSpace_{d,\bullet})=\bigoplus_{k=0}^{\infty}H_{1}(\totalSpace_{d,\bullet})_{k}=0, 
\end{equation*}
and the long exact sequence (\ref{eq:LES_of_Homologies}) reduces to
\begin{align*}
 H_{1}(\quotientSpace_{d,\bullet})=H_{0}(\subSpace_{d,\bullet})=
 H_{0}(\quotientSpace_{d,\bullet})=0.
\end{align*}
Therefore the dimension formula \eqref{eq:dim2cellcase} of the $G^r$-spline space on $\Delta_\tau$ reduces to
\begin{equation}\label{eq:dim_G_Mtau_by_chain_complex}
 \dim G^{r}_{d}\bigl(\Delta_\tau,\Phi\bigr)=\dim \totalSpace_{d}(\sigma_{1})+\dim \totalSpace_{d}(\sigma_{2})-\dim \quotientSpace_{d}(\tau).
\end{equation}
Notice that a similar reasoning yields the corresponding formulas in the bidegree case. 
\begin{proposition}\label{prop:Id_tau_two_patches}
 Let $\Delta_\tau$ be as above, and suppose the algebraic transition map $\phi_{12}$ is given by symmetric gluing data $[\gluea,\glueb]$ as in \eqref{eq:gluing_data_symmetric}. 
 If $d_{\gluea}=\deg \gluea \geqslant 1$ then
 \begin{align*}
 \dim G^{1}_{d}\bigl(\Delta_\tau,\Phi\bigr)
 &=
 d^{2}+d+1-d_{\gluea},\quad \text{ if $d\geqslant d_{\gluea}+1$, }
 \intertext{and}
 \dim G^{1}_{(d,d)}\bigl(\Delta_\tau,\Phi\bigr)&=2d^{2}+2d+1-d_{\gluea}, \quad \text{if $d\geqslant d_{\gluea}$}.
 \end{align*}
 If $\gluea$ is constant, and $d\geqslant 0$, then
 \begin{equation*}
 \dim G^{1}_{d}\bigl(\Delta_\tau,\Phi\bigr)=d^{2}+d+1, \; 
 \text{and\, }
 \dim G^{1}_{(d,d)}\bigl(\Delta_\tau,\Phi\bigr)=2d^{2}+2d.
 \end{equation*}
\end{proposition}
\begin{proof}
 The case $\deg \gluea\geqslant 1$ follows by replacing $\dim \quotientSpaceOne_{d}(\tau)$ obtained in Lemma \ref{lemma:Qd_tau_two_patches} in the dimension formula for $G_d^1(\Delta_\tau,\Phi)$ in \eqref{eq:dim_G_Mtau_by_chain_complex}. 
 If $\gluea$ is constant, the dimension formulas follow from Corollary \ref{Cor:Qd_tau_two_patches_planar}. 
\end{proof}

\begin{remark}
 It is worth mentioning that the dimension formulas for $G^1_{(d,d)}(\Delta_\tau, \Phi)$ in Proposition \ref{prop:Id_tau_two_patches} coincide with those in \cite{mourrain2016dimension}, where it was proved that
 \begin{equation*}
 \dim G^{1}_{(d,d)}\bigl(\Delta_\tau,\Phi\bigr)=2d(d+1)-\max\{\deg\glueb,\deg\gluea-1\}.
 \end{equation*}
 In this paper, we only consider the case of symmetric gluing data $[\gluea,\glueb]$, where $\glueb$ is a constant polynomial as in Example \ref{section:symmetric_gluing_data}.
\end{remark}

\begin{example}[$G^1$-splines on $\Delta_\tau$]\label{Eg:two_patches_eg1} 
 Let $\Delta$ be a 2-dimensional cell complex, and take $\Delta_\tau\subseteq \Delta$ a subcomplex composed by only two adjacent 2-faces $\sigma_1,\sigma_2\in\Delta$ as above.
 Suppose the edge $\tau=\sigma_1\cap\sigma_2$ has vertices $\gamma,\gamma'\in\Delta$ with valencies $s=3$ and $s=4$, respectively. 
 We define an algebraic transition map $\phi_{12}$ from the symmetric gluing data $[\gluea,\glueb]$ associated to the vertex valences at $\tau$ following the construction \eqref{eq:gluing_data_symmetric} in Example \ref{section:symmetric_gluing_data}.
 We get
 \begin{align*} 
 \phi_{12}\colon\RR[u_{1},v_{1}]/\langle u_{1}^{2}\rangle&\to\RR[u_{2},v_{2}]/\langle v_{2}^{2}\rangle
 \\
 u_{1}&\mapsto -v_{2},\\
 v_{1}&\mapsto u_{2}+v_{2}(-u_{2}^{2}+2u_{2}-1).
 \end{align*}
 Therefore, $
 \idealComplexOne(\tau)=\bigl\langle u_{1}+ v_{2},~v_{1}- u_{2}+v_{2}(u_{2}^{2}-2u_{2}+1),~u_{1}^{2},~v_{2}^{2}\bigr\rangle.
 $
 Hence, we take $\gluea(u_2)=-u_2^2+2u_2-1$, and in particular $d_{\gluea}=2$. 
 By Lemma \ref{lemma:Qd_tau_two_patches}\ref{casea}, replacing $\gluea(u_2)$
 in \eqref{eq:a_basis_of_Id_tau_totaldegree} we get a basis for $\subSpaceOne_{d}(\tau)$, and
 for $d\geqslant 3$, the dimension of the quotient space is $\dim \quotientSpaceOne_{d}(\tau)=2d+3$.
 If we replace $\quotientSpaceOne_{d}(\tau)$ in \eqref{eq:dim_G_Mtau_by_chain_complex}, we get that for any $d\geqslant 3$, the dimension
 \[
 \dim G^{1}_{d}\bigl(\Delta_\tau,\Phi\bigr)=2\binom{d+2}{2}-(2d+3)=d^{2}+d-1.
 \]
 In the bidegree $(d,d)$ case, by Lemma \ref{lemma:Qd_tau_two_patches}\ref{caseb}, a basis of $\subSpaceOne_{(d,d)}(\tau)$ can be obtained by replacing $\gluea(u_2)$ into the polynomials in \eqref{eq:bidbasisJtau}, and if $d\geqslant 2$ we have $\dim \quotientSpaceOne_{(d,d)}(\tau)=2d+3$. 
 Hence, the dimension formula for the $G^{1}$-spline space is 
 \begin{align*}
 \dim G^{1}_{(d,d)}\bigl(\Delta_\tau,\Phi\bigr)
 &=\dim \totalSpace_{(d,d)}(\sigma_{1})+\dim \totalSpace_{(d,d)}(\sigma_{2})-\dim \quotientSpaceOne_{(d,d)}(\tau)\\
 &=2(d+1)^{2}-(2d+3)=2d^{2}+2d-1,
 \end{align*}
 whenever $d\geqslant 2$.\hfill$\diamond$
\end{example}
\subsection{$G^r$-splines on the star of a vertex}\label{sec:star}
If $\Delta$ is a 2-dimensional cell complex, and $\gamma\in\Delta_0^\circ$ is an interior vertex of $\Delta$, then we define $\Delta_\gamma\subseteq \Delta$ as the subcomplex composed by all the faces of $\Delta$ containing $\gamma$. Recall that we always assume that a face of $\Delta$ contains all its faces.
We assume that the number of 2-faces $\sigma_t$ in $\Delta_\gamma$ is $s\geqslant 3$, and 
that $\sigma_t\cap\sigma_{t+1}=\tau_t\in\bigl(\Delta_\gamma\bigr)_1^\circ$ is an edge for every $t=1,\dots,s$, where we take $\sigma_{s+1}=\sigma_1$. 
\begin{lemma}\label{lemma:lemma_5_5}
 If $\bigl(\Delta_\gamma,\Phi\bigr)$ is a $G^r$-domain for some $r\geqslant 0$, then the homology terms of the chain complexes $\subSpace_{d,\bullet}$ and $\quotientSpace_{{d},\bullet}$ of $\bigl(\Delta_{\gamma},\Phi\bigr)$ in \eqref{eq:complex2D} satisfy $H_{0}(\quotientSpace_{d,\bullet})=0$ and $H_{1}(\quotientSpace_{d,\bullet})=H_{0}(\subSpace_{d,\bullet})$, for both total degree and bidegree case, for any degree $d\geqslant 0$.
\end{lemma} 
\begin{proof}
 Since $\Delta_\gamma$ has only one component, then $\Delta_\gamma'=\Delta_\gamma\setminus\partial\Delta_\gamma=\emptyset$ and by Lemma \ref{lemma:HomologyOfambRing}, we have that 
 $ H_{0}(\totalSpace_{d,\bullet})=H_{1}(\totalSpace_{d,\bullet})=0$ and $H_{2}(\totalSpace_{d,\bullet})\cong\RR$.
 Then, by the long exact sequence of homologies \eqref{eq:LES_of_Homologies} we deduce that 
 $H_{0}(\quotientSpace_{d,\bullet})=0$, and $H_{1}(\quotientSpace_{d,\bullet})\cong H_{0}(\subSpace_{d,\bullet})$.
\end{proof}
Recall that taking $r=1$ in \eqref{eq:idealJ}, we have that $\subSpaceOne_d(\gamma)=\totalSpace_d(\gamma)\cap \mathcal{I}^1(\gamma)$, where 
$\mathcal{T}_d(\gamma)=\bigl\{\sum_{t=1}^s f_{t}\colon f_t=f_{\sigma_t}, \;\text{for }\sigma_i\in\bigl(\Delta_\gamma\bigr)_2\bigr\}$ as defined in \eqref{eq:ringT}.
\begin{lemma}\label{lemma:H0_Igamma_is_0_symm_gluing}
 Following the notation above, let $(\Delta_{\gamma},\Phi)$ be a $G^1$-domain, where $\Phi$ is a collection of $G^1$-algebraic transition maps given by symmetric gluing data $[\gluea,\glueb]$ (as in \eqref{eq:gluing_data_symmetric} in Example \ref{section:symmetric_gluing_data}). 
 If $\subSpaceOne_{d,\bullet}$ and $\subSpaceOne_{(d,d),\bullet}$ are the chain complexes of $\Delta_\gamma$ in total and bidegree, respectively, then $H_{0}\bigl(\subSpaceOne_{d,\bullet}\bigr)=0$ for $d\geqslant 4$, and $H_{0}\bigl(\subSpaceOne_{(d,d),\bullet}\bigr)=0$ for $d\geqslant 3$.
\end{lemma}
\begin{proof}
 Taking $r=1$ and $\Delta_\gamma$ in \eqref{eq:complex2D}, we have that $H_{0}(\subSpaceOne_{d,\bullet})=\subSpaceOne_d(\gamma)/\img(\delta_1)$.
 If $s$ is the number of edges $\tau_i$ containing $\gamma$ as above, then we need to show that $\delta_{1}$ is surjective, where
 \[
 \delta_{1}\colon\bigoplus_{t=1}^{s}\subSpaceOne_{d}(\tau_{t})\to \subSpaceOne_{d}(\gamma).
 \]
 From Lemma \ref{lemma:Qd_tau_two_patches}\ref{casea} we know that if $d_{\gluea}=\deg \gluea\geq 1 $ then
 \begin{subnumcases}{\subSpaceOne_{d}(\tau_{t})=\spanset}
 u_{t}^{i}v_{t}^{j},\; u_{t+1}^{j}v_{t+1}^{i}, & for~$2\leqslant i\leqslant d~\mbox{and}~i+ j\leqslant d,$\nonumber\\
 u_{t}v_{t}^{i}+u_{t+1}^{i}v_{t+1},&for~$0\leqslant i\leqslant d-1,$\label{eq:genJtau2}\\
 v_{t}^{i}-\bigl(u_{t+1}^{i}+i\, u_{t+1}^{i-1}v_{t+1}\,\gluea(u_{t+1})\bigr), &for~$1\leqslant i \leqslant d-d_{\gluea}$.\hfill\label{eq:genJtau3}
 \end{subnumcases}
 On the other hand, $\subSpaceOne_d(\gamma)=\totalSpace_d(\gamma)\cap \mathcal{I}^1(\gamma)$ and we have two cases depending on the number $s$ of edges at $\gamma$. 
 If $s\neq 4$, by Lemma \ref{lemma:dim_F_gamma_nonsingular} we have that
 \[
 \idealComplexOne(\gamma)
 =
 \spanset\bigl\{u_{t}+v_{t+1},\; v_{t}-u_{t+1}-a_{0}v_{t+1},\; u_{t}^{2},\; u_{t}v_{t},\; v_{t}^{2}\colon t=1,\dots,s\bigr\},
 \]
 where $a_{0}$ is the constant term in $\gluea(u)=\sum_{k=0}^{d_{\gluea}}a_{k}u^{k}$.
 If $s=4$, Lemma \ref{lemma:dim_F_gamma_singular} implies
 \[
 \idealComplexOne(\gamma)=\spanset\bigl\{u_{t}+v_{t+1},\; v_{t}-u_{t+1},\; u_{t}^{2},\; v_{t}^{2}\colon t=1,2,3,4\bigr\}
 .\]
 Notice the generators of $\idealComplexOne(\gamma)$ are polynomials in $\totalSpace_d(\gamma)$, hence they generate $\subSpaceOne_d(\gamma)$, and it is enough to check that they all are in $\img(\delta_1)$.
 
 Notice that $u_{t}^{i}v_{t}^{j}\in \img(\delta_{1})$, for every $t=1,\dots, s$ whenever $i\geqslant 2$ or $j\geqslant 2$; 
 taking $i=0$ in \eqref{eq:genJtau2}, we see that $u_t+v_{t+1}\in\img(\delta_1)$.
 
 By hypothesis $d\geqslant 4\geqslant d_{\gluea}+2$, hence $d-d_{\gluea}\geqslant 2$. 
 Thus, if $s\neq 4$, we have that $v_{t}^{2}-\bigl(u_{t+1}^{2}+2u_{t+1}v_{t+1}\gluea(u_{t+1})\bigr)\in\img(\delta_{1})$.
 But $v_{t}^{2},~u_{t+1}^{2}\in\img(\delta_{1})$, so $2u_{t+1}v_{t+1}\gluea(u_{t+1})=2u_{t+1}v_{t+1}(a_{0}+a_{1}u_{t+1}+a_{2}u_{t+1}^{2})\in\img(\delta_{1})$. 
 Since $d\geqslant 4$, then $u_{t}^{2}v_{t},~u_{t}^{3}v_{t}\in\img(\delta_{1})$. 
 This implies $u_{t}v_{t}\in\img(\delta_{1})$, and taking $i=1$ in \eqref{eq:genJtau3} we see that also $v_{t}-u_{t+1}-a_{0}v_{t+1}\in\img(\delta_{1})$.
 It remains to verify that $v_{t}-u_{t+1}\in\img(\delta_{1})$ if $s=4$. But in that case $\gluea(u_{t+1})=a_{2}u_{t+1}^{2}$.
 Hence, $u_{t}v_{t}^{2}\in\img(\delta_{1})$ implies $v_{t}-u_{t+1}\in\img(\delta_{1})$. 
 This proves that $H_0\bigl(\subSpaceOne_{d,\bullet}\bigr)=0$. 
 
 By using Lemma \ref{lemma:Qd_tau_two_patches}\ref{caseb} and the corresponding results for the bidegree case $(d,d)$ in Lemma \ref{lemma:dim_F_gamma_nonsingular} and \ref{lemma:dim_F_gamma_singular}, an analogous argument implies $H_0\bigl(\subSpaceOne_{(d,d),\bullet}\bigr)=0$ for any $d\geq 3$.
\end{proof}
\begin{theorem}\label{theor:dimG1}
 For a $G^r$-domain $\bigr(\Delta_{\gamma},\Phi\bigr)$, the dimension of the $G^r$ spline space is given by
 \begin{equation}\label{eq:dimEulerC}
 \dim G^{r}_{d}\bigl(\Delta_{\gamma},\Phi\bigr)=\chi(\quotientSpace_{d,\bullet})+\dim H_{1}(\quotientSpace_{d,\bullet}).
 \end{equation}
 In particular, $\chi(\quotientSpace_{d,\bullet})$ is a lower bound of $\dim G^{r}_{d}((\GrDomain)_{\gamma})$. The corresponding formula follows for $G^{r}_{(d,d)}\bigl(\Delta_{\gamma},\Phi\bigr)$ in the bidegree case.
 If in addition $r=1$ and $\Phi$ is given by symmetric gluing data $[\gluea,\glueb]$, then 
 \begin{align}
 \dim G^{1}_{d}\bigl(\Delta_{\gamma},\Phi\bigr)&=\chi(\quotientSpaceOne_{d,\bullet}),\;\mbox{ for } d\geqslant 4\text{ \; and},\label{eq:dim_formula_G1_Mgamma_symm_gluing}\\
 \dim G^{1}_{(d,d)}\bigl(\Delta_{\gamma},\Phi\bigr)&=\chi(\quotientSpaceOne_{(d,d),\bullet}),\;\mbox{ for } d\geqslant 3. \nonumber 
 \end{align}
 If $d_\gluea=\deg \gluea$, then
 \begin{subnumcases}
 {\chi(\quotientSpaceOne_{d,\bullet})=}
 s\textstyle{\binom{d+2}{2}}-s(2d+d_{\gluea}+1)+3,&~if $s\neq 4,\text{\; and } d\geqslant d_{\gluea}+1$,\label{eq:chi_s_not_4}\\
 2(d^{2}-d-2),&~if $s=4,\; \gluea\neq 0 \text{\;and } d\geqslant 2$,\label{eq:s4}\\
 2(d^{2}-d+2),&~if $s=4,\; \gluea=0, \text{\;and } d\geqslant 2$\label{eq:s4a0},
 \end{subnumcases}
 and 
 \begin{equation*}
 \chi\bigl(\quotientSpaceOne_{(d,d),\bullet}\bigr)=\begin{cases}
 s(d+1)^{2}-s(2d+d_{\gluea}+1)+3,& \mbox{if } s\neq 4 \text{\; and } d\geqslant d_{\gluea},\\
 4(d^{2}-1),& \mbox{if } s=4,\; \gluea\neq 0 \text{\;and } d\geqslant 1,\\
 4d^{2},& \mbox{if } s=4,\; \gluea=0, \text{\;and } d\geqslant 1,
 \end{cases}
 \end{equation*}
 where $s$ is the number of 2-faces in $\Delta_\gamma$.
\end{theorem}
\begin{proof}
 From \eqref{eq:dim2cellcase} we know that $\dim G^{r}_{d}(\GrDomain)=\chi(\quotientSpace_{d,\bullet})+\dim H_{1}(\quotientSpace_{d,\bullet})-\dim H_{0}(\quotientSpace_{d,\bullet})$. By Lemma \ref{lemma:lemma_5_5}, $\dim H_{0}(\quotientSpace_{d,\bullet})=0$. This implies \eqref{eq:dimEulerC}.
 In addition, if $r=1$ and $\Phi$ is given by gluing data, then by Lemma \ref{lemma:H0_Igamma_is_0_symm_gluing} in the total degree case gives $\dim H_{1}(\quotientSpace_{d,\bullet})=0$ for every $d\geqslant 4$, which implies \eqref{eq:dim_formula_G1_Mgamma_symm_gluing}.
 The bidegree case follows similarly.
 
 The explicit formulas for $\chi\bigl(\quotientSpaceOne_{d,\bullet}\bigr)$ follow from \eqref{eq:Euler_char_Qd}. In the case of a vertex star with $s$ faces $\sigma_t\in(\Delta_\gamma)_2$ and $s$ interior edges $\tau_t$, \eqref{eq:Euler_char_Qd} can be rewritten as
 \begin{equation*}
 \chi\bigl(\quotientSpaceOne_{d,\bullet}\bigr)=\sum_{t=1}^{s}\dim \quotientSpaceOne_{d}(\sigma_{t})-\sum_{t=1}^{s}\dim\quotientSpaceOne_{d}(\tau_{t})+\dim\quotientSpaceOne_{d}(\gamma).
 \end{equation*} 
 By construction, $\dim\quotientSpaceOne_d(\sigma_t)=\binom{d+2}{2}$ for each $t=1,\dots, s$.
 By Lemma \ref{lemma:Qd_tau_two_patches}, if $d\geqslant d_{\gluea}+1\geqslant 2$ then $\dim \quotientSpaceOne_d(\tau)=2d+d_{\gluea}+1$ for each interior edge $\tau$ in $\Delta_{\gamma}$, and therefore $\sum_{t=1}^{s}\quotientSpaceOne_d(\tau_{t})=s(2d+d_{\gluea}+1)$.
 Furthermore, 
 from Lemma \ref{lemma:dim_F_gamma_nonsingular}, 
 if $s\neq 4$ and $d\geqslant 1$, then by \eqref{eqn:dim_Fgamma_symm_nonsingular} we have $\dim\quotientSpaceOne_{d}(\gamma)= \dim\quotientSpaceOne_{(d,d)}(\gamma)=3$. 
 Note that $s\neq 4$ implies that $d_{\gluea}\geqslant 1$. 
 This proves \eqref{eq:chi_s_not_4}.

Assume now $s=4$.
 If $d\geqslant 2$, Lemma \ref{lemma:dim_F_gamma_singular} implies
 $\dim\quotientSpaceOne_{d}(\gamma)=4$. 
 Additionally, for any edge $\tau_t=\sigma_t\cap\sigma_{t+1}\in(\Delta_\gamma)_1^\circ$, we can write the gluing data as in \eqref{eq:symgluedata}, and we get $\gluea(u_{t+1}) = 
 - 2\cos\left(\frac{2\pi}{s_t'}\right) u_{t+1}^2$, where $s_t'$ is the valence of the vertex $\gamma_t'\neq \gamma$ of $\tau_t$ in $\Delta$. 
 
 Then we have two cases, either $s'\neq 4$ and $d_{\gluea}=2$, or $s'=4$ and $\gluea=0$.
 If $d_{\gluea}=2$, then Lemma \ref{lemma:Qd_tau_two_patches} implies $\dim \quotientSpaceOne_d(\tau_t)=2d+3$ for every $d\geqslant 3$. 
In the second case $\gluea=0$, by Corollary \ref{Cor:Qd_tau_two_patches_planar} we know that $\dim \quotientSpaceOne_d(\tau_t)=2d+1$ for every $d\geqslant 0$. These two cases imply \eqref{eq:s4} and \eqref{eq:s4a0}, respectively.
The corresponding results in the bidegree case lead to the formulas for $\chi\bigl(\quotientSpaceOne_{(d,d)},\bullet\bigr)$.
\end{proof}
\begin{example}\label{eg:three_patches_symmetric}
Take $\Delta_\gamma$ as the $2$-dimensional cell complex in Figure \ref{fig:triangle}, which is composed by $s=3$ quadrilaterals $\sigma_i$ sharing a common vertex $\gamma$. 
Let $\Phi$ be the collection of algebraic transition maps obtain from  taking symmetric gluing data along each of the edges $\tau_i=[\gamma,\delta_i]$ of $\Delta_\gamma$. 
In this case, the valence of the vertex $\gamma$ is 3, and we will assume the valence at each vertex $\delta_i$ is 4. 
If we take the gluing data $\gluea$ as the gluing data in \eqref{eq:symgluedata}, we see that $\gluea\neq 0$ and it is a polynomial of degree $d_{\gluea}=2$. 
Then, Theorem \ref{theor:dimG1} implies
\begin{equation*}
\dim G_d^1\bigl(\Delta_\gamma,\Phi\bigr)
=
\frac{3}{2}(d^{2}-d-2),
\end{equation*}
for every $d\geqslant 4$, and in the bidegree case
\begin{equation*}
\dim G_{(d,d)}^1\bigl(\Delta_\gamma,\Phi\bigr)=3d^{2}-3, 
\end{equation*}
for every $d\geqslant 3$.\hfill$\diamond$
\begin{figure}
	\centering
	\includegraphics[scale=0.55]{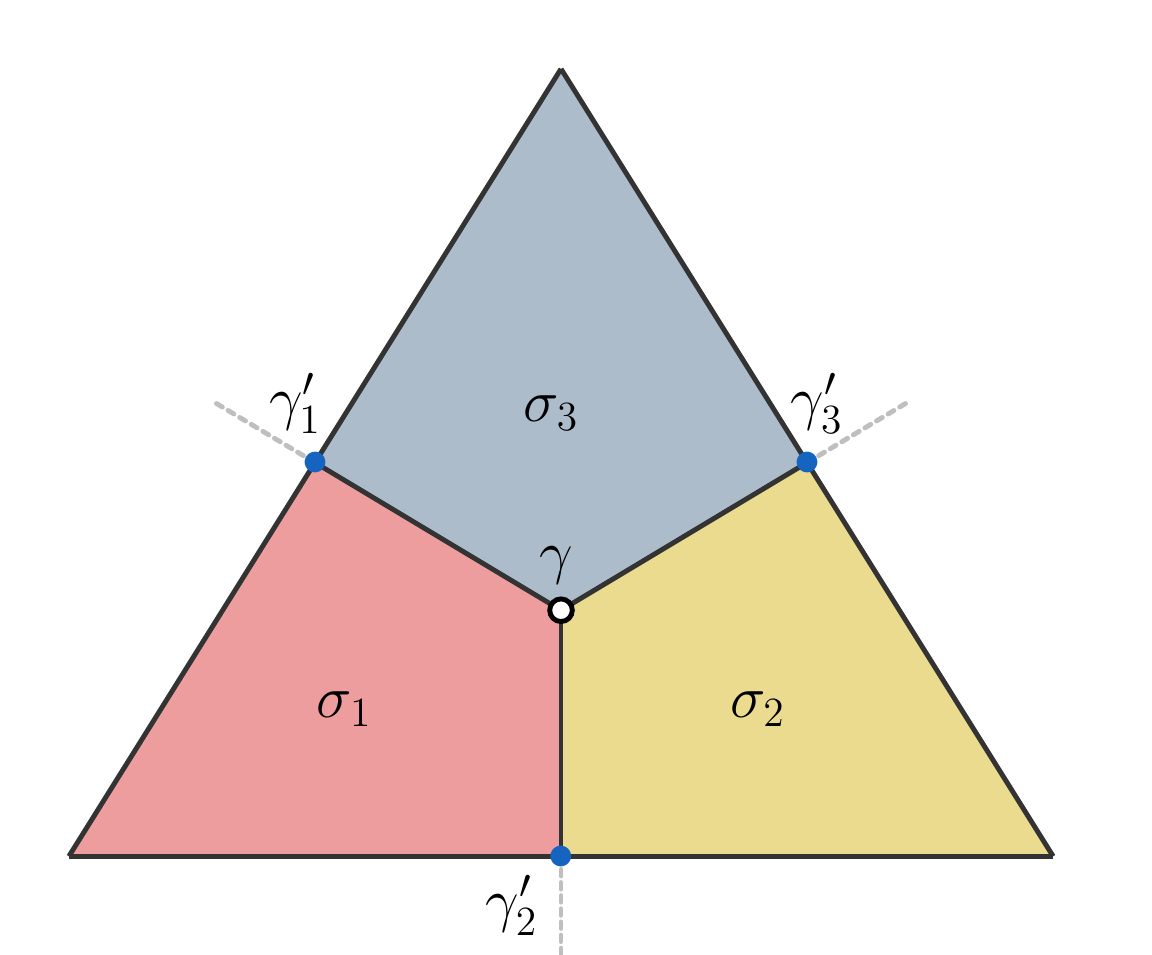}
	\caption{2-dimensional cell complex considered in Example \ref{eg:three_patches_symmetric}, which is the star of the vertex $\gamma$. The valence of $\gamma$ is 3, and we consider valence 4 at each vertex $\delta_i$, for $i=1,2,3$.} \label{fig:triangle}
\end{figure}
\end{example}

\subsection{A general $2$-dimensional $G^{r}$-domain}\label{section:computation_general_Gr_domain}
For a general $2$-dimensional cell complex $\Delta$, we can use \eqref{eq:dim2cellcase} and the formulas for the Euler characteristic equations from Theorem \ref{theor:dimG1} to estimate the dimension of a $G^{r}_{d}(\GrDomain)$ spline space.
Using formula \eqref{eq:dim2cellcase}, we may estimate the dimension of $G^{r}_{d}(\GrDomain)$ by Euler characteristic $\chi(\quotientSpace_{d,\bullet})$. Controlling the homology terms in \eqref{eq:dim2cellcase} is however not direct as shown by the following example where the $G^{r}$-domain cannot be topologically embedded into the plane. We call it a \emph{cube}, because the polyhedral cell complex $\Delta$ we consider is homeomorphic to the boundary of a cube: it has $6$ square faces, $12$ edges and $8$ vertices. The formula \eqref{eq:dim2cellcase} still works in this case. However, it is not necessary that $\dim G^{r}_{d}(\GrDomain)=\chi(\quotientSpace_{d,\bullet})$, even for $r=1$ and $\Phi$ given by symmetric gluing data.

\begin{example}\label{eg:G1_cube}
\begin{figure}
\centering
\includegraphics[height=45mm]{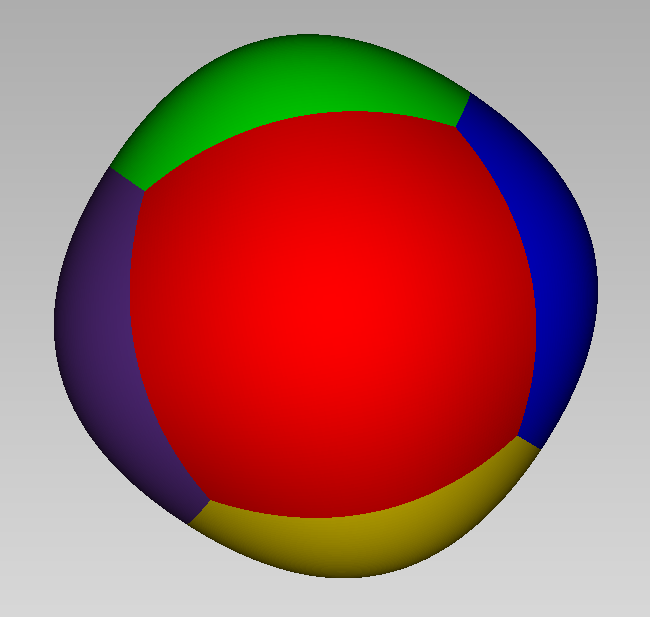}
\caption{\label{fig:cube22}A surface parametrized by degree $(2,2)$ $G^1$-splines over the cube for the gluing data $\gluea(u)=2u-1, \glueb=-1$. See Example \ref{eg:G1_cube}. The surface is interpolating the points $(\pm 1, 0, 0), (0,\pm 1, 0), (0, 0, \pm 1)$ at the center of $6$ faces.}
\end{figure}
We take $\Delta$ as a cell complex homeomorphic to the surface of a cube, so $\Delta$ is composed by six 2-dimensional faces, 12 interior edges, and 8 vertices, each of valence $s=3$. 
Following the notation in \eqref{eq:symgluedata}, the $G^1$-symmetric gluing data at every edge of $\Delta$ is given by
\begin{equation*}
\gluea(u)=-2u+1,~\glueb=-1.
\end{equation*}
In particular, the degree of $\gluea$ is $d_{\gluea}=1$. 
From \eqref{eq:dim2cellcase} we know that $\dim G^{1}_{d}(\GrDomain)=\chi(\quotientSpaceOne_{d,\bullet})+\dim H_{1}(\quotientSpaceOne_{d,\bullet})-\dim H_{0}(\quotientSpaceOne_{d,\bullet})$, and \eqref{eq:Euler_char_Qd} gives
\begin{align*}
	\chi(\quotientSpaceOne_{d,\bullet})=\sum_{\sigma\in\Delta_{2}}\dim \quotientSpaceOne_{d}(\sigma)-\sum_{\tau\in\Delta_{1}^{\circ}}\dim\quotientSpaceOne_{d}(\tau)+\sum_{\gamma\in\Delta_{0}^{\circ}}\dim\quotientSpaceOne_{d}(\gamma).
\end{align*}
By Lemma \ref{lemma:Qd_tau_two_patches}, for each $\tau\in\Delta_{1}^{\circ}$ we have
\begin{equation*}
	\begin{cases}
		\dim\quotientSpaceOne_{d}(\tau)=2d+2,& \mbox{for}~d\geqslant 2,\\
		\dim\quotientSpaceOne_{(d,d)}(\tau)=2d+2,& \mbox{for}~d\geqslant 1, 
	\end{cases}
\end{equation*}
and for each vertex $\gamma\in\Delta_{0}^{\circ}$,Lemma \ref{lemma:dim_F_gamma_nonsingular} implies
\begin{equation*}
\begin{cases}
 \dim\quotientSpaceOne_{d}(\gamma)=3,&\mbox{for}~d\geqslant 3,\\
 \dim\quotientSpaceOne_{(d,d)}(\gamma)=3,&\mbox{for}~d\geqslant 2.
\end{cases}
\end{equation*}
Hence,
\begin{equation*}
\begin{cases}
 \chi(\quotientSpaceOne_{d,\bullet})=3(d^{2}-5d+2),&\mbox{for}~d\geqslant 3,\\
 \chi(\quotientSpaceOne_{(d,d),\bullet})=6(d-1)^{2},&\mbox{for}~d\geqslant 2.
\end{cases}
\end{equation*}
In Table \ref{table:dim}, we put the exact dimension of the space $G^{1}_{d}(\GrDomain)$ which we independently compute using \texttt{Macaulay2} \cite{M2} and the \textsc{Julia} package \texttt{G1Splines.jl}\footnote{\texttt{http://gitlab.inria.fr/AlgebraicGeometricModeling/G1Splines.jl}}. 
The computations in \texttt{Macaulay2} \cite{M2} follow the basis construction described in Algorithm \ref{algorithm:basis_computation_Gr}, while \texttt{G1Splines.jl}  uses the Bernstein representation of splines and the linear relations deduced from sample points on the edges of the partition. 

These computations show that the difference between the exact dimension of the spline space and the Euler characteristic is constant. In the total degree case we have
\begin{equation*}
	\dim G^{1}_{d}(\GrDomain)-\chi(\quotientSpaceOne_{d,\bullet})=12,\quad \mbox{for}~4\leqslant d\leqslant 10, 
\end{equation*}
and in the bidegree case
\begin{equation*}
	\dim G^{1}_{(d,d)}(\GrDomain)=\chi(\quotientSpaceOne_{(d,d),\bullet}),\quad \mbox{for}~2\leqslant d\leqslant 10.
\end{equation*}
This dimension formula in the bidegree case was proved in \cite[Theorem 6.3]{mourrain2016dimension} for any $d\geqslant 1$.
\begin{table}[ht!]
\centering
\renewcommand{\arraystretch}{1.2}
\begin{tabular}{c||c|c|c|c|c|c|c|c|c|c}
 $d$ & 1& 2 & 3 & 4 & 5 & 6 & 7 & 8& 9& 10\\
 \hline
 $\dim G^1_{d}(\GrDomain)$ & 1 & 1 & 1 & 6 & 18 & 36 &60 &90 &126 &168 \\
 \hline 
 $\chi(\quotientSpaceOne_{d,\bullet})$& -6& -12& -12& -6& 6& 24& 48& 78& 114&156\\
 \hline
 $\dim G^1_{(d,d)}(\GrDomain)$ & 1 & 6 & 24 & 54 & 96 & 150 &216 &294 &384 &486
 \\
 \hline 
 $\chi(\quotientSpaceOne_{(d,d),\bullet})$& 0& 6& 24& 54& 96& 150& 216& 294& 384& 486
\end{tabular}
\vspace{0.2cm}
\caption{\label{table:dim} Dimension of the $G^1$-spline spaces from Example \ref{eg:G1_cube}, where the 2-dimensional cell complex $\Delta$ is homeomorphic to the surface of a cube, and the collection $\Phi$ of $G^1$-algebraic transition maps is defined from symmetric gluing data. For a given $d$, the space $G^1_{d}(\GrDomain)$ is the linear space of $G^1$-splines on $\Delta$ of total degree $\leqslant d$, and $G^1_{(d,d)}(\GrDomain)$ is the space of splines of bidegree $\leqslant (d,d)$. For small $d$, the Euler characteristics $\chi(\quotientSpaceOne_{d,\bullet})$ and $\chi(\quotientSpaceOne_{(d,d),\bullet})$ are computed in \texttt{Macaulay2} \cite{M2}.
}\label{table:cube}
\end{table}

We include a parametric surface of degree $(2,2)$ on the cube in Figure \ref{fig:cube22}. 
The coordinate functions are in the spline space $G^1_{(2,2)}(\GrDomain)$ of dimension $6$. They are computed so that the parametric surface interpolates the points $(\pm 1,0,0), ( 0,\pm 1,0), (0,0,\pm 1)$ at the center of the faces of the cube.
\hfill$\diamond$
\end{example}

\section{Concluding remarks}\label{sec:conclusions}
We have made a theoretical and algorithmic contribution to the study of geometrically continuous splines by developing an algebraic framework that extends Billera's algebraic methods for parametric splines. 
This framework has two advantages. 
First, we can use algebraic tools to study geometrically continuous splines. 
In fact, in this paper, we demonstrate this by applying chain complexes and the division algorithm to construct bases and compute dimensions for $G^{r}$-spline spaces. 
Second, $G^{r}$-spline spaces have rich algebraic structures that can be further explored using this framework. Thus, this framework makes it easy for algebraists to explore the exciting research topic of geometrically continuous splines. 

Extending the study from parametric splines to geometrically continuous splines enables us to construct shapes with more complex topology. 
For example, Example \ref{eg:G1_cube} in Section \ref{section:computation_general_Gr_domain} is one such shape that cannot be constructed within the scope of parametric splines. 
However, this flexibility comes at a cost: we need a more complicated algebraic model to describe the spline space. 
While parametric spline spaces have only one coordinate system for the whole domain, $G^{r}$-spline require several coordinate systems, one for each face. As a result, transition maps are needed to define the smoothness condition between different coordinate systems. In Section \ref{sec:Examples}, we have several examples of $G^{r}$-domains having non-trivial transition maps which fit in our framework.
Hence, unlike in the case of parametric splines, the $G^{r}$-spline space $G^{r}(\GrDomain)$ is not always a finitely generated module over a known polynomial ring. In fact, it is not even clear if $G^{r}(\GrDomain)$ is finitely generated as an $\RR$-algebra in general. This will be of interest for future research. 

We have obtained explicit dimension formulas for $G^{r}$-spline spaces over a two-patches configuration and a star of vertex in Section \ref{section:explicit_computations}. To obtain these formulas, we have written the dimension of $G^{r}_{d}(\GrDomain)$ into the sum of Euler characteristic $\chi(\quotientSpace_{d,\bullet})$ and homologies of the chain complex $\quotientSpace_{d,\bullet}$, as defined in Section \ref{section_BSS_deg_d}. Because the transition map does not always preserve degrees between polynomial rings, finding $\chi(\quotientSpace_{d,\bullet})$ and homologies of $\quotientSpace_{d,\bullet}$ is difficult. In fact, the whole Section \ref{section:spline_complex_dim2} is devoted in computing $\chi(\quotientSpace_{d,\bullet})$ for certain kinds of transition maps, where the division algorithm plays an important role in the computation. It may be of interest for future studies to find dimension formulas which work more generally.

The method used for the construction of the $G^r$-basis in Section \ref{section:the_real_engine}, particularly Lemma \ref{lemma:a_basis_of_Fd}, applies to a more general setting and provides a general method to generate certain subspaces of a quotient ring. 
We have implemented Algorithm \ref{algorithm:basis_computation_Gr} in \texttt{Macaulay2} \cite{M2}~ computation in Example \ref{eg:G1_cube}. It is of our interest to develop a software package based on this Algorithm \ref{algorithm:basis_computation_Gr}, which computes a basis for $G^{r}$-spline spaces over any given $G^{r}$-domain.

One issue that remains to be solved is the gluing data in this paper are assumed to be polynomials, while they may be in the form of rational functions as defined in the context of topological surfaces in \cite{mourrain2016dimension}. We are confident that our framework can be extended in the future to cover those cases where transition maps are given by rational gluing data.

\appendix
\section{The chain complex for $G^r$-splines}
\label{section_specializeToBSS}
Let $\Delta$ be a simplicial complex, and $S=\RR[x_{1},\dots,x_{n}]$. 
Following the notation in Section \ref{section:generalized_spline_complex}, we write 
\begin{equation*}
S_{\bullet}=C_{\bullet}(\Delta,\partial\Delta)\otimes_\RR S, 
\end{equation*}
where $C_{\bullet}(\Delta,\partial\Delta)$ 
 is the chain complex of $\Delta$ relative to its boundary $\partial\Delta$. 
For each interior face $\tau\in\Delta_{n-1}^\circ$, we denote by $\ell_{\tau}\in S$ a choice of a degree one polynomial vanishing on $\tau$.
We define $J(\tau)=\langle \ell_{\tau}^{r+1}\rangle\subseteq S$, the ideal generated by $\ell_{\tau}^{r+1}$.
If $\beta\in\Delta_k$, for $k=0,\dots, n-2$, we define 
\begin{equation*}
J(\beta)=\sum_{\tau\supseteq\beta}J(\tau)=\sum_{\tau\supseteq\beta} \langle \ell_{\tau}^{r+1}\rangle. 
\end{equation*}
If $\sigma\in\Delta_n$ we take $J(\sigma)=0$.
The ideals $I(\beta)$ define a sub-complex $I_\bullet$ of $S_{\bullet}$, where the boundary maps $\partial_i$ are the restriction of the differentials of $S_{\bullet}$:
\begin{equation*}
	\label{eqn_idealComplex}
	I_\bullet\colon 0\rightarrow \bigoplus_{\sigma\in\Delta_{n}} J(\sigma)\xrightarrow{\partial_n}\bigoplus_{\tau\in\Delta_{n-1}^\circ}J(\tau) \xrightarrow{\partial_{n-1}} \dots\xrightarrow{\partial_1}\bigoplus_{\gamma\in\Delta^{\circ}_0}J(\gamma)\rightarrow 0.
\end{equation*}
The Billera-Schenck-Stillman spline complex is the quotient $S_{\bullet}/I_{\bullet}$, namely 
\begin{equation*}
	\label{eqn_idealComplex}
	S_\bullet/I_\bullet\colon 0\rightarrow \bigoplus_{\sigma\in\Delta_{n}} S(\sigma)\xrightarrow{\bar\partial_n}\bigoplus_{\tau\in\Delta_{n-1}^\circ}S/J(\tau) \xrightarrow{\bar\partial_{n-1}} \dots\xrightarrow{\bar\partial_1}\bigoplus_{\gamma\in\Delta^{\circ}_0}S/J(\gamma)\rightarrow 0,
\end{equation*}
where $S(\beta)=S$ for each face $\beta\in\Delta$.

Given a $G^{r}$-domain $(\GrDomain)$, for $n$-faces $\sigma_1,\sigma_2\in\Delta_n$ such that $\tau=\sigma_1\cap\sigma_2\in\Delta_{n-1}$, a transition map $\phi_{12}\in\Phi$  $\sigma_1$ to $\sigma_2$ as defined in \eqref{eqn:def_algebraic_transition_maps} is an \emph{identity} if  $u_{1j}\mapsto u_{2j}$ for every $j=1,\dots,n$,
where $X(\sigma_i)$ is the system of coordinates of $\sigma_i$ and  $\mathcal{R}(\sigma_i)$ is the the polynomial ring in the coordinates $(u_{i1},\dots, u_{in})$ of $X(\sigma_i)$.
\begin{proposition}
Let $(\GrDomain)$ be a $G^{r}$-domain, where $\Delta$ is an $n$-dimensional simplicial complex and  all transition maps in $\Phi$ are identities, then the chain complex $(\splineRing_{\bullet},\delta_{\bullet})$ defined by  \eqref{eq:complexF} and \eqref{eq:delta} in Section \ref{section:generalized_spline_complex} is isomorphic to the Billera-Schenck-Stillman spline complex $S_{\bullet}/I_{\bullet}$.
\end{proposition}
\begin{proof}
If for all interior faces $\tau\in\Delta_{n-1}^\circ$ the transition maps in the $(\GrDomain)$ are identities, then the polynomial ring
\begin{equation*}
\ambRing(\sigma)\cong S,
\end{equation*}
for every $\sigma\in\Delta_n$.
In particular, if $\tau=\sigma_{1}\cap\sigma_{2}$ for $\sigma_i\in\Delta_n$, then \eqref{eq:ringalpha} takes the form
\begin{equation*}
\ambRing(\tau)\cong S\otimes_{\RR} S,
\end{equation*}
and \eqref{eq:idealtau} can be rewritten as
\begin{equation*}
\idealComplex(\tau)=\bigl\langle u_{11}-u_{21},\dots,~u_{1n}-u_{2n}\bigr\rangle+I_{\sigma_{1}}(\tau)^{r+1}\cdot\ambRing(\tau)+I_{\sigma_{2}}(\tau)^{r+1}\cdot\ambRing(\tau),
\end{equation*}
where $I_{\sigma_{i}}(\tau)=\bigl\langle \ell_{\tau}(u_{i1},\dots,u_{in})\bigr\rangle\subseteq \mathcal{R}(\sigma_i)$, for $i=1,2$. 
Take $J(\tau)=\bigl\langle \ell_{\tau}(x_{1},\dots,x_{n})^{r+1}\bigr\rangle\subseteq S$. 
It is clear that the homomorphism $\ambRing(\tau)\to S$ given by $
u_{ij}\mapsto x_{j}$, for $i=1,2$ and $j=1,\dots,n$, 
induces an isomorphism $\splineRing(\tau)\cong S/J(\tau)$. 
It is also easy to check that the diagram
\begin{center}
\begin{tikzcd}
 \bigoplus_{\sigma\in\Delta_{n}}\mathcal{F}^{r+1}(\sigma) \arrow[d, "\cong"] \arrow[r, "\delta_{n}"] & \bigoplus_{\tau\in\Delta_{n-1}^{\circ}}\mathcal{F}^{r+1}(\tau) \arrow[d, "\cong"] \\
 \bigoplus_{\sigma\in\Delta_{n}}S \arrow[r, "\partial_{n}"] & \bigoplus_{\tau\in\Delta_{n-1}^{\circ}}S/J(\tau) 
\end{tikzcd}
\end{center}
commutes. 
Similarly, for every $\alpha\in \Delta_{k}^{\circ}$, for $k=0,\dots,n-2$,  the homomorphism 
$\ambRing(\alpha)\to S$, given by
$u_{\sigma j}\mapsto x_{j}$ where $(u_{\sigma 1},\dots,u_{\sigma n})$ are the coordinates of $X(\sigma)$ for every $\sigma\supseteq\alpha$,
induces an isomorphism $\splineRing(\alpha)\cong S/J(\alpha)$.
Hence, also in this case the diagram
\begin{center}
\begin{tikzcd}
 \bigoplus_{\beta\in\Delta_{k+1}^{\circ}}\mathcal{F}^{r+1}(\beta) \arrow[d, "\cong"] \arrow[r, "\delta_{k+1}"] & \bigoplus_{\alpha\in\Delta_{k}^{\circ}}\mathcal{F}^{r+1}(\alpha) \arrow[d, "\cong"] \\
 \bigoplus_{\beta\in\Delta_{k+1}^{\circ}}S \arrow[r, "\partial_{k+1}"] & \bigoplus_{\alpha\in\Delta_{k}^{\circ}}S/J(\alpha) 
\end{tikzcd}
\end{center}
commutes. 
It follows that $\splineRing_{\bullet}\cong S_{\bullet}/I_{\bullet}$ as chain complexes.
\end{proof}
\section*{Acknowledgement}
N.~Villamizar and B.~Yuan were supported by the UK Engineering and Physical Sciences Research Council (EPSRC) New Investigator Award EP/V012835/1. A.~Mantzaflaris and N.~Villamizar wish to thank the support by The Alliance Hubert Curien Programme, project number: 515492678.

\bibliographystyle{plain}
\bibliography{ref.bib}

\end{document}